\documentclass[11pt]{amsart}
\usepackage{amsmath,amsfonts,amsthm,amssymb,url,graphicx}

\DeclareFontFamily{OT1}{rsfs}{}
\DeclareFontShape{OT1}{rsfs}{n}{it}{<-> rsfs10}{}
\DeclareMathAlphabet{\mathscr}{OT1}{rsfs}{n}{it}

\DeclareMathOperator{\supp}{supp}

\DeclareMathOperator{\mo}{\,mod}

\newtheorem{prop}{Proposition}[section]

\newtheorem*{main}{Main Theorem}

\newtheorem{lem}[prop]{Lemma}

\newtheorem*{defn*}{Definition}

\numberwithin{equation}{section}
\title{Minor arcs for Goldbach's problem}
\author{H. A. Helfgott}
\address{Harald Helfgott, 
\'Ecole Normale Sup\'erieure, D\'epartement de Math\'ematiques, 45 rue d'Ulm, F-75230 Paris, France}
\email{harald.helfgott@ens.fr}
\begin{document}
\begin{abstract}
The ternary Goldbach conjecture states that every odd number $n\geq 7$ 
is the sum
of three primes.
The estimation of sums of the form $\sum_{p\leq x} e(\alpha p)$, $\alpha =
a/q + O(1/q^2)$, has been a central part of the main approach to the
conjecture since (Vinogradov, 1937).
Previous work required $q$ or $x$ to be too large to make a 
proof of the conjecture for all $n$ feasible.

The present paper gives new bounds on minor arcs and the tails of major
arcs. This is part of the author's proof of the ternary Goldbach conjecture.


The new bounds are due to several qualitative improvements. In particular,
this paper presents a general method for reducing the cost of Vaughan's
identity, as well as a way to exploit the tails of minor arcs in the context of
the large sieve.
\end{abstract}
\maketitle
\tableofcontents
\section{Introduction}
The ternary Goldbach conjecture (or {\em three-prime conjecture}) 
states that every odd number greater than  $5$ is
the sum of three primes. 
I. M. Vinogradov \cite{Vin} showed in 1937 that every odd integer larger
than a very large constant $C$ is indeed the sum of three primes. His work
was based on the study of exponential sums
\[\sum_{n\leq N} \Lambda(n) e(\alpha n)\]
and their use within the circle method.

Unfortunately, further work has so far reduced $C$ only to 
$e^{3100}$ (\cite{MR1932763}; see also \cite{MR1046491}), which is still much
too large for all odd integers up to $C$ to be checked numerically.
The main problem has been that existing bounds for (\ref{eq:inaug}) in the
{\em minor arc} regime -- namely, $\alpha = a/q + O(1/q^2)$, 
$\gcd(a,q)=1$, $q$ relatively large -- have not been strong enough.

The present paper gives new bounds on smoothed exponential sums
\begin{equation}\label{eq:inaug}
S_\eta(\alpha,x) = \sum_n \Lambda(n) e(\alpha n) \eta(n/x) .
\end{equation}
These bounds are clearly stronger than those on smoothed or unsmoothed
exponential sums in the previous literature, including the bounds of
\cite{Tao}. (See also work by Ramar\'e \cite{MR2607306}.)

In particular, on all arcs around $a/q$, $q>1.5\cdot 10^5$ odd
or $q>3\cdot 10^5$ even, the bounds
are of the strength required for a full solution to the three-prime
conjecture. The same holds on the tails of arcs around $a/q$ for smaller
$q$. 

(The remaining arcs -- namely, those around $a/q$, $q$ small -- are the
{\em major arcs}; they are treated in the companion paper \cite{HelfMaj}.)



The quality of the results here is due to several new ideas of
general applicability. In particular, \S \ref{subs:vaucanc} introduces
a way to obtain cancellation from Vaughan's identity. 
Vaughan's identity is a two-log gambit, in that it introduces two 
convolutions (each of them at a cost of $\log$)
 and offers a great deal of flexibility in compensation. 
One of the ideas in the present paper is that at least one of two $\log$s 
can be successfully recovered after having been given away in the first stage of
the proof. This reduces
the cost of the use of this basic identity in this and, presumably, many
other problems. 

We will also see how to exploit being on the tail of a major arc, whether
in the large sieve (Lemma \ref{lem:ogor}, Prop. \ref{prop:kraken}) or
in other contexts.

There are also several technical improvements that make a qualitative
difference; see the discussions at the beginning of \S \ref{sec:typeI}
and \S \ref{sec:typeII}. Considering smoothed sums 
-- now a common idea -- also helps.
(Smooth sums here go back to
Hardy-Littlewood \cite{MR1555183} --
both in the general context of the circle method and in the context
of Goldbach's ternary problem. In recent work on the problem, they reappear in
\cite{Tao}.)

\subsection{Results}
The main bound we are about to see is essentially proportional to 
$((\log q)/\sqrt{\phi(q)})\cdot x$. The term $\delta_0$ serves to
 improve the bound when we are on the tail of an arc.
\begin{main}
Let $x\geq x_0$, $x_0 = 2.16\cdot 10^{20}$.
Let $S_{\eta}(\alpha,x)$ be as in (\ref{eq:inaug}),
with $\eta$ defined in (\ref{eq:eqeta}).
Let $2 \alpha = a/q + \delta/x$, $q\leq Q$,
$\gcd(a,q)=1$, $|\delta/x|\leq 1/q Q$, where $Q = (3/4) x^{2/3}$.
If $q\leq x^{1/3}/6$, then
\begin{equation}\label{eq:kraw}
\begin{aligned}
&|S_{\eta}(\alpha,x)| \leq 
 \frac{R_{x,\delta_0 q} \log \delta_0 q + 0.5}{\sqrt{\delta_0 \phi(q)}} \cdot x
+ \frac{2.5 x}{\sqrt{\delta_0 q}} + 
 \frac{2x}{\delta_0 q} \cdot L_{x,\delta_0,q}
+ 3.2 x^{5/6},\end{aligned}\end{equation} where
\begin{equation}\label{eq:tosca}\begin{aligned}
\delta_0 &= \max(2,|\delta|/4),\;\;\;\;\;
R_{x,t} = 0.27125 \log 
\left(1 + \frac{\log 4 t}{2 \log \frac{9 x^{1/3}}{2.004 t}}\right)
 + 0.41415 \\
L_{x,\delta,q} &=
\min\left(\frac{\log \delta^{\frac{7}{4}} q^{\frac{13}{4}} + \frac{80}{9}}{\phi(q)/q},
\frac{5}{6} \log x + \frac{50}{9}\right) +
\log q^{\frac{80}{9}} \delta^{\frac{16}{9}} + \frac{111}{5}.
\end{aligned}\end{equation}
If $q > x^{1/3}/6$, then
\[|S_{\eta}(\alpha,x)|\leq 
0.2727 x^{5/6} (\log x)^{3/2} + 1218 x^{2/3} \log x.\]
\end{main}
The factor $R_{x,t}$ is small in practice; for instance, for
$x = 10^{25}$ and $\delta_0 q = 5\cdot 10^5$ 
(typical ``difficult'' values), $R_{x,\delta_0 q}$ equals $0.59648\dotsc$.

The classical choice\footnote{Or, more precisely,
the choice made by Vinogradov and followed by most of the literature since him.
 Hardy and Littlewood
\cite{MR1555183} worked with $\eta(t) = e^{-t}$.} for $\eta$ in (\ref{eq:inaug}) 
is $\eta(t) = 1$ for $t\leq 1$, 
$\eta(t) = 0$ for $t>1$, which, of course, is not smooth, or even
continuous.
We use
\begin{equation}\label{eq:eqeta}
\eta(t) = \eta_2(t) = 4 \max(\log 2 - |\log 2 t|,0),\end{equation}
as in Tao \cite{Tao}, in part for purposes of comparison. (This is
the multiplicative convolution of the characteristic function of an
interval with itself.) Nearly all work
should be applicable to any other sufficiently smooth function $\eta$ 
of fast decay. It is important that $\widehat{\eta}$ decay at least
quadratically.

\subsection{History}
The following notes are here to provide some background; no claim to
completeness is made.

Vinogradov's proof \cite{Vin} was based on his 
novel estimates for exponential sums over primes. Most work on the problem
since then, 
including essentially all work with explicit constants, has been based
on estimates for exponential sums; there are some elegant proofs based on 
cancellation in other kinds of sums 
(\cite{MR834356}, \cite[\S 19]{MR2061214}), but they have
not been made to yield practical estimates.

The earliest explicit result is that of Vinogradov's student Borodzin 
(1939). 
Vaughan \cite{MR0498434} greatly simplified the proof by
introducing what is now called Vaughan's identity. 

The current record is that of Liu and Wang \cite{MR1932763}: the best
previous result was that of \cite{MR1046491}. Other recent work falls into
the following categories.

{\em Conditional results.} The ternary Goldbach conjecture has been proven
under the assumption of the generalized Riemann hypothesis \cite{MR1469323}.

{\em Ineffective results.} An example is the bound given by Buttkewitz
\cite{MR2776653}. The issue is usually a reliance on the Siegel-Walfisz 
theorem. In general, to obtain effective bounds with good constants, 
it is best to avoid 
analytic results on $L$-functions with large conductor. (The present paper
implicitly uses known 
results on the Riemann $\zeta$ function, but uses nothing at all
about other $L$-functions.) 

{\em Results based on Vaughan's identity.} Vaughan's identity
 \cite{MR0498434} greatly 
simplified matters; most textbook treatments are by now based on it.
The minor-arc treatment in \cite{Tao} updates this approach to current
technical standards (smoothing), while taking advantage of its flexibility
(letting variable ranges depend on $q$). 

{\em Results based on log-free identities.}
Using Vaughan's identity implies losing a factor of $(\log x)^2$
(or $(\log q)^2$, as in \cite{Tao}) in the first step. It thus makes sense
to consider other identities that do not involve such a loss.
Most approaches before Vaughan's identity involved larger losses, but already 
\cite[\S 9]{Vin} is relatively economical, at least for very large $x$.
The work of Daboussi \cite{MR1399341} and Daboussi and Rivat \cite{MR1803131}
explores other identities. (A reading of \cite{MR1803131} gave part of
the initial inspiration for the present work.)
Ramar\'e's work \cite{MR2607306} --
asymptotically the best to date -- is based on the Diamond-Steinig
inequality
(for $k$ large). 

\begin{center}
* * *
\end{center}

The author's work on the subject leading to the present
paper was at first based on the (log-free) Bombieri-Selberg identity
($k=3$), but has now been redone with Vaughan's identity in its
foundations. This is feasible thanks to the factor of $\log$ regained in
\S \ref{subs:vaucanc}.

\subsection{Comparison to earlier work}
Table \ref{tab:bloid} compares the bounds for the ratio
$|S_{\eta}(a/q,x)|/x$
given by this paper and by \cite{Tao} for $x= 10^{27}$ and different
values of $q$.
We are comparing worst cases: $\phi(q)$ as small
as possible ($q$ divisible by $2\cdot 3\cdot 5\dotsb$) in the result here,
and $q$ divisible by $4$ (implying $4\alpha \sim a/(q/4)$) in Tao's result.
The main term in the result in this paper improves
slowly with increasing $x$; the results in \cite{Tao} worsen slowly
with increasing $x$.

The qualitative gain with respect to \cite{Tao} is about
$\log(q) \sqrt{\phi(q)/q}$, which is $\sim \log(q)/\sqrt{e^{\gamma}
 (\log \log q)}$
in the worst case. 

The results in \cite{MR1803131} are unfortunately
worse than the trivial bound in this range.
Ramar\'e's results (\cite[Thm. 3]{MR2607306}, \cite[Thm. 6]{Ramlater}) 
are not applicable within the range, since
neither of the conditions $\log q \leq (1/50) (\log x)^{1/3}$,
$q\leq x^{1/48}$ is satisfied. Ramar\'e's bound in
\cite[Thm. 6]{Ramlater} is
\begin{equation}\label{eq:jomor}
\left|\sum_{x<n\leq 2x} \Lambda(n) e(an/q)\right|\leq 13000
\frac{\sqrt{q}}{\phi(q)} x\end{equation}
for $20\leq q\leq x^{1/48}$.
We should underline that, while both the constant $13000$ and the
condition $q\leq x^{1/48}$ keep (\ref{eq:jomor}) from being immediately
useful in the present context, (\ref{eq:jomor}) is asymptotically better
than the results here as $q\to \infty$. (Indeed, qualitatively speaking,
the form of (\ref{eq:jomor}) is the best one can expect from results
derived by the family of methods stemming from \cite{Vin}.) There
is also unpublished work by Ramar\'e (ca. 1993) with better constants for 
$q\ll (\log x/\log \log x)^4$.
\begin{table}
\begin{center} \begin{tabular}{l|l|l}
$q_0$ &$\frac{|S_{\eta}(a/q,x)|}{x}$, HH& 
$\frac{|S_{\eta}(a/q,x)|}{x}$, Tao\\ \hline
$10^5$ & $0.04522$ & $0.34475$\\
$1.5\cdot 10^5$ & $0.03821$ & $0.28836$\\
$2.5\cdot 10^5$ & $0.03097$ & $0.23194$\\
$5\cdot 10^5$ & $0.02336$ & $0.17416$\\
$7.5\cdot 10^5$ & $0.01984$ & $0.14775$\\
$10^6$ & $0.01767$ & $0.13159$\\
$10^7$ & $0.00716$ & $0.05251$
\end{tabular} 
\caption{Worst-case upper bounds on $x^{-1} |S_{\eta}(a/2q,x)|$ for
$q\geq q_0$,
$|\delta|\leq 8$, $x= 10^{27}$. The trivial bound is $1$.}\label{tab:bloid}
\end{center}\end{table}

\subsection{Acknowledgments} The author is very thankful to O. Ramar\'e for 
his crucial help and feedback, and to D. Platt for his prompt and helpful
answers. He is also much 
indebted to A. Booker, B. Green, 
H. Kadiri, T. Tao and M. Watkins for many discussions on Goldbach's problem and 
related issues. Thanks are also due to 
B. Bukh, A. Granville and P. Sarnak for their valuable advice.

Travel and other expenses 
were funded in part by the Adams Prize and the Philip Leverhulme Prize.
The author's work on the problem started at the
Universit\'e de Montr\'eal (CRM) in 2006; he is grateful to both the Universit\'e
de Montr\'eal and the \'Ecole Normale Sup\'erieure for providing pleasant
working environments.

The present work would most likely not have been possible without free and 
publicly available
software: Maxima, PARI, Gnuplot, QEPCAD, SAGE, and, of course, \LaTeX, Emacs,
the gcc compiler and GNU/Linux in general.

\section{Preliminaries}
\subsection{Notation}
Given positive integers $m$, $n$, we say $m|n^\infty$ if every prime
dividing
$m$ also divides $n$. We say a positive integer $n$ is {\em square-full} 
if, for every prime $p$ dividing $n$, the square $p^2$ also divides $n$.
(In particular, $1$ is square-full.) We say $n$ is {\em square-free} if
$p^2\nmid n$ for every prime $p$. For $p$ prime, $n$ a non-zero integer,
we define $v_p(n)$ to be the largest non-negative integer $\alpha$ such
that $p^\alpha|n$.

When we write $\sum_n$, we mean $\sum_{n=1}^{\infty}$, unless the contrary
is stated. 
As usual, $\mu$, $\Lambda$, $\tau$ and $\sigma$ denote the Moebius function,
the von Mangoldt function, the divisor function and the sum-of-divisors 
function, respectively.

As is customary, we write $e(x)$ for $e^{2\pi i x}$.
We write $|f|_r$ for the $L_r$ norm of a function $f$.

We write $O^*(R)$ to mean a quantity at most $R$ in absolute value.

\subsection{Fourier transforms and exponential sums}

The Fourier transform on $\mathbb{R}$ is normalized here as follows:
\[\widehat{f}(t) = \int_{-\infty}^\infty e(-xt) f(x) dx.
\]

If $f$ is compactly supported (or of fast decay) and piecewise continuous,
$\widehat{f}(t) = \widehat{f'}(t)/(2\pi i t)$ by integration by parts.
Iterating, we obtain that, if $f$ is compactly supported,
continuous and piecewise $C^1$, then 
\begin{equation}\label{eq:malcros}
\widehat{f}(t) = 
O^*\left(\frac{|\widehat{f''}|_\infty}{(2\pi t)^2}\right) =
O^*\left(\frac{|f''|_1}{(2\pi t)^2}\right) 
,\end{equation}
and so $\widehat{f}$ decays at least quadratically.

The following bound is standard (see, e.g., \cite[Lemma 3.1]{Tao}):
for $\alpha\in \mathbb{R}/\mathbb{Z}$ and 
$f:\mathbb{R}\to \mathbb{C}$ compactly supported and piecewise continuous,
\begin{equation}\label{eq:ra}
\left|\sum_{n\in \mathbb{Z}} f(n) e(\alpha n)\right| \leq 
\min\left(|f|_1 + \frac{1}{2} |f'|_1,\frac{\frac{1}{2} |f'|_1}{|\sin(\pi
\alpha)|}\right).\end{equation} (The first bound follows from
$\sum_{n\in \mathbb{Z}} |f(n)| \leq |f|_1 + (1/2) |f'|_1$, which, in turn
is a quick consequence of the fundamental theorem of calculus; the second
bound is proven by summation by parts.)
The alternative bound $(1/4) |f''|_1/|\sin(\pi \alpha)|^2$ given
in \cite[Lemma 3.1]{Tao} (for $f$ continuous and piecewise $C^1$) can usually be 
improved by the following estimate.

\begin{lem}\label{lem:areval}
Let $f:\mathbb{R}\to \mathbb{C}$ be compactly supported, continuous and
piecewise $C^1$. Then
\begin{equation}\label{eq:quisquil}
\left|\sum_{n\in \mathbb{Z}} f(n) e(\alpha n)\right| \leq 
\frac{\frac{1}{4} |\widehat{f''}|_{\infty}}{(\sin \alpha \pi)^2} 
\end{equation}
for every $\alpha \in \mathbb{R}$.
\end{lem}
As usual, the assumption of compact support could be easily relaxed to 
an assumption of fast decay.
\begin{proof}
By the Poisson summation formula,
\[\sum_{n=-\infty}^\infty f(n) e(\alpha n) = \sum_{n=-\infty}^\infty \widehat{f}(
n-\alpha).\]
Since $\widehat{f}(t) = \widehat{f'}(t)/(2\pi i t)$,
\[
\sum_{n=-\infty}^\infty \widehat{f}(n-\alpha) = 
 \sum_{n=-\infty}^\infty \frac{\widehat{f'}(n-\alpha)}{2\pi i (n-\alpha)} = 
 \sum_{n=-\infty}^\infty \frac{\widehat{f''}(n-\alpha)}{(2\pi i (n-\alpha))^2} .
\]
By Euler's formula $\pi \cot s \pi = 1/s + \sum_{n=1}^{\infty} (1/(n+s) - 
1/(n-s))$,
\begin{equation}\label{eq:euler}
\sum_{n=-\infty}^{\infty} \frac{1}{(n+s)^2} = - (\pi \cot s \pi)' = 
\frac{\pi^2}{(\sin s \pi)^2} .\end{equation}
Hence
\[\left|\sum_{n=-\infty}^\infty \widehat{f}(n-\alpha)\right| \leq
|\widehat{f''}|_\infty \sum_{n=-\infty}^\infty \frac{1}{(2\pi (n-\alpha))^2}
= |\widehat{f''}|_\infty \cdot \frac{1}{(2\pi)^2} \cdot \frac{\pi^2}{(\sin
\alpha \pi)^2} .\]
\end{proof}
The trivial bound $|\widehat{f''}|_\infty \leq |f''|_1$, applied to
(\ref{eq:quisquil}), recovers the bound in \cite[Lemma 3.1]{Tao}. In
order to do better, we will give a tighter bound for $|\widehat{f''}|_\infty$
when $f = \eta_2$ in Appendix \ref{sec:norms}.

Integrals of multiples of $f''$ (in particular,
 $|f''|_1$ and $\widehat{f''}$) can still be made sense of 
when $f''$ is undefined at a finite number
of points, provided $f$ is understood as a distribution (and $f'$ has
finite total variation). This is the case, in particular, for $f=\eta_2$.

\begin{center}
* * *
\end{center}

When we need to estimate $\sum_n f(n)$ precisely, we will use the Poisson 
summation formula: \[\sum_n f(n) = \sum_n \widehat{f}(n).\]
We will not have to worry about convergence here, since we will apply the Poisson summation formula only to compactly supported functions $f$ whose
Fourier transforms decay at least quadratically.

\subsection{Smoothing functions}

For the smoothing function $\eta_2$ in
(\ref{eq:eqeta}), 
\begin{equation}\label{eq:muggle}
|\eta_2|_1 = 1,\;\;\;\;\; |\eta_2'|_1 = 8 \log 2,\;\;\;\;\; |\eta_2''|_1 = 48,
\end{equation}
as per \cite[(5.9)--(5.13)]{Tao}.
Similarly, for $\eta_{2,\rho}(t) = \log(\rho t) \eta_2(t)$, where $\rho\geq 4$,
\begin{equation}\label{eq:cloclo}\begin{aligned}
|\eta_{2,\rho}|_1 &< \log(\rho) |\eta_2|_1 = \log(\rho)\\
|\eta_{2,\rho}'|_1 &= 2 \eta_{2,\rho}(1/2) = 2 \log(\rho/2) \eta_2(1/2) < 
(8 \log 2) \log \rho,\\
|\eta_{2,\rho}''|_1 &= 4 \log(\rho/4) + |2 \log \rho - 4 \log(\rho/4)| +
|4\log 2 - 4 \log \rho| \\ &+ |\log \rho - 4 \log 2| +
|\log \rho| < 48 \log \rho.
\end{aligned}\end{equation}
(In the first inequality, we are using the fact that $\log(\rho t)$
is always positive (and less than $\log(\rho)$) when $t$ is in the support
of $\eta_2$.)

Write $\log^+ x$ for $\max(\log x,0)$.

\subsection{Bounds on sums of $\mu(m)$ and $\Lambda(n)$}
We will need explicit bounds on $\sum_{n\leq N} \mu(n)/n$ and related sums involving $\mu$.
The situation here is less well-developed than for sums involving $\Lambda$. 
 The main reason is that the complex-analytic approach to
estimating $\sum_{n\leq N} \mu(n)$ would involve $1/\zeta(s)$ rather than 
$\zeta'(s)/\zeta(s)$, and thus strong explicit bounds on the residues of $1/\zeta(s)$
would be needed.

Fortunately all we need is a saving of $(\log n)$ or $(\log n)^2$ 
on the trivial bound.
This is provided by the following.
\begin{enumerate}
\item (Granville-Ramar\'e \cite{MR1401709}, Lemma 10.2)
\begin{equation}\label{eq:grara}
\left|\sum_{n\leq x: \gcd(n,q)=1} \frac{\mu(n)}{n}\right|\leq 1\end{equation}
for all $x$, $q\geq 1$,
\item (Ramar\'e \cite{Fromexpliest}; cf. El Marraki \cite{MR1378588}, \cite{ElMarraki})
\begin{equation}\label{eq:marraki}
\left|\sum_{n\leq x} \frac{\mu(n)}{n}\right|\leq \frac{0.03}{\log x}
\end{equation}
for $x\geq 11815$. 
\item (Ramar\'e \cite{Ramsev}) 
\begin{equation}\label{eq:ronsard}
\sum_{n\leq x: \gcd(n,q)=1}
\frac{\mu(n)}{n} = 
O^*\left(\frac{1}{\log x/q} \cdot \frac{4}{5} \frac{q}{\phi(q)}\right)
\end{equation} for all $x$ and all $q\leq x$;
\begin{equation}\label{eq:meproz}
\sum_{n\leq x: \gcd(n,q)=1}
\frac{\mu(n)}{n} \log \frac{x}{n} = O^*\left(1.00303 \frac{q}{\phi(q)}\right) 
\end{equation}
for all $x$ and all $q$.
\end{enumerate}
Improvements on these bounds would lead to improvements on type I estimates, but not in what are the worst terms overall at this point.

A computation carried out by the author has proven 
the following inequality for all real $x\leq 10^{12}$:
\begin{equation}\label{eq:ramare}
\left|\sum_{n\leq x} \frac{\mu(n)}{n}\right|\leq \sqrt{\frac{2}{x}}
\end{equation}
The computation was rigorous, in that it used D. Platt's implementation
\cite{Platt} of double-precision interval arithmetic based on Lambov's 
\cite{Lamb} ideas. For the sake of verification, we record that
\[5.42625\cdot 10^{-8}
\leq \sum_{n\leq 10^{12}} \frac{\mu(n)}{n}\leq
5.42898\cdot 10^{-8}.
\]

Computations also show that the stronger bound 
\[\left|\sum_{n\leq x} \frac{\mu(n)}{n}\right|\leq \frac{1}{2\sqrt{x}}\]
holds for all $3\leq x\leq 7727068587$, but not for $x=7727068588-\epsilon$.

Earlier, numerical work carried out by Olivier Ramar\'e 
\cite{Raexpliest}
had shown 
that (\ref{eq:ramare}) holds for all $x\leq 10^{10}$.

We will make reference to various bounds on $\Lambda(n)$ in the literature. The following bound can be easily derived from \cite[(3.23)]{MR0137689},
supplemented by a quick calculation of the contribution of powers of primes
$p<32$:
\begin{equation}\label{eq:rala}
\sum_{n\leq x} \frac{\Lambda(n)}{n} \leq \log x.
\end{equation}
We can derive a bound in the other direction from \cite[(3.21)]{MR0137689}
(for $x>1000$, adding the contribution of all prime powers $\leq 1000$)
and a numerical verification for $x\leq 1000$:
\begin{equation}\label{eq:ralobio}
\sum_{n\leq x} \frac{\Lambda(n)}{n} \geq \log x - \log \frac{3}{\sqrt{2}} .
\end{equation}

We also use the following older bounds:
\begin{enumerate}
\item By the second table in \cite[p. 423]{MR1320898}, supplemented by a computation for $2\cdot 10^6\leq V\leq 4\cdot 10^6$,
\begin{equation}\label{eq:trado1}\sum_{n\leq y} \Lambda(n)\leq 1.0004 y
\end{equation}
for $y\geq 2\cdot 10^6$.
\item \begin{equation}\label{eq:trado2}\sum_{n\leq y} \Lambda(n) < 1.03883 y\end{equation}
 for every $y>0$ \cite[Thm. 12]{MR0137689}.
\end{enumerate}

For all $y>663$,
\begin{equation}\label{eq:chronop}
\sum_{n\leq y} \Lambda(n) n < 1.03884 \frac{y^2}{2} ,\end{equation}
where we use (\ref{eq:trado2}) and partial summation for $y>200000$, and
a computation for $663 < y\leq 200000$. Using instead the second table in 
\cite[p. 423]{MR1320898}, together with computations for small $y<10^7$ and
partial summation, we get that
\begin{equation}\label{eq:nicro}
\sum_{n\leq y} \Lambda(n) n < 1.0008 \frac{y^2}{2} \end{equation}
for $y>1.6\cdot 10^6$.

Similarly,
\begin{equation}\label{eq:charol}
\sum_{n\leq y} \Lambda(n) < 2\cdot 1.0004 \sqrt{y}\end{equation}
for all $y\geq 1$.

It is also true that
\begin{equation}\label{eq:kast}
\sum_{y/2<p\leq y} (\log p)^2 \leq \frac{1}{2} y (\log y)
\end{equation}
for $y\geq 117$: this holds for $y\geq 2\cdot 758699$ by \cite[Cor. 2]{MR0457373}
(applied to $x = y$, $x = y/2$ and $x=2 y/3$) and for $117\leq y< 2\cdot 758699$
by direct computation.

\subsection{Basic setup}

We begin by applying Vaughan's identity \cite{MR0498434}: 
for any function 
$\eta:\mathbb{R}\to \mathbb{R}$, any completely multiplicative function
$f:\mathbb{Z}^+\to \mathbb{C}$ and any $x>0$, $U,V\geq 0$,
\begin{equation}\label{eq:bob}
\sum_n \Lambda(n) f(n) e(\alpha n) \eta(n/x) = S_{I,1} - S_{I,2} + S_{II} + 
S_{0,\infty},
\end{equation}
where
\begin{equation}\label{eq:nielsen}\begin{aligned}
S_{I,1} &= \sum_{m\leq U} \mu(m) f(m)\sum_n (\log n) e(\alpha m n) f(n) \eta(m n/x),\\
S_{I,2} &= \sum_{d\leq V} \Lambda(d) f(d) \sum_{m\leq U} \mu(m) f(m) \sum_n e(\alpha d m n) f(n)
 \eta(d m n/x),\\
S_{II} &= \sum_{m>U} f(m) \left(\mathop{\sum_{d>U}}_{d|m} \mu(d)\right) \sum_{n>V}
\Lambda(n) e(\alpha m n) f(n) \eta(m n/x),\\
S_{0,\infty} &= \sum_{n\leq V} \Lambda(n) e(\alpha n) f(n) \eta(n/x) .
\end{aligned}\end{equation}
The proof is essentially an application of the M\"obius inversion formula;
see, e.g., \cite[\S 13.4]{MR2061214}. In practice, we will use the function
\begin{equation}\label{eq:joroy}
f(n) = \begin{cases} 1 &\text{if $\gcd(n,v)=1$,}\\
0 &\text{otherwise,}\end{cases}\end{equation}
where $v$ is a small, positive, square-free integer.
(Our final choice will be $v=2$.)
 Then
\begin{equation}\label{eq:sofot}
S_\eta(x,\alpha) = S_{I,1} - S_{I,2} + S_{II} + S_{0,\infty} + S_{0,w},\end{equation}
where $S_{\eta}(x,\alpha)$ is as in (\ref{eq:inaug}) and
\[S_{0,v} = \sum_{n|v} \Lambda(n) e(\alpha n) \eta(n/x).\]

The sums $S_{I,1}$, $S_{I,2}$ are called ``of type I'' (or linear),
the sum $S_{II}$ is called ``of type II'' (or bilinear).
 The sum $S_0$ is in general negligible; for our later choice of $V$ and
$\eta$, it will be in fact $0$. The sum $S_{0,v}$ will be negligible as well.

\section{Type I}\label{sec:typeI}
There are here three main improvements
in comparison to standard treatments: 
\begin{enumerate}
\item\label{it:sept} The terms with $m$ divisible by $q$ get taken out and treated separately by analytic means. This all but eliminates what would otherwise be the main term.
\item For large $m$, the other terms get handled by improved estimates on
trigonometric sums.
\item 
The ``error'' term $\delta/x = \alpha - a/q$ is used to our advantage. This happens both
through the Poisson summation formula and through the use of two successive
approximations.
\end{enumerate}

\subsection{Trigonometric sums}

The following lemmas on trigonometric sums improve on the best
Vinogradov-type lemmas in the literature. (By this, we mean results
 of the type of 
Lemma 8a and Lemma 8b in \cite[Ch. I]{MR2104806}. See, in particular,
the work of Daboussi and Rivat \cite[Lemma 1]{MR1803131}.) The main
idea is to switch between different types of approximation within the sum,
rather than just choosing between bounding all terms either trivially 
(by $A$) or non-trivially (by $C/|\sin(\pi \alpha n)|^2$). There will
also\footnote{This is a change with respect to the first version of
this paper's preprint \cite{Helf}. The version of Lemma \ref{lem:gotog} 
there has, however, the advantage
of being immediately comparable to results in the literature.}
be improvements in our applications stemming from the fact that Lemmas
\ref{lem:gotog} and Lemma \ref{lem:couscous}  take quadratic 
($|\sin(\pi \alpha n)|^2$) rather than linear ($|\sin(\pi \alpha n)|$)
inputs.

\begin{lem}\label{lem:gotog}
Let $\alpha = a/q + \beta/q Q$, $\gcd(a,q)=1$, $|\beta|\leq 1$, $q\leq Q$.
Then, for any $A, C\geq 0$,
\begin{equation}\label{eq:betblu}
\sum_{y<n\leq y+q} \min\left(A, 
\frac{C}{|\sin (\pi \alpha n)|^2}\right)\leq 
\min\left(2 A + \frac{6 q^2}{\pi^2} C,
3 A + \frac{4q}{\pi} \sqrt{A C}\right)
.\end{equation}
\end{lem}
\begin{proof}
We start by letting
$m_0 = \lfloor y \rfloor + \lfloor (q+1)/2\rfloor$,
$j = n-m_0$, so that $j$ ranges in the interval $(-q/2,q/2\rbrack$.
We write
\[\alpha n = \frac{aj + c}{q} + \delta_1(j) + \delta_2 \mo 1,\]
where $|\delta_1(j)|$ and $|\delta_2|$ are both $\leq 1/2q$; we can assume
$\delta_2\geq 0$. The variable $r = aj+c \mo q$ occupies each residue class
$\mo p$ 
exactly once. 

One option is to bound the terms corresponding to $r=0, -1$ by $A$ each
and all the other terms by $C/|\sin(\pi \alpha n)|^2$. The terms 
corresponding to $r=-k$ and $r=k-1$ ($2\leq k\leq q/2$) contribute at most
\[\frac{1}{\sin^2 \frac{\pi}{q} (k - \frac{1}{2} - q \delta_2)}
+  \frac{1}{\sin^2 \frac{\pi}{q} (k - \frac{3}{2} + q \delta_2)}
\leq \frac{1}{\sin^2 \frac{\pi}{q} \left(k-\frac{1}{2}\right)} +
\frac{1}{\sin^2 \frac{\pi}{q} \left(k-\frac{3}{2}\right)},\]
since $x\mapsto \frac{1}{(\sin x)^2}$ is convex-up on $(0,\infty)$.
Hence the terms with $r\ne 0, 1$ contribute at most
\[\frac{1}{\left(\sin \frac{\pi}{2q}\right)^2} + 
2 \sum_{2\leq r \leq \frac{q}{2}} \frac{1}{\left(\sin \frac{\pi}{q} (r-1/2)
\right)^2} \leq \frac{1}{\left(\sin \frac{\pi}{2q}\right)^2} + 
2 \int_1^{q/2} \frac{1}{\left(\sin \frac{\pi}{q} x \right)^2} ,\]
where we use again the convexity of $x\mapsto 1/(\sin x)^2$. (We can 
assume $q>2$, as otherwise we have no terms other than $r=0,1$.)
Now
\[\int_1^{q/2} \frac{1}{\left(\sin \frac{\pi}{q} x\right)^2} dx = 
\frac{q}{\pi} \int_{\frac{\pi}{q}}^{\frac{\pi}{2}} \frac{1}{(\sin u)^2} du =
\frac{q}{\pi} \cot \frac{\pi}{q} .\] Hence
\[\sum_{y<n\leq y+q} \min\left(A, \frac{C}{(\sin \pi \alpha n)^2}\right)
\leq 2 A + \frac{C}{\left(\sin \frac{\pi}{2q}\right)^2} + 
C \cdot \frac{2 q}{\pi} \cot \frac{\pi}{q} .\]
Now, by \cite[(4.3.68)]{MR0167642} and \cite[(4.3.70)]{MR0167642},
for $t\in (-\pi,\pi)$,
\begin{equation}\label{eq:pordo1}\begin{aligned}
\frac{t}{\sin t} &= 1 + \sum_{k\geq 0} a_{2k+1} t^{2k+2} = 
1 + \frac{t^2}{6} + \dotsc\\
t \cot t &= 1 - \sum_{k\geq 0} b_{2k+1} t^{2k+2} = 1 - \frac{t^2}{3}
- \frac{t^4}{45} - \dotsc,
\end{aligned}\end{equation}
where $a_{2k+1}\geq 0$, $b_{2k+1}\geq 0$. Thus, for $t\in \lbrack 0,t_0\rbrack$,
$t_0<\pi$,
\begin{equation}\label{eq:pordo2}
\left(\frac{t}{\sin t}\right)^2 = 1 + \frac{t^2}{3} + c_0(t) t^4
\leq 1 + \frac{t^2}{3} + 
c_0(t_0) t^4,\end{equation} where
\[c_0(t) = \frac{1}{t^4} \left(\left(\frac{t}{\sin t}\right)^2 - 
\left(1 + \frac{t^2}{3}\right)\right),\]
which is an increasing function because $a_{2k+1}\geq 0$.
For $t_0=\pi/4$, $c_0(t_0) \leq 0.074807$. Hence,
\[\begin{aligned}\frac{t^2}{\sin^2 t} + t \cot 2 t &\leq
\left(1 + \frac{t^2}{3} + c_0\left(\frac{\pi}{4}\right)
 t^4\right) + \left(\frac{1}{2} - 
\frac{2 t^2}{3} - \frac{8 t^4}{45}\right)\\ &= \frac{3}{2} - \frac{t^2}{3} 
+\left(c_0\left(\frac{\pi}{4}\right) - \frac{8}{45}\right) t^4 \leq \frac{3}{2} - \frac{t^2}{3}
\leq \frac{3}{2}\end{aligned}\]
for $t\in \lbrack 0, \pi/4\rbrack$.

Therefore, the left
side of (\ref{eq:betblu}) is at most
\[2 A + C \cdot \left(\frac{2 q}{\pi}\right)^2 \cdot \frac{3}{2} = 2 A + \frac{6}{\pi^2}
C q^2 .\]
 
The following is an alternative approach yielding the other estimate in
(\ref{eq:betblu}).
We bound the terms corresponding to $r=0$, $r=-1$, $r=1$ 
by $A$ each. We let $r = \pm r'$ for $r'$ ranging from $2$ to $q/2$.
We obtain that the sum is at most
\begin{equation}\label{eq:cloison}
\begin{aligned}
3 A &+ \sum_{2\leq r'\leq q/2} 
 \min\left(A,\frac{C}{\left(\sin \frac{\pi}{q} 
\left(r' - \frac{1}{2} - q\delta_2 \right)\right)^2}\right) \\
&+ \sum_{2\leq r'\leq q/2} 
 \min\left(A,\frac{C}{\left(\sin \frac{\pi}{q} 
\left(r' - \frac{1}{2} + q\delta_2 \right)\right)^2}\right).\end{aligned}
\end{equation}

We bound a term $\min(A,C/\sin((\pi/q) (r'-1/2\pm q\delta_2))^2)$ by $A$ if
and only if $C/\sin((\pi/q) (r'-1\pm q \delta_2))^2 \geq A$. The number
of such terms is \[\leq \max(0,\lfloor
(q/\pi) \arcsin(\sqrt{C/A}) \mp q\delta_2\rfloor),\] and thus at most
$(2 q/\pi) \arcsin(\sqrt{C/A})$ in total. (Recall that $q\delta_2\leq 1/2$.)
Each other term gets bounded
by the integral of $C/\sin^2(\pi\alpha/q)$ from $r'-1\pm q\delta_2$
($\geq (q/\pi) \arcsin(\sqrt{C/A})$) to $r'\pm q\delta_2$, by convexity.
Thus (\ref{eq:cloison}) is at most
\[\begin{aligned}
3 A &+ \frac{2q}{\pi} A \arcsin \sqrt{\frac{C}{A}} +
2 \int_{\frac{q}{\pi}  \arcsin \sqrt{\frac{C}{A}}}^{q/2}\; \frac{C}{\sin^2 \frac{\pi t}{q}}
dt\\
&\leq 3 A + \frac{2q}{\pi} A \arcsin \sqrt{\frac{C}{A}} +
\frac{2 q}{\pi} C \sqrt{\frac{A}{C}-1} \end{aligned}\]

We can easily show (taking derivatives) that
$\arcsin x +
x (1-x^2) \leq 2 x$ for $0\leq x\leq 1$. Setting $x = C/A$, we see
that this implies that
\[3 A + \frac{2q}{\pi} A \arcsin \sqrt{\frac{C}{A}} +
\frac{2 q}{\pi} C \sqrt{\frac{A}{C}-1} \leq
 3 A + \frac{4q}{\pi} \sqrt{A C}.\]
(If $C/A>1$, then $3 A + (4q/\pi) \sqrt{A C}$ is greater than
$A q$, which is an obvious upper bound for the left side of (\ref{eq:betblu}).)
\end{proof}

\begin{lem}\label{lem:couscous}
Let $\alpha = a/q + \beta/q Q$, $\gcd(a,q)=1$, $|\beta|\leq 1$, $q\leq Q$.
Let $y_2>y_1\geq 0$. If $y_2-y_1\leq q$ and $y_2\leq Q/2$, then, for any
$A, C \geq 0$,
\begin{equation}\label{eq:dijkre}
\mathop{\sum_{y_1 < n \leq y_2}}_{q\nmid n} 
  \min\left(A,
\frac{C}{|\sin (\pi \alpha n)|^2}\right) \leq 
\min\left(\frac{20}{3 \pi^2} C q^2, 2 A + \frac{4 q}{\pi} \sqrt{A C}\right).
\end{equation}
\end{lem}
\begin{proof}
Clearly, $\alpha n$ equals $an/q + (n/Q) \beta/q$; since
$y_2\leq Q/2$, this means that $|\alpha n - an/q| \leq 1/2q$
for $n\leq y_2$; moreover, again for $n\leq y_2$, the sign of
$\alpha n - an/q$ remains constant. Hence the left side of
(\ref{eq:dijkre}) is at most
\[\begin{aligned}
\sum_{r=1}^{q/2} \min\left(A, \frac{C}{(\sin \frac{\pi}{q}
    (r-1/2))^2}\right)
+ \sum_{r=1}^{q/2} \min\left(A, \frac{C}{(\sin \frac{\pi}{q} r)^2}\right)
.
\end{aligned}\]
Proceeding as in the
proof of Lemma \ref{lem:gotog}, we obtain a bound of at most
\[
C \left(\frac{1}{(\sin \frac{\pi}{2q})^2}  + \frac{1}{(\sin \frac{\pi}{q})^2} +
\frac{q}{\pi} \cot \frac{\pi}{q} + \frac{q}{\pi} \cot \frac{3 \pi}{2 q}\right)\]
for $q\geq 2$. (If $q=1$, then the left-side of (\ref{eq:dijkre}) is
trivially zero.) Now, by (\ref{eq:pordo1}),
\[\begin{aligned}
\frac{t^2}{(\sin t)^2} + \frac{t}{2} \cot 2t &\leq
\left(1 + \frac{t^2}{3} + c_0\left(\frac{\pi}{4}\right) t^4\right) + 
\frac{1}{4} \left(1 - \frac{4 t^2}{3} - \frac{16 t^4}{45}\right)\\
&\leq \frac{5}{4} + \left(c_0\left(\frac{\pi}{4}\right) - \frac{4}{45}\right)
t^4  \leq \frac{5}{4}
\end{aligned}\]
for $t\in \lbrack 0,\pi/4\rbrack$, and
\[\begin{aligned}
\frac{t^2}{(\sin t)^2} + t \cot \frac{3 t}{2} &\leq
\left(1 + \frac{t^2}{3} + c_0\left(\frac{\pi}{2}\right) t^4\right) + 
\frac{2}{3} \left(1 - \frac{3 t^2}{4} - \frac{81 t^4}{2^4\cdot 45}\right)\\
&\leq \frac{5}{3} + \left(- \frac{1}{6} + \left( c_0\left(\frac{\pi}{2}\right)
- \frac{27}{360}\right) \left(\frac{\pi}{2}\right)^2\right) t^2 \leq \frac{5}{3}
\end{aligned}\]
for $t\in \lbrack 0,\pi/2\rbrack$. Hence,
\[
\left(\frac{1}{(\sin \frac{\pi}{2q})^2}  + \frac{1}{(\sin \frac{\pi}{q})^2} +
\frac{q}{\pi} \cot \frac{\pi}{q} + \frac{q}{\pi} \cot \frac{3 \pi}{2 q}\right)
\leq \left(\frac{2 q}{\pi}\right)^2 \cdot \frac{5}{4} + \left(\frac{q}{\pi}
\right)^2\cdot \frac{5}{3} \leq
\frac{20}{3 \pi^2} q^2.\]
Alternatively, we can follow the second approach in the proof of Lemma
\ref{lem:gotog}, and obtain an upper bound of $2 A + (4q/\pi) \sqrt{AC}$.

\end{proof}

The following bound will be useful when the constant $A$ in an application
of Lemma \ref{lem:couscous} would be too large. (This tends to happen for
$n$ small.)
\begin{lem}\label{lem:thina}
Let $\alpha = a/q + \beta/q Q$, $\gcd(a,q)=1$, $|\beta|\leq 1$, $q\leq Q$.
Let $y_2>y_1\geq 0$. If $y_2-y_1\leq q$ and $y_2\leq Q/2$, then, for
any $B,C\geq 0$,
\begin{equation}\label{eq:shtru}
\mathop{\sum_{y_1<n\leq y_2}}_{q\nmid n} 
\min\left(\frac{B}{|\sin(\pi \alpha n)|},\frac{C}{|\sin(\pi \alpha n)|^2}
\right) \leq 2 B \frac{q}{\pi} \max\left(2, \log \frac{C e^3 q}{B \pi}\right)
.
 \end{equation}
The upper bound $\leq (2 B q/\pi) \log (2 e^2 q/\pi)$ is also valid.
\end{lem}
\begin{proof}
As in the proof of Lemma \ref{lem:couscous}, we can bound the left side
of (\ref{eq:shtru}) by
\[
2 \sum_{r=1}^{q/2} \min\left(\frac{B}{
\sin \frac{\pi}{q} \left(r - \frac{1}{2}\right)}, 
\frac{C}{\sin^2 \frac{\pi}{q} \left(r - \frac{1}{2}\right)}\right)
.\]
Assume $B \sin(\pi/q) \leq C\leq B$. 
By the convexity of $1/\sin(t)$ and $1/\sin(t)^2$ for $t\in (0,\pi/2\rbrack$,
\[\begin{aligned}
\sum_{r=1}^{q/2} &\min\left(\frac{B}{
\sin \frac{\pi}{q} \left(r - \frac{1}{2}\right)}, 
\frac{C}{\sin^2 \frac{\pi}{q} \left(r - \frac{1}{2}\right)}\right) \\
&\leq \frac{B}{\sin \frac{\pi}{2 q}} + 
\int_1^{\frac{q}{\pi} \arcsin \frac{C}{B}} \frac{B}{\sin \frac{\pi}{q} t} 
dt + \int_{\frac{q}{\pi} \arcsin \frac{C}{B}}^{q/2} \frac{1}{\sin^2 \frac{\pi}{q} t}
dt \\ &\leq \frac{B}{\sin \frac{\pi}{2 q}} + 
\frac{q}{\pi} \left(B \left(\log \tan \left(\frac{1}{2} \arcsin \frac{C}{B}\right) -
\log \tan \frac{\pi}{2 q}\right) + C \cot \arcsin \frac{C}{B}\right)\\
&\leq \frac{B}{\sin \frac{\pi}{2 q}} + 
\frac{q}{\pi} \left(B \left(\log \cot \frac{\pi}{2 q} - \log \frac{C}{B-
\sqrt{B^2 - C^2}}\right) + \sqrt{B^2 - C^2}\right).
\end{aligned}\]

Now, for all $t\in (0,\pi/2)$,
\[\frac{2}{\sin t} + \frac{1}{t} \log \cot t 
< \frac{1}{t} \log\left(\frac{e^2}{t}\right);\]
we can verify this by comparing series. Thus
\[\frac{B}{\sin \frac{\pi}{2q}} + \frac{q}{\pi} B \log \cot \frac{\pi}{2 q}
\leq B \frac{q}{\pi} \log \frac{2 e^2 q}{\pi}\]
for $q\geq 2$.  (If $q=1$, the sum on the left of (\ref{eq:shtru}) is empty,
and so the bound we are trying to prove is trivial.)
We also have
\begin{equation}\label{eq:somos}
t \log(t - \sqrt{t^2-1}) + \sqrt{t^2 - 1} < -t \log 2t + t\end{equation}
for $t\geq 1$ (as this is equivalent to $\log(2 t^2 (1 - \sqrt{1 - t^{-2}}))
< 1 - \sqrt{1 - t^{-2}}$, which we check easily after changing variables to
$\delta = 1 - \sqrt{1 - t^{-2}}$). Hence
\[\begin{aligned}
\frac{B}{\sin \frac{\pi}{2 q}} &+ 
\frac{q}{\pi} \left(B \left(\log \cot \frac{\pi}{2 q} - \log \frac{C}{B-
\sqrt{B^2 - C^2}}\right) + \sqrt{B^2 - C^2}\right)\\ &\leq
B \frac{q}{\pi} \log \frac{2 e^2 q}{\pi} + \frac{q}{\pi} \left(B - B \log \frac{2 B}{C}
\right)
\leq B \frac{q}{\pi} \log \frac{C e^3 q}{B \pi}\end{aligned}\]
for $q\geq 2$.

Given any $C$, we can apply the above with $C=B$ instead, as, for any
$t>0$, $\min(B/t, C/t^2) \leq B/t \leq \min(B/t, B/t^2)$. 
(We refrain from applying (\ref{eq:somos}) so as to avoid worsening 
a constant.)
If $C < B \sin \pi/q$ (or even if $C < (\pi/q) B$), we relax the input to
$C = B \sin \pi/q$ and go through the above.
\end{proof}

\subsection{Type I estimates}

Our main type I estimate is the following.\footnote{The current version of
Lemma \ref{lem:bosta1} is an improvement over that included in the first
preprint of this paper.} One of the main innovations is the 
manner 
in which the ``main term'' ($m$ divisible by $q$) is separated; we are able
to keep error terms small thanks to the particular way in which we switch 
between two different approximations.

(These are {\em not} necessarily successive approximations in the sense
of continued fractions; we do not want to assume that the approximation
$a/q$ we are given arises from a continued fraction, and at any rate
we need more control on the denominator $q'$ of the new approximation 
$a'/q'$ than continued fractions would furnish.)

\begin{lem}\label{lem:bosta1}
Let $\alpha = a/q + \delta/x$, $\gcd(a,q)=1$, $|\delta/x|\leq 1/q Q_0$,
$q\leq Q_0$, $Q_0\geq 16$.
 Let $\eta$ be continuous, piecewise $C^2$ and compactly supported, with
$|\eta|_1 = 1$ and $\eta''\in L_1$. Let $c_0 \geq |\widehat{\eta''}|_\infty$.

Let $1\leq D\leq x$. Then, if $|\delta|\leq 1/2c_2$, where
$c_2 = (3 \pi/5\sqrt{c_0}) (1+\sqrt{13/3})$, the absolute value of 
\begin{equation}\label{eq:gorio}
\sum_{m\leq D} \mu(m) \sum_n e(\alpha m n) \eta\left(\frac{m n}{x}\right)
\end{equation} is at most
\begin{equation}\label{eq:cupcake}
\frac{x}{q} \min\left(1,\frac{c_0}{(2\pi \delta)^2}\right)
\left|
\mathop{\sum_{m\leq \frac{M}{q}}}_{\gcd(m,q)=1} \frac{\mu(m)}{m} \right| + 
O^*\left(c_0 
\left(\frac{1}{4} - \frac{1}{\pi^2}
\right)
\left(\frac{D^2}{2 x q} + \frac{D}{2x} \right)\right)
\end{equation}
plus
\begin{equation}\label{eq:kuche1}\begin{aligned} 
&\frac{2 \sqrt{c_0 c_1}}{\pi} D +
3 c_1 \frac{x}{q} \log^+ \frac{D}{c_2 x/q} 
+ \frac{\sqrt{c_0 c_1}}{\pi} q \log^+ \frac{D}{q/2}\\
&+ \frac{|\eta'|_1}{\pi} q \cdot \max\left(2, \log \frac{c_0 e^3 q^2}{4 \pi |\eta'|_1 x}\right) +
 \left(\frac{2 \sqrt{3 c_0 c_1}}{\pi} + 
\frac{3 c_1}{c_2}
+ \frac{55 c_0 c_2}{12 \pi^2} \right) q
 ,\end{aligned}\end{equation}
where $c_1 = 1 + |\eta'|_1/(2 x/D)$
 and $M\in \lbrack \min(Q_0/2,D),D\rbrack$.
The same bound holds if $|\delta|\geq 1/2c_2$ but $D\leq Q_0/2$.

In general, if $|\delta|\geq 1/2c_2$, the absolute value of
(\ref{eq:gorio}) is at most (\ref{eq:cupcake}) plus
\begin{equation}\label{eq:kallervo}\begin{aligned}
&\frac{2 \sqrt{c_0 c_1}}{\pi} \left(
D  + (1+\epsilon) \min\left(\left\lfloor \frac{x}{|\delta| q}\right\rfloor + 1, 2 D\right)
 \left(\varpi_\epsilon +
 \frac{1}{2} \log^+ \frac{2 D}{
\frac{x}{|\delta| q}}\right)\right)\\
&+ 3 c_1
 \left(2 + \frac{(1+\epsilon)}{\epsilon} \log^+ \frac{2D}{
\frac{x}{|\delta| q}}\right) \frac{x}{Q_0} +
\frac{35 c_0 c_2}{6 \pi^2} q,
\end{aligned}\end{equation}
for $\epsilon\in (0,1\rbrack$ arbitrary, where $\varpi_\epsilon = 
\sqrt{3+2\epsilon} + ((1+\sqrt{13/3})/4-1)/(2 (1+\epsilon))$.
\end{lem}
In (\ref{eq:cupcake}), $\min(1,c_0/(2\pi \delta)^2)$ always equals $1$ when
$|\delta|\leq 1/2c_2$ (since $(3/5) (1 + \sqrt{13/3}) > 1$).
\begin{proof}
Let $Q = \lfloor x/|\delta q|\rfloor$. Then
$\alpha = a/q + O^*(1/q Q)$ and $q\leq Q$. (If $\delta=0$, we 
let $Q=\infty$ and ignore the rest of the paragraph, since then we will never
need $Q'$ or the alternative approximation $a'/q'$.) Let
$Q' = \lceil (1+\epsilon) Q\rceil \geq Q+1$. 
Then $\alpha$ is {\em not} $a/q + O^*(1/q Q')$, and so
there must be a different approximation $a'/q'$, $\gcd(a',q')=1$, $q'\leq Q'$
such that $\alpha = a'/q' + O^*(1/q' Q')$ (since such an approximation
always exists).
Obviously, $|a/q - a'/q'|\geq 1/q q'$, yet, at the same time,
$|a/q- a'/q'|\leq 1/q Q + 1/q' Q' \leq 1/q Q + 1/((1+\epsilon) q' Q)$. 
Hence $q'/Q + q/((1+\epsilon)Q) \geq 1$, and so $q'\geq Q-q/(1+\epsilon)
\geq (\epsilon/(1+\epsilon)) Q$.
(Note also that $(\epsilon/(1+\epsilon)) Q \geq (2|\delta q|/x)\cdot
\lfloor x/\delta q\rfloor>1$, and so $q'\geq 2$.)

Lemma \ref{lem:couscous} will enable us to treat separately the
contribution from terms with $m$ divisible by $q$ and $m$ not divisible by
$q$, provided that $m\leq Q/2$. Let $M = \min(Q/2,D)$.
We start by considering all terms with $m\leq M$ divisible by $q$.
Then $e(\alpha m n)$ equals $e((\delta m/x) n)$. By Poisson summation,
\[\sum_n e(\alpha m n) \eta(m n/x) = \sum_n \widehat{f}(n),\]
where $f(u) = e((\delta m/x) u) \eta((m/x) u)$. Now
\[\widehat{f}(n) = \int e(-un) f(u) du = \frac{x}{m} \int
e\left(\left(\delta-\frac{x n}{m}\right) u\right) \eta(u) du = 
\frac{x}{m} \widehat{\eta}\left(\frac{x}{m} n -\delta\right).\]
By assumption, $m\leq M \leq 
Q/2\leq x/2|\delta q|$, and so $|x/m|\geq 2 |\delta q| \geq
2 \delta$. Thus, by (\ref{eq:malcros}),
\begin{equation}\label{eq:crinfo}\begin{aligned}
\sum_n \widehat{f}(n) &= \frac{x}{m} \left(\widehat{\eta}(-\delta) + 
 \sum_{n\neq 0} \widehat{\eta}\left(\frac{n x}{m} - \delta\right)\right)\\
&= \frac{x}{m} \left(\widehat{\eta}(-\delta) + O^*\left(\sum_{n\neq 0}
 \frac{1}{\left(2\pi \left(\frac{n x}{m} - \delta\right)\right)^2}\right)
\cdot \left|\widehat{\eta''}\right|_\infty\right)\\
&= 
\frac{x}{m} \widehat{\eta}(-\delta) + \frac{m}{x} \frac{c_0}{(2\pi)^2}
O^*\left(\max_{|r|\leq \frac{1}{2}} \sum_{n\neq 0} \frac{1}{(n-r)^2}\right) .
\end{aligned}\end{equation}
Since $x\mapsto 1/x^2$ is convex on $\mathbb{R}^+$, 
\[\max_{|r|\leq \frac{1}{2}}
\sum_{n\neq 0} \frac{1}{(n-r)^2} = \sum_{n\neq 0} \frac{1}{
\left(n-\frac{1}{2}\right)^2} = 
\pi^2 - 4.\]

Therefore, the sum of all terms with $m\leq M$
and $q|m$ is
\[\begin{aligned}
&\mathop{\sum_{m\leq M}}_{q|m} \frac{x}{m} \widehat{\eta}(-\delta) +
\mathop{\sum_{m\leq M}}_{q|m} \frac{m}{x} 
\frac{c_0}{(2 \pi)^2} (\pi^2 - 4) \\
&=\frac{x \mu(q)}{q}\cdot
\widehat{\eta}(-\delta) \cdot 
\mathop{\sum_{m\leq \frac{M}{q}}}_{\gcd(m,q)=1} \frac{\mu(m)}{m} \\
&+ 
O^*\left(\mu(q)^2 c_0 
\left(\frac{1}{4} - \frac{1}{\pi^2}
\right)
\left(\frac{D^2}{2 x q} + \frac{D}{2 x} \right)\right).
\end{aligned}.\]
We bound $|\widehat{\eta}(-\delta)|$ by (\ref{eq:malcros}).

Let
\[T_m(\alpha) = \sum_n e(\alpha m n) \eta\left(\frac{m n}{x}\right).\]
Then, by (\ref{eq:ra}) and Lemma \ref{lem:areval},
\begin{equation}\label{eq:aoro}\begin{aligned}
|T_m(\alpha)| \leq \min\left(\frac{x}{m} + \frac{1}{2} |\eta'|_1,
\frac{\frac{1}{2} |\eta'|_1 }{|\sin(\pi m \alpha)|},
\frac{m}{x} \frac{c_0}{4} \frac{1}{(\sin \pi m \alpha)^2}\right).
\end{aligned}\end{equation}

For any $y_2>y_1>0$ with $y_2-y_1\leq q$ and $y_2\leq Q/2$,
(\ref{eq:aoro}) gives us that
\begin{equation}\label{eq:gowo}
\mathop{\sum_{y_1<m\leq y_2}}_{q\nmid m} |T_m(\alpha)|\leq
\mathop{\sum_{y_1<m\leq y_2}}_{q\nmid m} \min\left(A, \frac{C}{
(\sin \pi m \alpha)^2}\right)\end{equation}
for $A = (x/y_1) ( 1 + |\eta'|_1/(2(x/y_1)))$ and $C = (c_0/4) (y_2/x)$.
We must now estimate the sum
\begin{equation}\label{eq:esthel}
\mathop{\sum_{m\leq M}}_{q\nmid m} 
|T_m(\alpha)| + \sum_{\frac{Q}{2} < m \leq D} |T_m(\alpha)|.
\end{equation}


To bound the terms with $m\leq M$, we can use Lemma \ref{lem:couscous}.
The question is then which one is smaller: the first or the second
bound given by Lemma \ref{lem:couscous}? A brief calculation gives that
the second bound is smaller (and hence preferable) exactly when 
$\sqrt{C/A} > (3 \pi/10 q) (1 + \sqrt{13/3})$. Since $\sqrt{C/A} \sim
(\sqrt{c_0}/2) m/x$, this means that it is sensible to prefer the
 second bound in Lemma \ref{lem:couscous} when $m> c_2 x /q$,
where $c_2 = (3 \pi/ 5\sqrt{c_0}) (1+\sqrt{13/3})$.

It thus makes sense to ask: does $Q/2\leq c_2 x/q$ (so that $m\leq M$
implies $m\leq c_2 x/q$)? This question divides our work into two
basic cases.

Case (a). {\em $\delta$ large:  $|\delta|\geq 1/2c_2$,
where $c_2 = (3 \pi/5\sqrt{c_0}) (1+\sqrt{13/3})$.}
Then $Q/2 \leq c_2 x/q$; this will induce us to 
bound the first sum in (\ref{eq:esthel})
by the very first bound in Lemma \ref{lem:couscous}. 

Recall that $M = \min(Q/2,D)$, and so $M\leq c_2 x/q$.
By (\ref{eq:gowo}) and Lemma \ref{lem:couscous}, 
\begin{equation}\label{eq:jenuf}\begin{aligned}
\mathop{\sum_{1 \leq m \leq M}}_{q\nmid m} &|T_m(\alpha)| \leq
\sum_{j=0}^{\infty} \mathop{\sum_{j q < m\leq \min((j+1) q ,
    M)}}_{q\nmid m}
\min\left( 
\frac{x}{jq+1} +\frac{|\eta'|_1}{2}, 
\frac{\frac{c_0}{4} \frac{(j+1) q}{x}}{(\sin \pi m \alpha)^2}\right) \\
&\leq \frac{20}{3 \pi^2} \frac{c_0 q^3}{4 x} 
\sum_{0\leq j\leq \frac{M}{q}} (j+1) \leq \frac{20}{3 \pi^2} \frac{c_0 q^3}{4 x} \cdot
\left(\frac{1}{2} \frac{M^2}{q^2}+ 
\frac{3}{2} \frac{c_2 x}{q^2} + 1\right) \\ &\leq
 \frac{5 c_0 c_2}{6  \pi^2} M + 
\frac{5 c_0 q}{3 \pi^2}  \left(\frac{3}{2} c_2 + 
\frac{q^2}{x}\right) \leq
 \frac{5 c_0 c_2}{6  \pi^2} M + 
\frac{35 c_0 c_2}{6 \pi^2} q,
\end{aligned}\end{equation}
where, to bound the smaller terms, we are using the inequality 
$Q/2\leq c_2 x/q$, and where we are also 
using the observation that, since $|\delta/x| \leq 1/q Q_0$, 
the assumption $|\delta|\geq 1/2c_2$ implies that $q\leq 2 c_2 x/Q_0$;
moreover, since $q\leq Q_0$, this gives us that $q^2\leq 2 c_2 x$.
In the main term, we are bounding
$q M^2/x$ from above by $M \cdot q Q/2 x\leq M/2\delta \leq c_2 M$.

If $D\leq (Q+1)/2$, then $M\geq \lfloor D\rfloor$ 
and so (\ref{eq:jenuf}) is all we need.
Assume from now on that $D> (Q+1)/2$. The first sum in (\ref{eq:esthel}) is then
bounded by (\ref{eq:jenuf}) (with $M=Q/2$).
To bound the second sum in (\ref{eq:esthel}), we use
the approximation $a'/q'$
 instead of $a/q$. 
By (\ref{eq:gowo}) (without the restriction $q\nmid m$) 
and Lemma \ref{lem:gotog},
\[\begin{aligned}
&\sum_{Q/2 < m \leq D} |T_m(\alpha)| 
\leq \sum_{j=0}^\infty \sum_{jq' + \frac{Q}{2}< m \leq 
\min((j+1)q'+Q/2,D)} |T_m(\alpha)|\\
&\leq \sum_{j=0}^{\left\lfloor \frac{D-(Q+1)/2}{q'}\right\rfloor} 
\left(3 c_1 \frac{x}{jq'+ \frac{Q+1}{2}} 
+ \frac{4 q'}{\pi} \sqrt{\frac{c_1 c_0}{4}
 \frac{x}{jq'+ (Q+1)/2}
\frac{(j+1)q'+Q/2}{x}}\right)\\
&\leq \sum_{j=0}^{\left\lfloor \frac{D-(Q+1)/2}{q'}\right\rfloor} 
\left(3 c_1 \frac{x}{jq'+ \frac{Q+1}{2}}
+ \frac{4 q'}{\pi}  \sqrt{\frac{c_1 c_0}{4}
 \left(1 + \frac{q'}{jq'+ (Q+1)/2}\right)}\right),
\end{aligned}\]
where we recall that $c_1 = 1 + |\eta'|_1/(2x/D)$.
Since $q'\geq (\epsilon/(1+\epsilon)) Q$,
\begin{equation}\label{eq:tenda}
\sum_{j=0}^{\left\lfloor \frac{D-(Q+1)/2}{q'}\right\rfloor} 
\frac{x}{jq'+ \frac{Q+1}{2}} 
\leq \frac{x}{Q/2} + \frac{x}{q'} \int_{\frac{Q+1}{2}}^D \frac{1}{t} dt \leq 
\frac{2 x}{Q} + \frac{(1+\epsilon) x}{\epsilon Q} \log^+ \frac{D}{
\frac{Q+1}{2}} .
\end{equation}
Recall now that $q' \leq (1+\epsilon) Q + 1 \leq (1+\epsilon)(Q+1)$.
Therefore,
\begin{equation}\label{eq:beatri}\begin{aligned}
q' &\sum_{j=0}^{\lfloor \frac{D-(Q+1)/2}{q'}\rfloor}
\sqrt{1 + \frac{q'}{jq'+ (Q+1)/2}} \leq 
q' \sqrt{1 + \frac{(1+\epsilon)Q+1}{(Q+1)/2}} + 
\int_{\frac{Q+1}{2}}^{D} \sqrt{1 + \frac{q'}{t}} dt\\
&\leq q'\sqrt{3 + 2\epsilon} + \left(D-\frac{Q+1}{2}\right) 
+ \frac{q'}{2} \log^+ \frac{D}{
\frac{Q+1}{2}} 
.\end{aligned}\end{equation}
We conclude that 
$\sum_{Q/2 < m\leq D} |T_m(\alpha)|$ is at most
\begin{equation}\label{eq:sauna}\begin{aligned}
&\frac{2 \sqrt{c_0 c_1}}{\pi} \left(
D  + \left((1+\epsilon) \sqrt{3+2\epsilon}  - \frac{1}{2} \right) (Q+1) +
 \frac{(1+\epsilon)Q+1}{2} \log^+ \frac{D}{\frac{Q+1}{2}}\right)\\
&+ 3 c_1 \left(2 + \frac{(1+\epsilon)}{\epsilon} \log^+ \frac{D}{\frac{Q+1}{2}} \right)
\frac{x}{Q}
\end{aligned}\end{equation}
 We sum this to
(\ref{eq:jenuf})
(with $M=Q/2$), and obtain that (\ref{eq:esthel}) is at most
\begin{equation}\label{eq:kullervo}\begin{aligned}
&\frac{2 \sqrt{c_0 c_1}}{\pi} \left(
D  + (1+\epsilon) (Q+1) \left(\varpi_\epsilon +
 \frac{1}{2} \log^+ \frac{D}{\frac{Q+1}{2}}\right)\right)\\
&+ 3 c_1
 \left(2 + \frac{(1+\epsilon)}{\epsilon} \log \frac{D}{\frac{Q+1}{2}}\right) 
\frac{x}{Q} + \frac{35 c_0 c_2}{6 \pi^2} q,
\end{aligned}\end{equation}
where we are bounding
\begin{equation}\label{eq:semin1}
\frac{5 c_0 c_2}{6 \pi^2} = 
\frac{5 c_0}{6 \pi^2} \frac{3\pi}{5 \sqrt{c_0}} \left(1 +
  \sqrt{\frac{13}{3}} \right) = \frac{\sqrt{c_0}}{2 \pi}
\left(1 + \sqrt{\frac{13}{3}}\right) \leq \frac{2 \sqrt{c_0 c_1}}{\pi} \cdot 
\frac{1}{4}
\left(1 + \sqrt{\frac{13}{3}}\right)
\end{equation}
and defining
\begin{equation}\label{eq:semin2}
\varpi_\epsilon = \sqrt{3 + 2 \epsilon} + 
\left(\frac{1}{4} \left(1 + \sqrt{\frac{13}{3}}\right) -
  1\right) \frac{1}{2 (1+\epsilon)}.\end{equation}
(Note that $\varpi_\epsilon<\sqrt{3}$ for $\epsilon < 0.1741$).
A quick check against (\ref{eq:jenuf}) shows that (\ref{eq:kullervo})
is valid also when $D\leq Q/2$, even when $Q+1$ is replaced by $\min(Q+1,2D)$. 
We bound $Q$ from above by $x/|\delta| q$
and $\log^+ D/((Q+1)/2)$ by $\log^+ 2 D/(x/|\delta| q+1)$,
and obtain the result. 

Case (b): {\em $|\delta|$ small: $|\delta|\leq 1/2 c_2$ or $D\leq Q_0/2$.}
Then $\min(c_2 x/q,D) \leq Q/2$. We start by bounding the first $q/2$ terms 
in (\ref{eq:esthel}) by (\ref{eq:aoro}) and Lemma \ref{lem:thina}:
\begin{equation}\label{eq:prokof}\begin{aligned}
\sum_{m\leq q/2} |T_m(\alpha)| &\leq
\sum_{m\leq q/2} \min\left(
\frac{\frac{1}{2} |\eta'|_1}{|\sin(\pi m \alpha)|},
\frac{c_0 q/8 x}{|\sin(\pi m \alpha)|^2}\right)\\ &\leq
\frac{|\eta'|_1}{\pi} q \max\left(2, \log \frac{c_0 e^3 q^2}{4 \pi |\eta'|_1 x}\right)
.\end{aligned}\end{equation}

If $q^2 < 2 c_2 x$, we estimate the terms with $q/2 < m \leq c_2 x/q$ by
Lemma \ref{lem:couscous}, which is applicable because $\min(c_2 x/q,D) <
Q/2$:
\begin{equation}\label{eq:martinu}\begin{aligned}
\mathop{\sum_{\frac{q}{2} < m \leq D'}}_{q\nmid m} &|T_m(\alpha)|\;
\leq \; \sum_{j=1}^{\infty} 
\mathop{\mathop{\sum_{\left(j-\frac{1}{2}\right) q < m\leq \left(j + \frac{1}{2}
\right) q}}_{m \leq \min\left(\frac{c_2 x}{q},D\right)}}_{q\nmid m} \min\left(\frac{x}{\left(j - 
\frac{1}{2}\right) q} + \frac{|\eta_1'|}{2}, \frac{\frac{c_0}{4}
\frac{(j+1/2) q}{x}}{(\sin \pi m \alpha)^2}\right)\\
&\leq \frac{20}{3 \pi^2} \frac{c_0 q^3}{4 x}
\sum_{1\leq j\leq \frac{D'}{q} + \frac{1}{2}} \left(j+\frac{1}{2}\right)
\leq \frac{20}{3 \pi^2} \frac{c_0 q^3}{4 x} 
\left(\frac{c_2 x}{2 q^2} \frac{D'}{q} + \frac{3}{2}
\left(\frac{c_2 x}{q^2}\right) + \frac{5}{8}\right)\\
&\leq \frac{5 c_0}{6 \pi^2} \left(c_2 D' + 3 c_2 q + 
\frac{5}{4} \frac{q^3}{x}\right) \leq
\frac{5 c_0 c_2}{6 \pi^2} \left(D' + \frac{11}{2} q \right),
\end{aligned}\end{equation}
where we write $D' = \min(c_2 x/q,D)$.
If $c_2 x/q\geq D$, we stop here. 
Assume that $c_2 x/q < D$. Let $R = \max(c_2 x/q, q/2)$. The terms we
have already estimated are precisely those with $m\leq R$.
We bound the terms $R < m\leq D$ by
the second bound in Lemma \ref{lem:gotog}:
\begin{equation}\label{eq:caron}\begin{aligned}
\sum_{R< m \leq D} &|T_m(\alpha)| \leq
\sum_{j=0}^{\infty} \mathop{\sum_{m > j q + R}}_{m\leq
\min\left((j+1) q + R,D\right)}
\min\left(\frac{c_1 x}{jq + R},
\frac{\frac{c_0}{4} \frac{(j+1) q + R}{x}}{(\sin \pi m \alpha)^2}
\right)\\
&\leq \sum_{j=0}^{\left\lfloor \frac{1}{q} \left(D - R\right)
\right\rfloor} \frac{3 c_1 x}{j q + R} + 
\frac{4q}{\pi} \sqrt{\frac{c_1 c_0}{4}
\left(1 + \frac{q}{j q + R}\right)}
\end{aligned}\end{equation}
(Note there is no need to use two successive approximations 
$a/q$, $a'/q'$ as in case (a). We are also including all terms with $m$
divisible by $q$, as we may, since $|T_m(\alpha)|$ is non-negative.) 
Now, much as before,
\begin{equation}\label{eq:kosto}
\sum_{j=0}^{\left\lfloor \frac{1}{q} \left(D - R\right)
\right\rfloor} \frac{x}{j q + R} \leq
\frac{x}{R} + \frac{x}{q} \int_{R}^{D} \frac{1}{t} dt
\leq \min\left(\frac{q}{c_2}, \frac{2 x}{q}\right) 
+ \frac{x}{q} \log^+ \frac{D}{c_2 x/q} ,\end{equation}
and
\begin{equation}\label{eq:kostas}
\sum_{j=0}^{\left\lfloor \frac{1}{q} \left(D - R\right)
\right\rfloor}  \sqrt{1 + \frac{q}{j q + R}} \leq
\sqrt{1 + \frac{q}{R}} + \frac{1}{q} \int_R^D \sqrt{1 + \frac{q}{t}} dt
\leq \sqrt{3} + \frac{D-R}{q} + \frac{1}{2} \log^+ \frac{D}{q/2}.
\end{equation}
We sum with (\ref{eq:prokof}) and (\ref{eq:martinu}),
and we obtain that (\ref{eq:esthel}) is at most
\begin{equation}\label{eq:kabanova}\begin{aligned}
&\frac{2 \sqrt{c_0 c_1}}{\pi} \left(\sqrt{3} q + D + \frac{q}{2}
\log^+ \frac{D}{q/2}\right) 
+ 
\left(3 c_1 \log^+ \frac{D}{c_2 x/q}\right) \frac{x}{q}\\
&+ 3 c_1 \min\left(\frac{q}{c_2}, \frac{2 x}{q}\right) 
+ \frac{55 c_0 c_2}{12 \pi^2} q
+ \frac{|\eta'|_1}{\pi} q\cdot \max\left(2, \log \frac{c_0 e^3 q^2}{4 \pi |\eta'|_1 x}\right),\end{aligned}\end{equation}
where we are using the fact that $5 c_0 c_2/6\pi^2 < 2 \sqrt{c_0 c_1}/\pi$.
A quick check against (\ref{eq:martinu}) shows that (because of the
fact just stated) (\ref{eq:kabanova})
 is also valid when
$c_2 x/q \geq D$.
\end{proof}

We will need a version of Lemma
\ref{lem:bosta1} with $m$ and $n$ restricted to the
odd numbers. (We will barely be using
the restriction of $m$, whereas the restriction on $n$ is both (a) slightly
harder to deal with, (b) something that can be turned to our advantage.)
\begin{lem}\label{lem:bosta2}
Let $\alpha\in \mathbb{R}/\mathbb{Z}$ with
$2 \alpha = a/q + \delta/x$, $\gcd(a,q)=1$, $|\delta/x|\leq 1/q Q_0$,
$q\leq Q_0$, $Q_0\geq 16$.
 Let $\eta$ be continuous, piecewise $C^2$ and compactly supported, with
$|\eta|_1 = 1$ and $\eta''\in L_1$. Let $c_0 \geq |\widehat{\eta''}|_\infty$.

Let $1\leq D\leq x$. Then, if $|\delta|\leq 1/2c_2$, where
$c_2 = 6 \pi/5 \sqrt{c_0}$, the absolute value of 
\begin{equation}\label{eq:gorio2}
\mathop{\sum_{m\leq D}}_{\text{$m$ odd}} \mu(m) 
\sum_{\text{$n$ odd}} e(\alpha m n) \eta\left(\frac{m n}{x}\right)
\end{equation} is at most
\begin{equation}\label{eq:asparto}
\frac{x}{2 q} \min\left(1,\frac{c_0}{(\pi \delta)^2}\right)
\left|
\mathop{\sum_{m\leq \frac{M}{q}}}_{\gcd(m,2q)=1} \frac{\mu(m)}{m} \right| + 
O^*\left(\frac{c_0 q}{x} 
\left(\frac{1}{8} - \frac{1}{2 \pi^2}\right)
\left(\frac{D}{q} + 1 \right)^2\right)
\end{equation}
plus
\begin{equation}\label{eq:keks}\begin{aligned} 
&\frac{2 \sqrt{c_0 c_1}}{\pi} D +
\frac{3 c_1}{2} \frac{x}{q} 
\log^+ \frac{D}{c_2 x/q} 
+ \frac{\sqrt{c_0 c_1}}{\pi} q \log^+ \frac{D}{q/2}\\
&+ \frac{2 |\eta'|_1}{\pi} q \cdot \max\left(1, \log \frac{c_0 e^3 q^2}{4 \pi |\eta'|_1 x}\right) +
 \left(\frac{2 \sqrt{3 c_0 c_1}}{\pi} + 
\frac{3 c_1}{2 c_2}
+ \frac{55 c_0 c_2}{6 \pi^2} \right) q ,\end{aligned}\end{equation}
where $c_1 = 1 + |\eta'|_1/(x/D)$
 and $M\in \lbrack \min(Q_0/2,D),D\rbrack$.
The same bound holds if $|\delta|\geq 1/2c_2$ but $D\leq Q_0/2$.

In general, if $|\delta|\geq 1/2c_2$, the absolute value of
(\ref{eq:gorio}) is at most (\ref{eq:asparto}) plus
\begin{equation}\label{eq:kallervo2}\begin{aligned}
&\frac{2 \sqrt{c_0 c_1}}{\pi} \left(
D  + (1+\epsilon) \min\left(\left\lfloor \frac{x}{|\delta| q}\right\rfloor + 1, 2 D\right)
 \left(\sqrt{3+2\epsilon} +
 \frac{1}{2} \log^+ \frac{2 D}{\frac{x}{|\delta| q}}\right)\right)\\
&+ \frac{3}{2} c_1
 \left(2 + \frac{(1+\epsilon)}{\epsilon} \log^+ \frac{2D}{
\frac{x}{|\delta| q}}\right) \frac{x}{Q_0} +
\frac{35 c_0 c_2}{3 \pi^2} q,
\end{aligned}\end{equation}
for $\epsilon\in (0,1\rbrack$ arbitrary.
\end{lem}

If $q$ is even, the sum (\ref{eq:asparto}) can be replaced by $0$.
\begin{proof}
The proof is almost exactly that of Lemma \ref{lem:bosta1}; we go
over the differences. The parameters $Q$, $Q'$, $a'$, $q'$ and $M$
are defined just as before (with $2\alpha$ wherever we had $\alpha$).

Let us first consider $m\leq M$ odd and divisible by $q$.
(Of course, this case arises only if $q$ is odd.) For $n = 2 r + 1$,
\[\begin{aligned}
e(\alpha m n) &= e(\alpha m (2 r + 1)) = e(2\alpha r m) e(\alpha m)
 = e\left(\frac{\delta}{x} r m\right) 
e\left(\left(\frac{a}{2 q} + \frac{\delta}{ 2 x} + \frac{\kappa}{2}\right)
m\right)\\ &= 
e\left(\frac{\delta (2 r + 1)}{2 x} m\right)
e\left(\frac{a + \kappa q}{2} \frac{m}{q}\right) = 
\kappa' e\left(\frac{\delta (2 r + 1)}{2 x} m\right),\end{aligned}\]
where $\kappa \in \{0,1\}$ and $\kappa' = e((a+\kappa q)/2) \in \{-1,1\}$ 
are independent of $m$ and $n$. Hence, by Poisson summation,
\begin{equation}\label{eq:kormo}\begin{aligned}
\sum_{\text{$n$ odd}}
 e(\alpha m n) \eta(m n/x) &= \kappa' \sum_{\text{$n$ odd}}
 e((\delta m/2 x) n) \eta(m n/x) \\&= \frac{\kappa'}{2}
\left(\sum_n \widehat{f}(n) - \sum_n \widehat{f}(n+1/2)\right),\end{aligned}
\end{equation}
where $f(u) = e((\delta m/2 x) u) \eta((m/x) u)$. Now
\[\widehat{f}(t) = \frac{x}{m} \widehat{\eta}\left(\frac{x}{m} t - 
\frac{\delta}{2}\right).\] Just as before, $|x/m|\geq 2 |\delta q|\geq 2 \delta$. Thus
\begin{equation}\label{eq:karma}\begin{aligned}
\frac{1}{2}&\left|\sum_n \widehat{f}(n) - \sum_n \widehat{f}(n+1/2)\right|
\leq \frac{x}{m} \left(\frac{1}{2}\left|
\widehat{\eta}\left(-\frac{\delta}{2}\right) \right|
+ \frac{1}{2} \sum_{n\ne 0} \left|\widehat{\eta}\left(\frac{x}{m} \frac{n}{2} -
\frac{\delta}{2}\right)\right|\right)\\
&= \frac{x}{m} \left(\frac{1}{2}\left|
\widehat{\eta}\left(-\frac{\delta}{2}\right) \right|
+ \frac{1}{2} \cdot  O^*\left(\sum_{n\ne 0} \frac{1}{\left(\pi \left(\frac{n x}{m}
- \delta\right)\right)^2}\right) \cdot  \left|\widehat{\eta''}\right|_\infty
\right)\\
&= \frac{x}{2 m} \left|\widehat{\eta}\left(-\frac{\delta}{2}\right)\right| + \frac{m}{x} \frac{c_0}{2 \pi^2}
 (\pi^2 - 4)x
.\end{aligned}\end{equation}
The contribution of the second term in the last line of (\ref{eq:karma})
is
\[\begin{aligned}
\mathop{\mathop{\sum_{m\leq M}}_{\text{$m$ odd}}}_{q|m} \frac{m}{x} \frac{c_0}{
2\pi^2} (\pi^2-4) = \frac{q}{x} \frac{c_0}{2\pi^2} (\pi^2-4) \cdot
\mathop{\sum_{m\leq M/q}}_{\text{$m$ odd}} m = \frac{q c_0}{x} \left(\frac{1}{8}
- \frac{1}{2\pi^2}\right) \left(\frac{M}{q}+1\right)^2.
\end{aligned}\]
Hence, the absolute value of the sum of all terms with $m\leq M$ and 
$q|m$ is given by (\ref{eq:asparto}).

We define $T_{m,\circ}(\alpha)$ by
\begin{equation}\label{eq:baxter}
T_{m,\circ}(\alpha) = \sum_{\text{$n$ odd}} e(\alpha m n) \eta\left(\frac{m n}{x}\right)
.\end{equation}
Changing variables by $n = 2 r + 1$, we see that
 \[|T_{m,\circ}(\alpha)| = \left|\sum_r e(2\alpha \cdot m r) \eta(m (2 r + 1)/x)
\right|.\]
Hence,
instead of (\ref{eq:aoro}), we get that
\begin{equation}\label{eq:trompais}
|T_{m,\circ}(\alpha)| \leq \min\left(\frac{x}{2 m} + \frac{1}{2} |\eta'|_1,
\frac{\frac{1}{2} |\eta'|_1}{|\sin(2\pi m \alpha)|},
\frac{m}{x} \frac{c_0}{2} \frac{1}{(\sin 2 \pi m \alpha)^2}\right)
.\end{equation}
We obtain (\ref{eq:gowo}), but with $T_{m,\circ}$ instead of $T_m$,
$A = (x/2 y_1) (1 + |\eta'|_1/(x/y_1))$ and $C = (c_0/2) (y_2/x)$,
and so $c_1 = 1 + |\eta'|_1/(x/D)$.

The rest of the proof of Lemma \ref{lem:bosta1} carries almost over 
word-by-word. (For the sake of simplicity, 
we do not really try to take advantage of the odd support
of $m$ here.) Since $C$ has doubled, it would seem to make
sense to reset the value of $c_2$ to be
$c_2 = (3 \pi/5\sqrt{2 c_0}) (1+\sqrt{13/3})$; this would cause
complications related to the fact that $5 c_0 c_2/3 \pi^2$ would become
larger than $2\sqrt{c_0}/\pi$, and so we set $c_2$ to the slightly smaller
value $c_2 = 6\pi/5\sqrt{c_0}$ instead. This implies 
\begin{equation}\label{eq:sosot}
\frac{5 c_0 c_2}{3 \pi^2} = \frac{2\sqrt{c_0}}{\pi}.\end{equation}
 The bound from (\ref{eq:jenuf}) gets multiplied by $2$ 
(but the value of $c_2$ has changed),
the second line in (\ref{eq:sauna}) gets halved,
(\ref{eq:semin1}) gets replaced by (\ref{eq:sosot}),
the second term in the maximum in the second line of (\ref{eq:prokof}) gets
doubled, the bound from (\ref{eq:martinu}) gets doubled,
and the bound from
(\ref{eq:kosto}) gets halved. 
\end{proof}

We will also need a version of Lemma \ref{lem:bosta1} (or
rather Lemma \ref{lem:bosta2}; we will decide to work with the
restriction that $n$ and $m$ be odd)
 with
a factor of $(\log n)$ within the inner sum.
\begin{lem}\label{lem:bostb1}
Let $\alpha\in \mathbb{R}/\mathbb{Z}$ with
$2 \alpha = a/q + \delta/x$, $\gcd(a,q)=1$, $|\delta/x|\leq 1/q Q_0$,
$q\leq Q_0$, $Q_0\geq \max(16,2 \sqrt{x})$.
 Let $\eta$ be continuous, piecewise $C^2$ and compactly supported, with
$|\eta|_1 = 1$ and $\eta''\in L_1$. Let $c_0 \geq |\widehat{\eta''}|_\infty$.
Assume that, 
for any $\rho\geq \rho_0$, $\rho_0$ a constant, the function
$\eta_{(\rho)}(t) = \log(\rho t) \eta(t)$ satisfies
\begin{equation}\label{eq:puella}
|\eta_{(\rho)}|_1 \leq \log(\rho) |\eta|_1,\;\;\;\;
|\eta_{(\rho)}'|_1 \leq \log(\rho) |\eta'|_1,\;\;\;\;
|\widehat{\eta_{(\rho)}''}|_\infty \leq c_0 \log(\rho) 
\end{equation}

Let $\sqrt{3}\leq D\leq \min(x/\rho_0,x/e)$. Then, if $|\delta|\leq
1/2c_2$, where $c_2 = 6 \pi/5 \sqrt{c_0}$,
 the absolute value of 
\begin{equation}\label{eq:bovary}
\mathop{\sum_{m\leq D}}_{\text{$m$ odd}}
 \mu(m) \mathop{\sum_n}_{\text{$n$ odd}}
 (\log n) e(\alpha m n) \eta\left(\frac{m n}{x}\right)
\end{equation} is at most
\begin{equation}\label{eq:cupcake2}\begin{aligned}
&\frac{x}{q} 
 \min\left(1,\frac{c_0/\delta^2}{(2\pi)^2}\right)
\left|\mathop{\sum_{m\leq \frac{M}{q}}}_{\gcd(m,q)=1} \frac{\mu(m)}{m} 
\log \frac{x}{m q}\right| + \frac{x}{q}
|\widehat{\log \cdot \eta}(-\delta)| \left|\mathop{\sum_{m\leq \frac{M}{q}}}_{\gcd(m,q)=1} \frac{\mu(m)}{m}\right|\\
&+ 
O^*\left(c_0 
\left(\frac{1}{2} - \frac{2}{\pi^2}
\right)
\left(\frac{D^2}{4 q x} \log \frac{e^{1/2} x}{
D} + \frac{1}{e} \right) 
\right)
\end{aligned}\end{equation}
plus
\begin{equation}\label{eq:kuche2}\begin{aligned} 
&\frac{2 \sqrt{c_0 c_1}}{\pi}  
D \log \frac{e x}{D} + 
\frac{3 c_1}{2} \frac{x}{q} \log^+ \frac{D}{c_2 x/q} \log \frac{q}{c_2}\\
&+ \left( \frac{2 |\eta'|_1}{\pi} 
\max\left(1, \log \frac{c_0 e^3 q^2}{4 \pi |\eta'|_1 x}\right) 
\log x + \frac{2 \sqrt{c_0 c_1}}{\pi} 
\left(\sqrt{3} + \frac{1}{2} \log^+ \frac{D}{q/2}\right)
\log \frac{q}{c_2}\right) q\\
&+  \frac{3 c_1}{2} \sqrt{\frac{2 x}{c_2}} \log \frac{2 x}{c_2} 
+ \frac{20 c_0 c_2^{3/2}}{3 \pi^2} \sqrt{2 x} \log \frac{2 \sqrt{e} x}{c_2} 
\end{aligned}\end{equation}
for $c_1 = 1 + |\eta'|_1/(x/D)$.
The same bound holds if $|\delta|\geq 1/2c_2$ but $D\leq Q_0/2$.

In general, if $|\delta|\geq 1/2c_2$, the absolute value of (\ref{eq:bovary}) is at most
\begin{equation}\label{eq:zygota}
\begin{aligned}
&\frac{2 \sqrt{c_0 c_1}}{\pi} 
D \log \frac{e x}{D} +\\
&\frac{2 \sqrt{c_0 c_1}}{\pi} 
(1+\epsilon) \left(\frac{x}{|\delta| q}+1\right) \left(\sqrt{3+2\epsilon} 
\cdot \log^+ 2 \sqrt{e} |\delta| q
 + \frac{1}{2} \log^+ \frac{2 D}{\frac{x}{|\delta| q}} \log^+ 2 |\delta| q
\right)\\
&+ \left(
\frac{3 c_1}{4} \left(\frac{2}{\sqrt{5}} + \frac{1+\epsilon}{2\epsilon}
\log x\right) +
\frac{40}{3} \sqrt{2} c_0 c_2^{3/2} \right) \sqrt{x} \log x
\end{aligned}\end{equation}
for $\epsilon\in (0,1\rbrack$.
\end{lem}
\begin{proof}
Define $Q$, $Q'$, $M$, $a'$ and $q'$  as in the proof of Lemma
\ref{lem:bosta1}. The same method of proof works 
as for Lemma \ref{lem:bosta1}; we go over the differences.
When applying Poisson summation or (\ref{eq:ra}), use $\eta_{(x/m)}(t) = 
(\log xt/m) \eta(t)$ instead of $\eta(t)$. Then use 
the bounds in (\ref{eq:puella}) with $\rho = x/m$; in particular,
 \[|\widehat{\eta_{(x/m)}''}|_\infty \leq c_0 \log \frac{x}{m}.\]
For $f(u) = e((\delta m/2 x) u) (\log u) \eta((m/x) u)$,
\[\widehat{f}(t) = \frac{x}{m} \widehat{\eta_{(x/m)}}\left(\frac{x}{m} t - 
\frac{\delta}{2}\right)\]
and so
\[\begin{aligned}
\frac{1}{2} &\sum_n \left|\widehat{f}(n/2)\right| \leq
\frac{x}{m} \left(\frac{1}{2}
\left|\widehat{\eta_{(x/m)}}\left(-\frac{\delta}{2}\right) \right|+
\frac{1}{2} \sum_{n\ne 0} \left|\widehat{\eta}\left(\frac{x}{m} \frac{n}{2}
- \frac{\delta}{2}\right) \right|\right)\\
&= \frac{1}{2} \frac{x}{m}
\left(\widehat{\log \cdot \eta}\left(-\frac{\delta}{2}\right) + 
\log\left(\frac{x}{m}\right) \widehat{\eta}\left(-\frac{\delta}{2}\right)
 \right) + \frac{m}{x} \left(\log \frac{x}{m}\right) \frac{c_0}{2\pi^2} 
(\pi^2 - 4) .\end{aligned}\]

The part of
the main term involving $\log(x/m)$ becomes
\[
\frac{x \widehat{\eta}(-\delta)}{2} 
\mathop{\mathop{\sum_{m\leq M}}_{\text{$m$ odd}}}_{q|m}
\frac{\mu(m)}{m} \log \left(\frac{x}{m }\right)
= \frac{x \mu(q)}{q} 
 \widehat{\eta}(-\delta) \cdot
\mathop{\sum_{m\leq M/q}}_{\gcd(m,2q)= 1} \frac{\mu(m)}{m} 
\log\left(\frac{x}{m q}\right) 
\]
for $q$ odd. (We can see that this, like the rest of the main term,
vanishes for $m$ even.)

In the term in front of $\pi^2-4$, we find the sum
\[\mathop{\mathop{\sum_{m\leq M}}_{\text{$m$ odd}}}_{q|m} 
\frac{m}{x} \log\left(\frac{x}{m}\right)
\leq \frac{M}{x} \log \frac{x}{M} + \frac{q}{2}  
\int_0^{M/q} t \log \frac{x/q}{t} dt = 
\frac{M}{x} \log \frac{x}{M} + \frac{M^2}{4 q x} \log \frac{e^{1/2} x}{M},\]
where we use the fact that $t\mapsto t \log(x/t)$ is increasing for
$t\leq x/e$. By the same fact (and by $M\leq D$), 
$(M^2/q) \log(e^{1/2} x/M)\leq (D^2/q) \log(e^{1/2} x/D)$. It is also easy
to see that $(M/x) \log(x/M) \leq 1/e$ (since $M\leq D \leq x$).

The basic estimate for the rest of the proof (replacing (\ref{eq:aoro}))
is
\[\begin{aligned}
&T_{m,\circ}(\alpha) =
\sum_{\text{$n$ odd}} e(\alpha m n) (\log n) \eta\left(\frac{m n}{x}\right) =
\sum_{\text{$n$ odd}} e(\alpha m n) \eta_{(x/m)} \left(\frac{m n}{x}\right)\\
&= O^*\left( \min\left(\frac{x}{2m} |\eta_{(x/m)}|_1 + \frac{|\eta_{(x/m)}'|_1}{2},
\frac{\frac{1}{2} |\eta_{(x/m)}'|_1 }{|\sin(2 \pi m \alpha)|},
\frac{m}{x} 
\frac{\frac{1}{2} |\widehat{\eta_{(x/m)}''}|_\infty}{(\sin 2 \pi m \alpha)^2}\right)\right)\\
&= O^*\left(\log \frac{x}{m} \cdot
\min\left(\frac{x}{2 m}  + \frac{|\eta'|_1}{2},
\frac{\frac{1}{2} |\eta'|_1 }{|\sin(2 \pi m \alpha)|},
\frac{m}{x} \frac{c_0}{2}
\frac{1}{(\sin 2 \pi m \alpha)^2}\right)
\right).
\end{aligned}\]

We wish to bound
\begin{equation}\label{eq:esthel2}
\mathop{\mathop{\sum_{m\leq M}}_{q\nmid m}}_{\text{$m$ odd}} 
|T_{m,\circ}(\alpha)| + \sum_{\frac{Q}{2} < m \leq D} |T_{m,\circ}(\alpha)|.
\end{equation}

Just as in the proofs of Lemmas \ref{lem:bosta1} and \ref{lem:bosta2}, 
we give two bounds,
one valid for $|\delta|$ large ($|\delta| \geq
1/2 c_2$) and the other for $\delta$ small ($|\delta|\leq 1/2 c_2$).
Again as in the proof of Lemma \ref{lem:bosta2}, we ignore
the condition that $m$ is odd in (\ref{eq:esthel}).

Consider the case of
$|\delta|$ large first. Instead of (\ref{eq:jenuf}), we have
\begin{equation}
\mathop{\sum_{1\leq m\leq M}}_{q\nmid m} |T_m(\alpha)|
\leq \frac{40}{3 \pi^2} \frac{c_0 q^3}{2 x}
\sum_{0\leq j\leq \frac{M}{q}} (j+1) \log \frac{x}{j q+1}.
\end{equation}
Since
\[\begin{aligned}
\sum_{0\leq j\leq \frac{M}{q}} &(j+1) \log \frac{x}{j q+1}
\leq \log x + 
\frac{M}{q} \log \frac{x}{M} + \sum_{1\leq j\leq \frac{M}{q}} \log 
\frac{x}{j q} + \sum_{1\leq j \leq \frac{M}{q}-1} j \log \frac{x}{j q}\\
&\leq  \log x +  \frac{M}{q} \log \frac{x}{M} +
\int_0^{\frac{M}{q}} \log \frac{x}{t q} dt + \int_1^{\frac{M}{q}}
t \log \frac{x}{t q} dt\\
&\leq  \log x +  \left(\frac{2 M}{q} +
\frac{M^2}{2 q^2}\right)
 \log \frac{e^{1/2} x}{M},
\end{aligned}\]
this means that
\begin{equation}\label{eq:luiw}\begin{aligned}
\mathop{\sum_{1\leq m\leq M}}_{q\nmid m} |T_m(\alpha)|
&\leq  \frac{40}{3 \pi^2} \frac{c_0 q^3}{4 x} \left(\log x +  \left(\frac{2 M}{q} +
\frac{M^2}{2 q^2}\right)
 \log \frac{e^{1/2} x}{M}\right)\\
&\leq \frac{5 c_0 c_2}{3 \pi^2} M \log \frac{\sqrt{e} x}{M}
+ \frac{40}{3} \sqrt{2} c_0 c_2^{3/2} \sqrt{x} \log x,
\end{aligned}\end{equation}
where we are using the bounds $M\leq Q/2\leq c_2 x/q$ and
$q^2\leq 2c_2 x$ (just as in (\ref{eq:jenuf})).
Instead of (\ref{eq:tenda}), we have
\[\begin{aligned}
\sum_{j=0}^{\left\lfloor \frac{D- (Q+1)/2}{q'}\right\rfloor}
\left(\log \frac{x}{j q' + \frac{Q+1}{2}}\right) 
\frac{x}{j q' + \frac{Q+1}{2}}
&\leq \frac{x}{Q/2} \log \frac{2 x}{Q} + \frac{x}{q'}
\int_{\frac{Q+1}{2}}^D \log \frac{x}{t} \frac{dt}{t}\\
&\leq \frac{2 x}{Q} \log \frac{2 x}{Q} + 
\frac{x}{q'} \log \frac{2x}{Q} \log^+ \frac{2D}{Q};
\end{aligned}\]
recall that the coefficient in front of this sum will be halved by
the condition that $n$ is odd.
 Instead of (\ref{eq:beatri}), we obtain
\[\begin{aligned}
&q' \sum_{j=0}^{\lfloor \frac{D-(Q+1)/2}{q'}\rfloor}
\sqrt{1 + \frac{q'}{jq'+ (Q+1)/2}}
\left(\log \frac{x}{j q' + \frac{Q+1}{2}}\right) \\
 &\leq q' \sqrt{3+2\epsilon} \cdot \log \frac{2 x}{Q+1} + 
\int_{\frac{Q+1}{2}}^D \left(1 + \frac{q'}{2 t}\right) 
\left(\log \frac{x}{t}\right) dt \\ &\leq q' \sqrt{3 + 2\epsilon} \cdot 
\log \frac{2 x}{Q+1} + D \log \frac{e x}{D} - \frac{Q+1}{2} \log
\frac{2 e x}{Q+1} 
+ \frac{q'}{2} \log \frac{2 x}{Q+1} \log \frac{2 D}{Q+1}.
\end{aligned}\]
(The bound $\int_a^b \log(x/t) dt/t \leq \log(x/a) \log(b/a)$
will be more practical than the exact expression for the integral.)
Hence $\sum_{Q/2<m\leq D} |T_m(\alpha)|$ is at most
\[\begin{aligned}
&\frac{2 \sqrt{c_0 c_1}}{\pi} D \log \frac{e x}{D} \\
&+\frac{2 \sqrt{c_0 c_1}}{\pi} \left(
(1+\epsilon) \sqrt{3 + 2\epsilon}  +
\frac{(1+\epsilon)}{2} \log \frac{2 D}{Q+1}
\right) (Q+1) \log \frac{2 x}{Q+1}\\
&- \frac{2 \sqrt{c_0 c_1}}{\pi} \cdot \frac{Q+1}{2} \log
\frac{2 e x}{Q+1}
+ \frac{3 c_1}{2} \left(\frac{2}{\sqrt{5}} + \frac{1+\epsilon}{\epsilon}
\log^+ \frac{D}{Q/2} \right) \sqrt{x} \log \sqrt{x}.
\end{aligned}\]
Summing this to (\ref{eq:luiw}) (with $M = Q/2$), and using
(\ref{eq:semin1}) and (\ref{eq:semin2}) as before,
 we obtain that (\ref{eq:esthel2}) is at most
\[\begin{aligned}
&\frac{2 \sqrt{c_0 c_1}}{\pi} 
D \log \frac{e x}{D} \\ &+ \frac{2 \sqrt{c_0 c_1}}{\pi} 
(1+\epsilon)  (Q+1) 
\left(\sqrt{3+2\epsilon} 
\log^+ \frac{2 \sqrt{e} x}{Q+1}
 + \frac{1}{2} \log^+ \frac{2 D}{Q+1} \log^+ \frac{2 x}{Q+1}
\right)\\
&+ \frac{3 c_1}{2}\left(\frac{2}{\sqrt{5}} + \frac{1+\epsilon}{\epsilon}
\log^+ \frac{D}{Q/2} \right) \sqrt{x} \log \sqrt{x}
+ \frac{40}{3} \sqrt{2} c_0 c_2^{3/2} \sqrt{x} \log x .
\end{aligned}\]

Now we go over the case of $|\delta|$ small
(or $D\leq Q_0/2$).
Instead of (\ref{eq:prokof}),
we have
\begin{equation}\label{eq:bobo}
\sum_{m\leq q/2} |T_{m,\circ}(\alpha)| \leq 
\frac{2 |\eta'|_1}{\pi} q
\max\left(1, \log \frac{c_0 e^3 q^2}{4 \pi |\eta'|_1 x}\right) 
\log x.\end{equation}
Suppose $q^2<2 c_2 x$. 
Instead of (\ref{eq:martinu}), we have
\begin{equation}\label{eq:bocio}\begin{aligned}
\mathop{\sum_{\frac{q}{2} < m\leq D'}}_{q\nmid m} &|T_{m,\circ}(\alpha)|
\leq \frac{40}{3 \pi^2} \frac{c_0 q^3}{6 x} \sum_{1\leq j\leq 
\frac{D'}{q} + \frac{1}{2}} \left(j + \frac{1}{2}\right)
\log \frac{x}{\left(j - \frac{1}{2}\right) q}\\
&\leq \frac{10 c_0 q^3}{3 \pi^2 x} \left(\log \frac{2 x}{q} + 
\frac{1}{q} \int_0^{D'} \log \frac{x}{t} dt + 
\frac{1}{q} \int_0^{D'} t \log \frac{x}{t} dt 
+ \frac{D'}{q} \log \frac{x}{D'}\right)\\
&=
\frac{10 c_0 q^3}{3 \pi^2 x} \left(\log \frac{2 x}{q} + 
\left(\frac{2 D'}{q} + \frac{(D')^2}{2 q^2}\right) 
\log \frac{\sqrt{e} x}{D'}\right)\\
&\leq \frac{5 c_0 c_2}{3 \pi^2} \left(4 \sqrt{2 c_2 x} \log \frac{2 x}{q} +
4 \sqrt{2 c_2 x} \log \frac{\sqrt{e} x}{D'} + D' \log \frac{\sqrt{e} x}{D'}
\right) \\ &\leq
\frac{5 c_0 c_2}{3 \pi^2} \left(D' \log \frac{\sqrt{e} x}{D'} +
4 \sqrt{2 c_2 x} \log \frac{2 \sqrt{e}}{c_2} x
\right)
\end{aligned}\end{equation}
where $D' = \min(c_2 x/q,D)$. (We are using the bounds
$q^3/x \leq (2 c_2)^{3/2}$, $D' q^2 / x \leq c_2 q < c_2^{3/2} \sqrt{2 x}$
and $D' q/x \leq c_2$.)
Instead of (\ref{eq:caron}), we have
\[\begin{aligned}\sum_{R< m \leq D} &|T_{m,\circ}(\alpha)| \leq
\sum_{j=0}^{\left\lfloor \frac{1}{q} \left(D - R\right)
\right\rfloor} \left(\frac{\frac{3 c_1}{2} x}{j q + R} + 
\frac{4q}{\pi} \sqrt{\frac{c_1 c_0}{4}
\left(1 + \frac{q}{j q + R}\right)}\right) \log \frac{x}{jq + R},
\end{aligned}\]
where $R = \max(c_2 x/q,q/2)$. We can simply reuse (\ref{eq:kosto}),
multiplying it by $\log x/R$; we replace (\ref{eq:kostas}) by
\[\begin{aligned}
q &\sum_{j=0}^{\left\lfloor \frac{1}{q} \left(D - R\right)\right\rfloor}
\sqrt{1 + \frac{q}{jq+R}} \log \frac{x}{jq + R}
\leq q \sqrt{1 + \frac{q}{R}} \log \frac{x}{R} 
+  \int_R^D \sqrt{1 + \frac{q}{t}}
\log \frac{x}{t} dt\\
&\leq \sqrt{3} q \log \frac{q}{c_2}
+ \left(D \log \frac{e x}{D} - R \log \frac{e x}{R}\right)
 + \frac{q}{2} \log \frac{q}{c_2} \log^+ 
\frac{D}{R} .
\end{aligned}\]
We sum with (\ref{eq:bobo}) and (\ref{eq:bocio}), and obtain
(\ref{eq:zygota}) as an upper bound for
(\ref{eq:esthel2}).
\end{proof}

We will apply the following only for $q$ relatively large.
\begin{lem}\label{lem:bogus}
Let $\alpha\in \mathbb{R}/\mathbb{Z}$ with
$2 \alpha = a/q + \delta/x$, $\gcd(a,q)=1$, $|\delta/x|\leq 1/q Q_0$,
$q\leq Q_0$, $Q_0\geq \max(2e,2 \sqrt{x})$. 
 Let $\eta$ be continuous, piecewise $C^2$ and compactly supported, with
$|\eta|_1 = 1$ and $\eta''\in L_1$. Let $c_0 \geq |\widehat{\eta''}|_\infty$.
Let $c_2 = 6 \pi/5 \sqrt{c_0}$. Assume that $x \geq e^2 c_2/2$.

Let $U,V\geq 1$ satisfy $UV + (19/18) Q_0\leq x/5.6$. 
Then, if $|\delta| \leq 1/2c_2$,
 the absolute value of 
\begin{equation}\label{eq:gargam}
\left|\mathop{\sum_{v\leq V}}_{\text{$v$ odd}} \Lambda(v) 
\mathop{\sum_{u\leq U}}_{\text{$u$ odd}} \mu(u) \mathop{\sum_n}_{\text{$n$ odd}}
 e(\alpha v u n)  \eta(v u n/x)\right|\end{equation}
is at most
\begin{equation}\label{eq:cupcake3}\begin{aligned}
&\frac{x}{2 q} \min\left(1, \frac{c_0}{(\pi \delta)^2}\right) \log V q \\
&+ 
O^*\left(\frac{1}{4} - \frac{1}{\pi^2}\right) \cdot
c_0 \left(\frac{D^2 \log V}{2 q x} + \frac{3 c_4}{2} \frac{U V^2}{x} 
+ \frac{(U+1)^2 V}{2 x} \log q\right)
\end{aligned}\end{equation}
plus
\begin{equation}\label{eq:piececake}\begin{aligned}
&\frac{2 \sqrt{c_0 c_1}}{\pi} \left(
D \log \frac{D}{\sqrt{e}} +
q \left(\sqrt{3} \log \frac{c_2 x}{q} +
\frac{\log D}{2} 
\log^+ \frac{D}{q/2} \right)\right)\\ 
&+ \frac{3 c_1}{2} \frac{x}{q} \log D \log^+ \frac{D}{c_2 x/q}
+ 
\frac{2 |\eta'|_1}{\pi} q
\max\left(1, \log \frac{c_0 e^3 q^2}{4 \pi |\eta'|_1 x}\right) \log \frac{q}{2}
\\
&+ \frac{3 c_1}{2 \sqrt{2 c_2}} \sqrt{x} \log \frac{c_2 x}{2} + 
\frac{25 c_0}{4 \pi^2} (2 c_2)^{3/2} \sqrt{x} \log x ,
\end{aligned} \end{equation}
where $D = UV$ and $c_1 = 1 + |\eta'|_1/(2x/D)$ and $c_4 = 1.03884$.
The same bound holds if $|\delta| \geq 1/2 c_2$ but
$D\leq Q_0/2$.

In general,
 if $|\delta| \geq 1/2c_2$, the absolute value of (\ref{eq:gargam}) is at most
(\ref{eq:cupcake3}) plus
\begin{equation}\label{eq:tvorog}\begin{aligned}
&\frac{2 \sqrt{c_0 c_1}}{\pi} D \log \frac{D}{e}\\ 
&+ \frac{2 \sqrt{c_0 c_1}}{\pi} (1+\epsilon) \left(\frac{x}{|\delta| q}+1
\right) \left(
(\sqrt{3+2\epsilon}-1) \log \frac{\frac{x}{|\delta| q}+1}{\sqrt{2}}
 + \frac{1}{2} \log D \log^+ \frac{e^2 D}{\frac{x}{|\delta|q}}\right)\\
&+
\left(\frac{3 c_1}{2} \left( \frac{1}{2}  + 
\frac{3(1+\epsilon)}{16 \epsilon} \log x\right) 
+ \frac{20 c_0}{3 \pi^2} (2 c_2)^{3/2}
\right) \sqrt{x} \log x
\end{aligned}\end{equation}
for $\epsilon\in (0,1\rbrack$.
\end{lem}
\begin{proof}
We proceed essentially as in Lemma \ref{lem:bosta1} and Lemma \ref{lem:bosta2}.
Let $Q$, $q'$ and $Q'$ be as in the proof of Lemma \ref{lem:bosta2}, that
is, with $2 \alpha$ where Lemma \ref{lem:bosta1} uses $\alpha$.


 Let $M = \min(UV, Q/2)$. We first consider the terms with $uv\leq M$, 
$u$ and $v$ odd, $uv$ divisible by $q$. If $q$ is even, there are no such terms.
Assume $q$ is odd. Then, by (\ref{eq:kormo}) and (\ref{eq:karma}),
the absolute value of the contribution of these terms is at most 
\begin{equation}\label{eq:hoho}
\mathop{\mathop{\sum_{a\leq M}}_{\text{$a$ odd}}}_{q|a} 
\left(\mathop{\sum_{v|a}}_{a/U\leq v\leq V}
\Lambda(v) \mu(a/v)\right)
\left(
\frac{x \widehat{\eta}(-\delta/2)}{2a} +
O\left(\frac{a}{x} \frac{|\widehat{\eta''}|_\infty}{2 \pi^2} \cdot
(\pi^2 - 4)\right)\right).
\end{equation}

Now
\[\begin{aligned}
\mathop{\mathop{\sum_{a\leq M}}_{\text{$a$ odd}}}_{q|a} \;
&\mathop{\sum_{v|a}}_{a/U\leq v\leq V} \frac{\Lambda(v) \mu(a/v)}{a}\\
&= \mathop{\mathop{\sum_{v\leq V}}_{\text{$v$ odd}}}_{\gcd(v,q)=1} 
 \frac{\Lambda(v)}{v} 
\mathop{\mathop{\sum_{u\leq \min(U,M/V)}}_{\text{$u$ odd}}}_{q|u}
\frac{\mu(u)}{u} + 
\mathop{\mathop{\sum_{p^\alpha\leq V}}_{\text{$p$ odd}}}_{p|q}
 \frac{\Lambda(p^{\alpha})}{p^{\alpha}}
\mathop{\mathop{\sum_{u\leq \min(U,M/V)}}_{\text{$u$ odd}}}_{\frac{q}{\gcd(q,p^{\alpha})}|u} \frac{\mu(u)}{u} \\
&= \frac{\mu(q)}{q}
\mathop{\mathop{\sum_{v\leq V}}_{\text{$v$ odd}}}_{\gcd(v,q)=1} 
 \frac{\Lambda(v)}{v} 
\mathop{\sum_{u\leq \min(U/q,M/Vq)}}_{\gcd(u,2q)=1}
\frac{\mu(u)}{u} \\ &+ \frac{\mu\left(\frac{q}{\gcd(q,p^\alpha)}\right)}{q}
\mathop{\mathop{\sum_{p^\alpha\leq V}}_{\text{$p$ odd}}}_{p|q}
\frac{\Lambda(p^{\alpha})}{p^{\alpha}/\gcd(q,p^{\alpha})}
\mathop{\mathop{\sum_{u\leq \min\left(\frac{U}{q/\gcd(q,p^\alpha)},\frac{M/V}{
q/\gcd(q,p^{\alpha})}\right)}}_{\text{$u$ odd}}}_{
\gcd\left(u,\frac{q}{\gcd(q,p^{\alpha})}\right) = 1} \frac{\mu(u)}{u}\\
&= \frac{1}{q} \cdot O^*\left(\mathop{\sum_{v\leq V}}_{\gcd(v,2q)=1}
\frac{\Lambda(v)}{v} + \mathop{\mathop{\sum_{p^\alpha\leq V}}_{\text{$p$ odd}}}_{p|q}
\frac{\log p}{p^{\alpha}/\gcd(q,p^{\alpha})}\right),
\end{aligned}\]
where we are using (\ref{eq:grara}) to bound the sums on $u$ by $1$.
We notice that
\[\begin{aligned}
\mathop{\mathop{\sum_{p^\alpha\leq V}}_{\text{$p$ odd}}}_{p|q}
\frac{\log p}{p^{\alpha}/\gcd(q,p^{\alpha})} &\leq
\mathop{\sum_{\text{$p$ odd}}}_{p|q} (\log p) \left(v_p(q) + 
\mathop{\sum_{\alpha>v_p(q)}}_{p^\alpha\leq V} \frac{1}{p^{\alpha-v_p(q)}}\right)\\
&\leq \log q + 
\mathop{\sum_{\text{$p$ odd}}}_{p|q} (\log p) 
\mathop{\sum_{\beta>0}}_{p^\beta \leq \frac{V}{p^{v_p(q)}}} \frac{\log p}{p^\beta}
\leq \log q + \mathop{\mathop{\sum_{v\leq V}}_{\text{$v$ odd}}}_{\gcd(v,q)=1}
\frac{\Lambda(v)}{v},\end{aligned}\]
and so
\[\begin{aligned}
\mathop{\mathop{\sum_{a\leq M}}_{\text{$a$ odd}}}_{q|a} \;
\mathop{\sum_{v|a}}_{a/U\leq v\leq V} \frac{\Lambda(v) \mu(a/v)}{a} &=
\frac{1}{q} \cdot O^*\left(\log q + 
\mathop{\sum_{v\leq V}}_{\gcd(v,2)=1} \frac{\Lambda(v)}{v}\right)\\
&= \frac{1}{q} \cdot O^*(\log q + \log V)\end{aligned}\]
by (\ref{eq:rala}). The absolute value of the sum of the terms with
$\widehat{\eta}(-\delta/2)$ in (\ref{eq:hoho}) is thus at most
\[\frac{x}{q} \frac{\widehat{\eta}(-\delta/2)}{2} (\log q + \log V)
\leq
\frac{x}{2 q} \min\left(1, \frac{c_0}{(\pi \delta)^2}\right) \log Vq,\]
where we are bounding $\widehat{\eta}(-\delta/2)$ by
(\ref{eq:malcros}).

The other terms in (\ref{eq:hoho}) contribute at most
\begin{equation}\label{eq:billy}
(\pi^2 - 4) \frac{|\widehat{\eta''}|_\infty}{2 \pi^2} \frac{1}{x} 
\mathop{\mathop{\mathop{\sum_{u\leq U} \sum_{v\leq V}}_{\text{$uv$ odd}}}_{
uv \leq M,\; q|uv}}_{\text{$u$ sq-free}}
\Lambda(v) u v.\end{equation}

For any $R$, $\sum_{u\leq R, \text{$u$ odd}, q|u} \leq R^2/4q + 3 R/4$.
Using the estimates (\ref{eq:rala}), (\ref{eq:trado2}) and (\ref{eq:chronop}),
we obtain that the double sum in (\ref{eq:billy}) is at most
\begin{equation}\label{eq:etoile}
\begin{aligned}
\mathop{\sum_{v\leq V}}_{\gcd(v,2q)=1} &\Lambda(v) v 
\mathop{\mathop{\sum_{u\leq \min(U,M/v)}}_{\text{$u$ odd}}}_{q|u} u +
\mathop{\mathop{\sum_{p^\alpha\leq V}}_{\text{$p$ odd}}}_{p|q} (\log p) p^\alpha 
\mathop{\mathop{\sum_{u\leq U}}_{\text{$u$ odd}}}_{\frac{q}{\gcd(q,p^\alpha)} | u} u\\
&\leq \mathop{\sum_{v\leq V}}_{\gcd(v,2q)=1} \Lambda(v) v \cdot
\left(\frac{(M/v)^2}{4 q} + \frac{3 M}{4 v}\right) +
\mathop{\mathop{\sum_{p^\alpha\leq V}}_{\text{$p$ odd}}}_{p|q} (\log p) p^\alpha \cdot
\frac{(U+1)^2}{4}\\
&\leq \frac{M^2 \log V}{4 q} + \frac{3 c_4}{4} M V 
+ \frac{(U+1)^2}{4} V \log q,
\end{aligned}\end{equation}
where $c_4 = 1.03884$.

From this point onwards, we use the easy bound
\[
\left|\mathop{\sum_{v|a}}_{a/U\leq v\leq V}
\Lambda(v) \mu(a/v)\right| \leq \log a .
\]
What we must bound now is
\begin{equation}\label{eq:esthel3}
\mathop{\mathop{\sum_{m\leq UV}}_{\text{$m$ odd}}}_{\text{$q\nmid m$ or 
$m>M$}} (\log m) \sum_{\text{$n$ odd}} e(\alpha m n) \eta(m n/x).\end{equation}
The inner sum is the same as the sum $T_{m,\circ}(\alpha)$ in (\ref{eq:baxter});
we will be using the bound (\ref{eq:trompais}). Much as before, we will be able
to ignore the condition that $m$ is odd.

Let $D= UV$. What remains to do is similar to what we did
 in the proof of Lemma \ref{lem:bosta1} (or Lemma \ref{lem:bosta2}).

Case (a). {\em $\delta$ large: $|\delta|\geq 1/2c_2$.}
 Instead of (\ref{eq:jenuf}),
we have
\[\mathop{\sum_{1\leq m\leq M}}_{q\nmid m} (\log m) |T_{m,\circ}(\alpha)| \leq
 \frac{40}{3 \pi^2} \frac{c_0 q^3}{4 x} \sum_{0\leq j\leq \frac{M}{q}}
(j+1) \log (j+1)q,\]
and, since $M\leq \min(c_2 x/q,D)$,
$q\leq \sqrt{2 c_2 x}$ (just as in the proof of Lemma \ref{lem:bosta1}) and
\[\begin{aligned}
\sum_{0\leq j\leq \frac{M}{q}}
(j+1) \log (j+1)q &\leq \frac{M}{q} \log M + \left(\frac{M}{q} + 1\right) 
\log (M+1) + \frac{1}{q^2} \int_0^M t \log t\; dt\\
&\leq \left(2 \frac{M}{q} + 1\right) \log x 
 + \frac{M^2}{2 q^2} \log \frac{M}{\sqrt{e}},
\end{aligned}\]
we conclude that 
\begin{equation}\label{eq:cocolo}\begin{aligned}
\mathop{\sum_{1\leq m\leq M}}_{q\nmid m} |T_{m,\circ}(\alpha)| \leq
 \frac{5 c_0 c_2}{3 \pi^2} M \log \frac{M}{\sqrt{e}} +
\frac{20 c_0}{3 \pi^2} (2 c_2)^{3/2} 
\sqrt{x} \log x
.\end{aligned}\end{equation}

Instead of (\ref{eq:tenda}), we have
\[\begin{aligned}
\sum_{j=0}^{\lfloor \frac{D - (Q+1)/2}{q'}\rfloor} \frac{x}{j q' + \frac{Q+1}{2}}
&\log \left(j q' + \frac{Q+1}{2}\right)
\leq \frac{x}{\frac{Q+1}{2}} \log \frac{Q+1}{2}  + \frac{x}{q'} \int_{\frac{Q+1}{2}}^D \frac{\log t}{t} dt\\
&\leq \frac{2 x}{Q} \log \frac{Q}{2} + \frac{(1+\epsilon) x}{2 \epsilon Q} \left((\log D)^2 - 
 \left(\log \frac{Q}{2}\right)^2\right)
.\end{aligned}\]
Instead of (\ref{eq:beatri}), we estimate
\[\begin{aligned}
q' &\sum_{j=0}^{\left\lfloor \frac{D - \frac{Q+1}{2}}{q'}\right\rfloor} 
\left(\log \left(\frac{Q+1}{2} + j q'\right)\right)
\sqrt{1 + \frac{q'}{j q' + \frac{Q+1}{2}}}\\ &\leq
q' \left(\log D + (\sqrt{3+2\epsilon} - 1) \log \frac{Q+1}{2} \right) + 
\int_{\frac{Q+1}{2}}^D \log t\; dt + 
\int_{\frac{Q+1}{2}}^D \frac{q' \log t}{2 t} dt\\
&\leq q' \left(\log D + \left(\sqrt{3+2\epsilon} - 1\right) \log \frac{Q+1}{2} \right) + 
 \left(D \log \frac{D}{e} - 
\frac{Q+1}{2} \log \frac{Q+1}{2 e}\right) \\ &+ 
\frac{q'}{2} \log D \log^+ \frac{D}{\frac{Q+1}{2}} .\end{aligned}\]
We conclude that, when $D\geq Q/2$, the sum 
 $\sum_{Q/2 < m\leq D} (\log m) |T_m(\alpha)|$ is at most
\[\begin{aligned}
&\frac{2 \sqrt{c_0 c_1}}{\pi} \left(D \log \frac{D}{e}
+ (Q+1)
\left((1+\epsilon) (\sqrt{3 + 2 \epsilon} - 1) \log \frac{Q+1}{2} -
\frac{1}{2} \log \frac{Q+1}{2 e}\right)\right)\\
&+ \frac{\sqrt{c_0 c_1}}{\pi} (Q+1) 
(1+\epsilon) \log D \log^+ \frac{e^2 D}{\frac{Q+1}{2}} \\
&+
\frac{3 c_1}{2} \left( \frac{2 x}{Q} \log \frac{Q}{2} + \frac{(1+\epsilon) x}{2 \epsilon Q} \left((\log D)^2 - 
 \left(\log \frac{Q}{2}\right)^2\right)\right).
\end{aligned}\]
We must now add this to (\ref{eq:cocolo}). Since
\[
 (1+\epsilon) (\sqrt{3 + 2\epsilon} - 1) \log \sqrt{2}
- \frac{1}{2} \log 2 e + \frac{1 + \sqrt{13/3}}{2} \log 2 \sqrt{e}
> 0\]
and $Q\geq 2 \sqrt{x}$,
 we conclude that (\ref{eq:esthel3}) is at most
\begin{equation}\label{eq:iulia}\begin{aligned}
&\frac{2 \sqrt{c_0 c_1}}{\pi} D \log \frac{D}{e}\\ 
&+ \frac{2 \sqrt{c_0 c_1}}{\pi} (1+\epsilon) (Q+1) \left(
(\sqrt{3+2\epsilon}-1) \log \frac{Q+1}{\sqrt{2}}
 + \frac{1}{2} \log D \log^+ \frac{e^2 D}{\frac{Q+1}{2}}\right)\\
&+
\left(\frac{3 c_1}{2} \left( \frac{1}{2}  + 
\frac{3(1+\epsilon)}{16 \epsilon} \log x\right) 
+ \frac{20 c_0}{3 \pi^2} (2 c_2)^{3/2}
\right) \sqrt{x} \log x.
\end{aligned}\end{equation}

{\em Case (b). $\delta$ small: $|\delta|\leq 1/2 c_2$
or $D\leq Q_0/2$.}
The analogue of (\ref{eq:prokof}) is a bound of
\[\leq \frac{2 |\eta'|_1}{\pi} q
\max\left(1, \log \frac{c_0 e^3 q^2}{4 \pi |\eta'|_1 x}\right) \log \frac{q}{2}
\] 
for the terms with $m\leq q/2$. 
If $q^2 < 2 c_2 x$, then, much
as in (\ref{eq:martinu}), we have
\begin{equation}\label{eq:jotoy}\begin{aligned}
\mathop{\sum_{\frac{q}{2} < m \leq D'}}_{q\nmid m} |T_{m,\circ}(\alpha)|
(\log m)
&\leq \frac{10}{\pi^2} \frac{c_0 q^3}{3 x} \sum_{1\leq j\leq \frac{D'}{q} + 
\frac{1}{2}} \left(j + \frac{1}{2}\right) \log (j+1/2) q\\
&\leq \frac{10}{\pi^2} \frac{c_0 q}{3 x} \int_q^{D' + \frac{3}{2} q}
x \log x \; dx.\end{aligned}\end{equation}
Since
\[\begin{aligned}
&\int_q^{D' + \frac{3}{2} q} x \log x \; dx = 
\frac{1}{2} \left(D' +
\frac{3}{2} q \right)^2 \log \frac{D'+\frac{3}{2} q}{\sqrt{e}} 
- \frac{1}{2} q^2 \log \frac{q}{\sqrt{e}}\\
&= \left(\frac{1}{2} D'^2 + \frac{3}{2} D' q\right)
\left(\log \frac{D'}{\sqrt{e}} + \frac{3}{2} \frac{q}{D'}\right) + 
\frac{9}{8} q^2 \log \frac{D' + \frac{3}{2} q}{\sqrt{e}}
 - \frac{1}{2} q^2 \log \frac{q}{\sqrt{e}}\\
&= \frac{1}{2} D'^2 \log \frac{D'}{\sqrt{e}} + \frac{3}{2} D' q \log D' +
\frac{9}{8} q^2 \left(\frac{2}{9} + \frac{3}{2} + 
\log \left(D' + \frac{19}{18} q\right)\right),
\end{aligned}\]
where $D' = \min(c_2 x/q,D)$, and since the assumption
$(UV + (19/18) Q_0)\leq x/5.6$ implies that $(2/9 + 3/2 + 
\log(D' + (19/18) q)) \leq x$, 
 we conclude that
\begin{equation}\label{eq:prok}\begin{aligned}
&\mathop{\sum_{\frac{q}{2} < m \leq D'}}_{q\nmid m} |T_{m,\circ}(\alpha)| (\log m)
\\ &\leq \frac{5 c_0 c_2}{3 \pi^2} D' \log \frac{D'}{\sqrt{e}} +
\frac{10 c_0}{3 \pi^2} \left( \frac{3}{4} (2 c_2)^{3/2} \sqrt{x} \log x
+ \frac{9}{8} (2 c_2)^{3/2} \sqrt{x} \log x\right)\\
&\leq \frac{5 c_0 c_2}{3 \pi^2} D' \log \frac{D'}{\sqrt{e}} +
\frac{25 c_0}{4 \pi^2} (2 c_2)^{3/2} \sqrt{x} \log x 
.\end{aligned}\end{equation}
Let $R = \max(c_2 x/q,q/2)$.
We bound the terms $R<m\leq D$ as in (\ref{eq:caron}), with a factor of
$\log (j q + R)$ inside the sum. The analogues of (\ref{eq:kosto}) and
(\ref{eq:kostas}) are
\begin{equation}\label{eq:gator1}\begin{aligned}
\sum_{j=0}^{\left\lfloor \frac{1}{q} (D- R)\right\rfloor} &\frac{x}{jq + R} \log(jq + R) \leq
\frac{x}{R} \log R + \frac{x}{q} \int_R^D \frac{\log t}{t} dt\\
&\leq \sqrt{ \frac{ 2 x}{c_2}} \log \sqrt{\frac{c_2 x}{2}} + \frac{x}{q}
\log D \log^+ \frac{D}{R}, 
\end{aligned}\end{equation} where we use the assumption that
$x\geq e^2 c/2$, and
\begin{equation}\label{eq:gator2}\begin{aligned}
\sum_{j=0}^{\left\lfloor \frac{1}{q} (D- R)\right\rfloor} &\log(jq + R) 
\sqrt{1 + \frac{q}{jq + R}}
\leq \sqrt{3} \log R\\
&+ \frac{1}{q} \left(D \log \frac{D}{e} - R \log \frac{R}{e}\right)
+ \frac{1}{2} \log D \log \frac{D}{R} 
\end{aligned}\end{equation}
(or $0$ if $D<R$).
We sum with (\ref{eq:prok}) and the terms with $m\leq q/2$, and 
obtain, for $D' = c_2 x/q = R$,
\[\begin{aligned}
&\frac{2 \sqrt{c_0 c_1}}{\pi} \left(
D \log \frac{D}{\sqrt{e}} +
q \left(\sqrt{3} \log \frac{c_2 x}{q} +
\frac{\log D}{2} 
\log^+ \frac{D}{q/2} \right)\right)\\ 
&+ \frac{3 c_1}{2} \frac{x}{q} \log D \log^+ \frac{D}{c_2 x/q}
+ 
\frac{2 |\eta'|_1}{\pi} q
\max\left(1, \log \frac{c_0 e^3 q^2}{4 \pi |\eta'|_1 x}\right) \log \frac{q}{2}
\\
&+ \frac{3 c_1}{2 \sqrt{2 c_2}} \sqrt{x} \log \frac{c_2 x}{2} + 
\frac{25 c_0}{4 \pi^2} (2 c_2)^{3/2} \sqrt{x} \log x 
,\end{aligned}\]
which, it is easy to check, is also valid even if $D'=D$ (in which case
(\ref{eq:gator1}) and (\ref{eq:gator2}) do not appear) or $R=q/2$ (in which
case (\ref{eq:prok}) does not appear). 
\end{proof}

\section{Type II}\label{sec:typeII}
We must now consider the sum
\begin{equation}\label{eq:adoucit}
S_{II} = \mathop{\sum_{m>U}}_{\gcd(m,v)=1}
 \left(\mathop{\sum_{d>U}}_{d|m} \mu(d)\right) \mathop{\sum_{n>V}}_{\gcd(n,v)=1}
 \Lambda(n)
 e(\alpha m n) \eta(m n/x).\end{equation}

Here the main improvements over classical treatments are as follows:
\begin{enumerate}
\item\label{it:bbc}
 obtaining cancellation in the term 
\[
\mathop{\sum_{d>U}}_{d|m} \mu(d)
\] leading to
a gain of a factor of $\log$; 
\item using a large sieve for primes, getting
rid of a further $\log$;
\item exploiting, via a non-conventional application of the principle of the
 large sieve (Lemma \ref{lem:ogor}),
the fact that $\alpha$ is in the tail of an interval (when that is the case).
\end{enumerate}
Some of the techniques developed for (\ref{it:bbc}) should be applicable 
to other instances of Vaughan's identity in the literature.

It is technically helpful to express $\eta$ as the 
(multiplicative) convolution of two functions of compact support --
preferrably the same function:
\begin{equation}\label{eq:conque}
\eta(x) = \int_0^\infty \eta_1(t) \eta_1(x/t) \frac{dt}{t}.\end{equation}
For the smoothing function $\eta(t) = \eta_2(t) = 4 \max(\log 2 - |\log 2 t|, 
0)$, (\ref{eq:conque}) holds with
\begin{equation}\label{eq:tindot}
\eta_1(t) = \begin{cases} 2
 &\text{if $t\in (1/2,1\rbrack$}\\ 0 &\text{otherwise.}\end{cases}
\end{equation}
We will work with $\eta_1(t)$ as in (\ref{eq:tindot}) for
convenience, yet what follows should carry over to other (non-negative)
choices of $\eta_1$. 

By (\ref{eq:conque}), the sum (\ref{eq:adoucit}) equals
\begin{equation}\label{eq:bycaus}\begin{aligned}
4 \int_0^{\infty}
&\mathop{\sum_{m>U}}_{\gcd(m,v)=1}
 \left(\mathop{\sum_{d>U}}_{d|m} \mu(d)\right)
 \mathop{\sum_{n>V}}_{\gcd(n,v)=1} \Lambda(n)
 e(\alpha m n) \eta_1(t) \eta_1\left(\frac{m n/x}{t}\right) \frac{dt}{t}\\
&= 4 \int_V^{x/U} 
\mathop{\sum_{\max\left(\frac{x}{2W},U\right)<m\leq \frac{x}{W}}}_{\gcd(m,v)=1} 
\left(\mathop{\sum_{d>U}}_{d|m} \mu(d)\right)
 \mathop{\sum_{\max\left(V,\frac{W}{2}\right)<n\leq W}}_{\gcd(n,v)=1} \Lambda(n)
 e(\alpha m n)  \frac{dW}{W}
\end{aligned}\end{equation}
by the substitution $t = (m/x) W$.
(We can assume $V \leq W \leq x/U$ because otherwise one of the sums in
(\ref{eq:costo}) is empty.)

We separate $n$ prime and $n$ non-prime.
By Cauchy-Schwarz,
the expression within the integral in (\ref{eq:bycaus})
is then at most $\sqrt{S_1(U,W)\cdot S_2(U,V,W)} +
\sqrt{S_1(U,W) \cdot S_3(W)}$, where
\begin{equation}\label{eq:costo}\begin{aligned}
S_1(U,W) &= \mathop{\sum_{\max\left(\frac{x}{2W},U\right)<m\leq \frac{x}{W}}}_{
\gcd(m,v)=1} 
\left(\mathop{\sum_{d>U}}_{d|m} \mu(d)\right)^2 ,\\
S_2(U,V,W) &= \mathop{\sum_{\max\left(\frac{x}{2W},U\right)<m\leq \frac{x}{W}}}_{
\gcd(m,v)=1} 
\left|
 \mathop{\sum_{\max\left(V,\frac{W}{2}\right)<p\leq W}}_{\gcd(p,v)=1} (\log p)
 e(\alpha m p)\right|^2.\end{aligned}
\end{equation}
and
\begin{equation}\label{eq:negli}\begin{aligned}
S_3(W) &= 
 \mathop{\sum_{\frac{x}{2 W} < m\leq \frac{x}{W}}}_{\gcd(m,v)=1}
\left| \mathop{\sum_{n\leq W}}_{\text{$n$ non-prime}} \Lambda(n)\right|^2\\
&= \mathop{\sum_{\frac{x}{2 W} < m\leq \frac{x}{W}}}_{\gcd(m,v)=1}
 \left(1.42620 W^{1/2}\right)^2
\leq 1.0171 x + 2.0341 W\end{aligned}\end{equation}
(by \cite[Thm. 13]{MR0137689}). We will assume $V\leq w$; thus the condition
$\gcd(p,v)=1$ will be fulfilled automatically and can be removed.

The contribution of $S_3(W)$ will be negligible.
We must bound $S_1(U,W)$ and $S_2(U,V,W)$ from above.

\subsection{The sum $S_1$: cancellation}\label{subs:vaucanc}
We shall bound
\[S_1(U,W) = \mathop{\sum_{\max(U,x/2W)<m\leq x/W}}_{\gcd(m,v)=1}
 \left(\mathop{\sum_{d>U}}_{d|m} \mu(d)\right)^2.\]

There will be
what is perhaps a surprising amount of cancellation: the expression within
the sum will be bounded by a constant on average.

\subsubsection{Reduction to a sum with $\mu$}

We can write
\begin{equation}\label{eq:crusto}\begin{aligned}
\mathop{\sum_{\max(U,x/2W)<m\leq x/W}}_{\gcd(m,v)=1}
 &\left(\mathop{\sum_{d>U}}_{d|m} \mu(d)\right)^2
= \mathop{\sum_{\frac{x}{2W} < m \leq \frac{x}{W}}}_{\gcd(m,v)=1} \sum_{d_1,d_2|m} \mu(d_1>U) \mu(d_2>U)\\
&= \mathop{\mathop{\sum_{r_1<x/WU} \sum_{r_2<x/WU}}_{\gcd(r_1,r_2)=1}}_{\gcd(r_1 r_2,
v)=1} 
\mathop{
\mathop{\mathop{\sum_l}_{\gcd(l,r_1 r_2)=1}}_{r_1 l, r_2 l>U}}_{\gcd(\ell,v)=1} \mu(r_1 l) \mu(r_2 l)
 \mathop{\mathop{\sum_{\frac{x}{2W} < m\leq \frac{x}{W}}}_{r_1 r_2 l|m}}_{\gcd(m,v)=1} 1,\end{aligned}
\end{equation}
where we write $d_1 = r_1 l$, $d_2 = r_2 l$, $l=\gcd(d_1,d_2)$.
(The inequality $r_1<x/WU$ comes from $r_1 r_2 l|m$, $m\leq x/W$, $r_2 l>U$;
$r_2<x/WU$ is proven in the same way.)
Now (\ref{eq:crusto}) equals
\begin{equation}\label{eq:cudo}\mathop{\sum_{s<\frac{x}{WU}}}_{\gcd(s,v)=1}
 \mathop{\mathop{\sum_{r_1< \frac{x}{WUs}} \sum_{r_2<\frac{x}{WUs}}}_{\gcd(r_1,r_2)=1}}_{
\gcd(r_1 r_2,v)=1}
\mu(r_1) \mu(r_2) \mathop{\mathop{\sum_{
\max\left(\frac{U}{\min(r_1,r_2)},\frac{x/W}{2 r_1 r_2 s}
\right) < l \leq \frac{x/W}{r_1 r_2 s}}}_{
\gcd(l,r_1 r_2)=1, (\mu(l))^2 = 1}}_{\gcd(\ell,v)=1} 1,\end{equation}
where we have set $s = m/(r_1 r_2 l)$. 

\begin{lem}\label{lem:monro}
Let $z,y>0$. Then
\begin{equation}\label{eq:srto}
\mathop{\mathop{\sum_{r_1<y} \sum_{r_2<y}}_{\gcd(r_1,r_2)=1}}_{\gcd(r_1 r_2,v)=1}
 \mu(r_1) \mu(r_2) 
 \mathop{
\mathop{\sum_{\min\left(\frac{z/y}{\min(r_1,r_2)}, \frac{z}{2 r_1 r_2}\right)
    < l \leq \frac{z}{r_1 r_2}}}_{\gcd(l,r_1r_2)=1,
(\mu(l))^2=1}}_{\gcd(\ell,v)=1} 1
\end{equation} equals
\begin{equation}\label{eq:muted}\begin{aligned}
\frac{6 z}{\pi^2}  &\frac{v}{\sigma(v)}
\mathop{\mathop{\sum_{r_1<y}\; \sum_{r_2<y}}_{\gcd(r_1,r_2)=1}}_{\gcd(r_1 r_2,v)=1} 
\frac{\mu(r_1) \mu(r_2)}{\sigma(r_1) \sigma(r_2)} \left(1 - \max\left(
\frac{1}{2}, \frac{r_1}{y}, \frac{r_2}{y}\right)\right) \\ &+ 
O^*\left(5.08\; \zeta\left(\frac{3}{2}\right)^2 y \sqrt{z} \cdot
\prod_{p|v} \left(1 + \frac{1}{\sqrt{p}}\right) \left(1 - \frac{1}{p^{3/2}}\right)^2\right).\end{aligned}\end{equation}
If $v=2$, the error term in (\ref{eq:muted}) can be replaced by
\begin{equation}\label{eq:mudo}
O^*\left( 1.27 \zeta\left(\frac{3}{2}\right)^2 y \sqrt{z} \cdot
\left(1 + \frac{1}{\sqrt{2}}\right)
\left(1 - \frac{1}{2^{3/2}}\right)^2\right).   
\end{equation}
\end{lem}
\begin{proof}
By M\"obius inversion, (\ref{eq:srto}) equals
\begin{equation}\label{eq:etze}\begin{aligned}
\mathop{\mathop{\sum_{r_1<y}\; \sum_{r_2<y}}_{\gcd(r_1,r_2)=1}}_{\gcd(r_1 r_2,v)=1}
 \mu(r_1) \mu(r_2) \mathop{\mathop{\sum_{l \leq \frac{z}{r_1 r_2}}}_{l>
\min\left(\frac{z/y}{\min(r_1,r_2)}, \frac{z}{2 r_1 r_2}\right)}}_{\gcd(\ell,v)=1}
  &\mathop{\sum_{d_1|r_1, d_2|r_2}}_{d_1 d_2|l} \mu(d_1) \mu(d_2) \\
&\mathop{\sum_{d_3|v}}_{d_3|l} \mu(d_3)
  \mathop{\sum_{m^2|l}}_{\gcd(m,r_1 r_2 v)=1} \mu(m).
\end{aligned}\end{equation}
We can change the order of summation of $r_i$ and $d_i$ by defining
$s_i=r_i/d_i$, and we can also use the obvious fact that the number
of integers in an interval $(a,b\rbrack$ divisible by $d$ 
is $(b-a)/d + O^*(1)$. Thus (\ref{eq:etze}) equals
\begin{equation}\label{eq:dbamain}\begin{aligned}
&\mathop{\mathop{\sum_{d_1, d_2 <y}}_{\gcd(d_1, d_2)=1}}_{\gcd(d_1 d_2,v)=1}
 \mu(d_1) \mu(d_2)
 \mathop{\mathop{\mathop{\sum_{s_1<y/d_1}}_{s_2<y/d_2}}_{\gcd(d_1 s_1, d_2 s_2)=1}}_{
\gcd(s_1 s_2,v)=1}
\mu(d_1 s_1) \mu(d_2 s_2)\\
&\sum_{d_3|v} \mu(d_3) 
\mathop{\sum_{m\leq \sqrt{\frac{z}{d_1^2 s_1 d_2^2 s_2 d_3}}}}_{\gcd(m,d_1 s_1
  d_2 s_2 v)=1} \frac{\mu(m)}{d_1 d_2 d_3 m^2} \frac{z}{s_1 d_1 s_2 d_2} 
\left(1-\max\left(\frac{1}{2}, \frac{s_1 d_1}{y}, \frac{s_2 d_2}{y}\right)\right)
\end{aligned}\end{equation} 
plus
\begin{equation}\label{eq:dbaerr}
O^*\left(
\mathop{\sum_{d_1, d_2 <y}}_{\gcd(d_1 d_2,v)=1}
\mathop{\mathop{\sum_{s_1<y/d_1}}_{s_2<y/d_2}}_{\gcd(s_1 s_2,v)=1} \sum_{d_3|v}
\mathop{\sum_{m\leq \sqrt{\frac{z}{d_1^2 s_1 d_2^2 s_2 d_3}}}}_{\text{$m$
sq-free}} 1\right)
.\end{equation} 
If we complete the innermost sum in (\ref{eq:dbamain}) by removing the condition
$m\leq \sqrt{z/(d_1^2 s d_2^2 s_2)}$, we obtain 
(reintroducing the
variables $r_i = d_i s_i$)
\begin{equation}\label{eq:johnny}\begin{aligned}z\cdot
\mathop{\mathop{\sum_{r_1, r_2<y}}_{\gcd(r_1,r_2)=1}}_{\gcd(r_1 r_2,v)=1} 
\frac{\mu(r_1) \mu(r_2)}{r_1 r_2} 
&\left(1 - \max\left(
\frac{1}{2}, \frac{r_1}{y}, \frac{r_2}{y}\right)\right) \\
&\mathop{\sum_{d_1|r_1}}_{d_2|r_2} \sum_{d_3|v} 
\mathop{\sum_{m}}_{\gcd(m,r_1 r_2 v)=1} \frac{\mu(d_1) \mu(d_2) \mu(m) \mu(d_3)}{d_1
  d_2 d_3 m^2}\end{aligned}\end{equation} times $z$. Now (\ref{eq:johnny}) equals
\[\begin{aligned}
\mathop{\mathop{\sum_{r_1, r_2<y}}_{\gcd(r_1,r_2)=1}}_{\gcd(r_1 r_2,v)=1} 
&\frac{\mu(r_1) \mu(r_2) z}{r_1 r_2} 
\left(1 - \max\left(
\frac{1}{2}, \frac{r_1}{y}, \frac{r_2}{y}\right)\right) 
\prod_{p|r_1 r_2 v} \left(1 - \frac{1}{p}\right)
\mathop{\prod_{p\nmid r_1 r_2}}_{p\nmid v} \left(1 - \frac{1}{p^2}\right)\\
&= \frac{6 z}{\pi^2} \frac{v}{\sigma(v)} 
\mathop{\mathop{\sum_{r_1, r_2<y}}_{\gcd(r_1,r_2)=1}}_{\gcd(r_1 r_2,v)=1} 
 \frac{\mu(r_1) \mu(r_2)}{\sigma(r_1) \sigma(r_2)} 
\left(1 - \max\left(
\frac{1}{2}, \frac{r_1}{y}, \frac{r_2}{y}\right)\right),
\end{aligned}\]
i.e., the main term in (\ref{eq:muted}). It remains to estimate
the terms used to complete the sum; their total is, by definition,
 given exactly by (\ref{eq:dbamain}) with the inequality
$m\leq \sqrt{z/(d_1^2 s d_2^2 s_2 d_3)}$ changed to
$m>\sqrt{z/(d_1^2 s d_2^2 s_2 d_3)}$. This is a total of size at most
\begin{equation}\label{eq:shalco}
\frac{1}{2}
\mathop{\sum_{d_1, d_2 <y}}_{\gcd(d_1 d_2,v)=1}
 \mathop{\mathop{\sum_{s_1<y/d_1}}_{s_2<y/d_2}}_{\gcd(s_1 s_2,v)=1} \sum_{d_3|v}
\mathop{\sum_{m>\sqrt{\frac{z}{d_1^2 s_1 d_2^2 s_2 d_3}}}}_{\text{$m$ sq-free}}
\frac{1}{d_1 d_2 d_3 m^2} \frac{z}{s_1 d_1 s_2 d_2} .\end{equation}
Adding this to (\ref{eq:dbaerr}), we obtain, as our total error term,
\begin{equation}\label{eq:totor}
\mathop{\sum_{d_1, d_2 <y}}_{\gcd(d_1 d_2,v)=1}
 \mathop{\mathop{\sum_{s_1<y/d_1}}_{s_2<y/d_2}}_{\gcd(s_1 s_2,v)=1} \sum_{d_3|v}
f\left(\sqrt{\frac{z}{d_1^2 s_1 d_2^2 s_2 d_3}}\right),\end{equation}
where \[f(x) := \mathop{\sum_{m\leq x}}_{\text{$m$ sq-free}} 1 +
\frac{1}{2} \mathop{\sum_{m>x}}_{\text{$m$ sq-free}} \frac{x^2}{m^2} .\]
It is easy to see that $f(x)/x$ has a local maximum exactly when $x$
is a square-free (positive) integer. We can hence check that
\[f(x) \leq \frac{1}{2} \left(2 + 2 \left(\frac{\zeta(2)}{\zeta(4)}-1.25\right)
\right) x = 1.26981\dotsc x \]
for all $x\geq 0$ by checking all integers smaller than a constant and using
$\{m: \text{$m$ sq-free}\}\subset \{m: 4\nmid m\}$ and $1.5\cdot (3/4) <
1.26981$ to bound $f$ from below for $x$ larger than a constant.
Therefore, (\ref{eq:totor}) is at most
\[\begin{aligned}
1.27 &\mathop{\sum_{d_1, d_2 <y}}_{\gcd(d_1 d_2,v)=1}
 \mathop{\mathop{\sum_{s_1<y/d_1}}_{s_2<y/d_2}}_{\gcd(s_1 s_2,v)=1} \sum_{d_3|v}
\sqrt{\frac{z}{d_1^2 s_1 d_2^2 s_2 d_3}}\\
&= 1.27 \sqrt{z}
\prod_{p|v} \left(1 + \frac{1}{\sqrt{p}}\right) \cdot 
\left(\mathop{\sum_{d<y}}_{\gcd(d,v)=1} \mathop{\sum_{s<y/d}}_{\gcd(s,v)=1}
\frac{1}{d \sqrt{s}}\right)^2.\end{aligned}\]
We can bound the double sum simply by
\[\mathop{\sum_{d<y}}_{\gcd(d,v)=1} \sum_{s<y/d} \frac{1}{\sqrt{s} d} \leq
2 \sum_{d<y} \frac{\sqrt{y/d}}{d} \leq
2 \sqrt{y} \cdot \zeta\left(\frac{3}{2}\right) \prod_{p|v} \left(1 - \frac{1}{p^{3/2}}\right).\]
Alternatively, if $v=2$, we bound
\[\mathop{\sum_{s<y/d}}_{\gcd(s,v)=1} \frac{1}{\sqrt{s}} = 
\mathop{\sum_{s<y/d}}_{\text{$s$ odd}} \frac{1}{\sqrt{s}} \leq
1 + \frac{1}{2} \int_1^{y/d} \frac{1}{\sqrt{s}} ds = \sqrt{y/d}\]
and thus
\[\mathop{\sum_{d<y}}_{\gcd(d,v)=1} 
\mathop{\sum_{s<y/d}}_{\gcd(s,v)=1} \frac{1}{\sqrt{s} d} \leq
\mathop{\sum_{d<y}}_{\gcd(d,2)=1} \frac{\sqrt{y/d}}{d} \leq \sqrt{y} \left(1 - \frac{1}{2^{3/2}}\right) \zeta\left(\frac{3}{2}\right)
.\]

\end{proof}

Applying Lemma \ref{lem:monro} with $y=S/s$ and $z=x/Ws$, where $S = x/WU$,
we obtain that
(\ref{eq:cudo}) equals
\begin{equation}\label{eq:flatow}\begin{aligned}
&\frac{6 x}{\pi^2 W} \frac{v}{\sigma(v)} \mathop{\sum_{s<S}}_{\gcd(s,v)=1}
 \frac{1}{s}
\mathop{\mathop{\sum_{r_1<S/s} \sum_{r_2<S/s}}_{\gcd(r_1,r_2)=1}}_{\gcd(r_1 r_2,v)=1} 
\frac{\mu(r_1) \mu(r_2)}{\sigma(r_1) \sigma(r_2)} \left(1 - \max\left(
\frac{1}{2}, \frac{r_1}{S/s}, \frac{r_2}{S/s}\right)\right) \\ &+ 
O^*\left(5.04 \zeta\left(\frac{3}{2}\right)^3 S \sqrt{\frac{x}{W}}
\prod_{p|v} \left(1 + \frac{1}{\sqrt{p}}\right)
\left(1 - \frac{1}{p^{3/2}}\right)^3
\right),
\end{aligned}\end{equation}
with $5.04$ replaced by $1.27$ if $v=2$.
The main term in (\ref{eq:flatow}) can be written as 
\begin{equation}\label{eq:fleming}
\frac{6 x}{\pi^2 W} \frac{v}{\sigma(v)} \mathop{\sum_{s\leq S}}_{\gcd(s,v)=1}
 \frac{1}{s} \int_{1/2}^1 
\mathop{\mathop{\sum_{r_1\leq \frac{uS}{s}} \sum_{r_2\leq \frac{uS}{s}}}_{\gcd(r_1,r_2)=1}}_{\gcd(r_1 r_2,v)=1}
\frac{\mu(r_1) \mu(r_2)}{\sigma(r_1) \sigma(r_2)} du .
\end{equation}
From now on, we will focus on the cases $v=1$ and $v=2$ for simplicity. (Higher
values of $v$ do not seem to be really profitable in the last analysis.)

\subsubsection{Explicit bounds for a sum with $\mu$}

We must estimate the expression within parentheses in (\ref{eq:fleming}).
It is not too hard to show that it tends to $0$; the first part of the proof
of Lemma \ref{lem:yutto} will reduce this to the fact that
$\sum_n \mu(n)/n = 0$. Obtaining good bounds is a more delicate matter.
For our purposes, we will need the expression to converge to $0$ at least
as fast as $1/(\log)^2$, with a good constant in front. For this task, the bound
(\ref{eq:marraki}) on $\sum_{n\leq x} \mu(n)/n$ is enough.

\begin{lem}\label{lem:yutto}
Let
\[g_v(x) := \mathop{\mathop{\sum_{r_1\leq x} \sum_{r_2\leq x}}_{\gcd(r_1,r_2)=1}}_{
\gcd(r_1 r_2,v)=1} 
\frac{\mu(r_1) \mu(r_2)}{\sigma(r_1) \sigma(r_2)},\]
where $v=1$ or $v=2$.
Then
\[|g_1(x)| \leq \begin{cases}
1/x &\text{if $33\leq x\leq 10^6$,}\\
\frac{1}{x} (111.536 + 55.768 \log x) &\text{if $10^6\leq x< 10^{10}$,}\\
\frac{0.0044325}{(\log x)^2} + \frac{0.1079}{\sqrt{x}} 
&\text{if $x\geq 10^{10}$,}
\end{cases}\]
\[|g_2(x)| \leq \begin{cases}
2.1/x &\text{if $33\leq x\leq 10^6$,}\\
\frac{1}{x} (1634.34 + 817.168\log x) 
 &\text{if $10^6\leq x< 10^{10}$,}\\
\frac{0.038128}{(\log x)^2}  + \frac{0.2046}{\sqrt{x}}.
 &\text{if $x\geq 10^{10}$.}
\end{cases}\]
\end{lem}
Tbe proof involves what may be called a version of
 Rankin's trick, using Dirichlet series and the behavior of $\zeta(s)$ near
$s=1$. The statements for $x\leq 10^6$ are proven by direct computation.\footnote{Using D. Platt's implementation \cite{Platt} of double-precision interval arithmetic. (In fact, one gets $2.0895071/x$ instead of $2.1/x$.)}
\begin{proof}
Clearly
\begin{equation}\label{eq:anna}\begin{aligned} 
g(x) &= \mathop{\sum_{r_1\leq x} \sum_{r_2\leq x}}_{\gcd(r_1 r_2,v)=1} 
\left(\sum_{d|\gcd(r_1,r_2)} \mu(d)\right)
\frac{\mu(r_1) \mu(r_2)}{\sigma(r_1) \sigma(r_2)}
\\ &= \mathop{\sum_{d\leq x}}_{\gcd(d,v)=1} 
\mu(d) \mathop{\mathop{\sum_{r_1\leq x} \sum_{r_2\leq x}}_{d|\gcd(r_1,r_2)}}_{
\gcd(r_1 r_2, v)=1}
\frac{\mu(r_1) \mu(r_2)}{\sigma(r_1) \sigma(r_2)}\\
&= \mathop{\sum_{d\leq x}}_{\gcd(d,v)=1} \frac{\mu(d)}{(\sigma(d))^2}
 \mathop{\sum_{u_1\leq x/d}}_{\gcd(u_1,d v)=1} \mathop{\sum_{u_2\leq x/d}}_{\gcd(u_2,d v)=1}
\frac{\mu(u_1) \mu(u_2)}{\sigma(u_1) \sigma(u_2)} \\ 
&= \mathop{\sum_{d\leq x}}_{\gcd(d,v)=1} \frac{\mu(d)}{(\sigma(d))^2}
\left(\mathop{\sum_{r\leq x/d}}_{\gcd(r,d v)=1} 
\frac{\mu(r)}{\sigma(r)}\right)^2.\end{aligned}\end{equation}
Moreover,
\[\begin{aligned}
\mathop{\sum_{r\leq x/d}}_{\gcd(r,d v)=1} 
\frac{\mu(r)}{\sigma(r)} &=
\mathop{\sum_{r\leq x/d}}_{\gcd(r,d v)=1} 
\frac{\mu(r)}{r} \sum_{d'|r} \prod_{p|d'} 
  \left(\frac{p}{p+1}-1\right)\\
&= \mathop{\mathop{\sum_{d'\leq x/d}}_{\mu(d')^2=1}}_{\gcd(d',d v)=1}
 \left(\prod_{p|d'} \frac{-1}{p+1}\right)
  \mathop{\mathop{\sum_{r\leq x/d}}_{\gcd(r,d v)=1}}_{d'|r} 
\frac{\mu(r)}{r}\\
&= \mathop{\mathop{\sum_{d'\leq x/d}}_{\mu(d')^2=1}}_{\gcd(d',d v)=1} 
\frac{1}{d' \sigma(d')}
  \mathop{\sum_{r\leq x/dd'}}_{\gcd(r,dd' v)=1} \frac{\mu(r)}{r}\end{aligned}\]
and
\[
\mathop{\sum_{r\leq x/dd'}}_{\gcd(r,dd'v)=1} \frac{\mu(r)}{r} =
\mathop{\sum_{d''\leq x/dd'}}_{d''|(d d'v)^{\infty}}
 \frac{1}{d''} \sum_{r\leq x/d d' d''} \frac{\mu(r)}{r}.\]
Hence
\begin{equation}\label{eq:onno}|g(x)|\leq \mathop{\sum_{d\leq x}}_{\gcd(d,v)=1}
 \frac{(\mu(d))^2}{(\sigma(d))^2}
 \left(\mathop{\mathop{\sum_{d'\leq x/d}}_{\mu(d')^2=1}}_{\gcd(d',d v)=1}
\frac{1}{d' \sigma(d')} \mathop{\sum_{d''\leq x/d d'}}_{d''|(d d' v)^{\infty}}
 \frac{1}{d''} f(x/d d' d'')\right)^2,\end{equation}
where $f(t) = \left|\sum_{r\leq t} \mu(r)/r\right|$.

We intend to bound the function $f(t)$ by 
a linear combination of terms of the form $t^{-\delta}$, 
$\delta\in \lbrack 0,1/2)$. Thus it makes sense now to estimate
$F_v(s_1,s_2,x)$, defined to be
the quantity
\[\begin{aligned}
\mathop{\sum_d}_{\gcd(d,v)=1} \frac{(\mu(d))^2}{(\sigma(d))^2} 
&\left(\mathop{\sum_{d_1'}}_{\gcd(d_1',d v)=1} 
\frac{\mu(d_1')^2}{d_1' \sigma(d_1')}
\sum_{d_1''|(d d_1' v)^{\infty}}
 \frac{1}{d_1''} \cdot (d d_1' d_1'')^{1-s_1}\right)\\
 &\left(\mathop{\sum_{d_2'}}_{\gcd(d_2',d v)=1} 
\frac{\mu(d_2')^2}{d_2' \sigma(d_2')}
\sum_{d_2''|(d d_2' v)^{\infty}}
 \frac{1}{d_2''} \cdot (d d_2' d_2'')^{1-s_2}\right) .
\end{aligned}\]
for $s_1,s_2\in \lbrack 1/2,1\rbrack$. This is equal to
\[\begin{aligned}
\mathop{\sum_d}_{\gcd(d,v)=1} &\frac{\mu(d)^2}{d^{s_1+s_2}} \prod_{p|d}
\frac{1}{\left(1+p^{-1}\right)^2 \left(1-p^{-s_1}\right) \prod_{p|v}
\frac{1}{(1-p^{-{s_1}}) (1 - p^{-s_2})}
\left(1-p^{-s_2}\right)}\\
&\cdot \left(\mathop{\sum_{d'}}_{\gcd(d',d v)=1} 
\frac{\mu(d')^2}{(d')^{s_1+1}} \prod_{p'|d'}
\frac{1}{\left(1+p'^{-1}\right) \left(1-p'^{-s_1}\right)}\right)\\
&\cdot \left(\mathop{\sum_{d'}}_{\gcd(d',d v)=1} 
\frac{\mu(d')^2}{(d')^{s_2+1}} \prod_{p'|d'}
\frac{1}{\left(1+p'^{-1}\right) \left(1-p'^{-s_2}\right)}\right) ,
\end{aligned}\]
which in turn can easily be seen to equal
\begin{equation}\label{eq:prof}\begin{aligned}
\prod_{p\nmid v} &\left(1 + 
\frac{p^{-s_1} p^{-s_2}}{(1 - p^{-s_1} +p^{-1})
(1 - p^{-s_2} +p^{-1})}\right)
\prod_{p|v} \frac{1}{(1 - p^{-s_1}) (1 - p^{-s_2})}
\\
&\cdot \prod_{p\nmid v} \left(1 + \frac{p^{-1} p^{-s_1}}{(1+p^{-1}) 
(1-p^{-s_1})}\right)
\cdot \prod_{p\nmid v} \left(1 + \frac{p^{-1} p^{-s_2}}{(1+p^{-1}) 
(1-p^{-s_2})}\right)
\end{aligned}\end{equation}
Now, for any $0<x\leq y\leq x^{1/2}<1$, 
\[(1+x-y) (1-xy) (1-xy^2) - (1+x) (1-y) (1- x^3) = (x-y) ( y^2-x) ( xy - x
-1) x \leq 0,\] and so
\begin{equation}1 + \frac{xy}{(1+x)(1-y)} = 
\frac{(1+x-y) (1-xy) (1-x y^2)}{(1+x) (1-y) (1-xy) (1-xy^2)} \leq
\frac{(1-x^3)}{(1-xy) (1-xy^2)}.\end{equation}

For any $x\leq y_1,y_2<1$ with $y_1^2\leq x$, $y_2^2\leq x$,
\begin{equation}\label{eq:odmalicka}
1 + \frac{y_1 y_2}{(1-y_1+x) (1-y_2+x)} \leq
 \frac{(1-x^3)^2 (1-x^4)}{(1- y_1 y_2) (1 - y_1 y_2^2) (1 - y_1^2 y_2)} .
\end{equation}
This can be checked as follows: multiplying by the denominators and changing
variables to $x$, $s=y_1+y_2$ and $r=y_1 y_2$, we obtain an inequality where
the left side, quadratic on $s$ with positive leading coefficient, must be
less than or equal to the right side, which is linear on $s$. The left side
minus the right side can be maximal for given $x$, $r$ only when $s$ is
maximal or minimal. This happens when $y_1=y_2$ or when either
 $y_i = \sqrt{x}$ or $y_i = x$ for at least one of $i=1,2$. In each of these
cases, we have reduced (\ref{eq:odmalicka}) to an inequality in two variables
that can be proven automatically\footnote{In practice, the case $y_i = \sqrt{x}$
leads to a polynomial of high degree, and quantifier elimination increases
sharply in complexity as the degree increases; a stronger inequality of
lower degree (with $(1-3 x^3)$ instead of $(1-x^3)^2 (1-x^4)$) 
was given to QEPCAD to prove in this case.} 
by a quantifier-elimination program; the 
author has used QEPCAD \cite{QEPCAD} to do this.

Hence $F_v(s_1,s_2,x)$ is at most
\begin{equation}\label{eq:tausend}\begin{aligned}
\prod_{p\nmid v} &\frac{(1-p^{-3})^2 (1-p^{-4})}{(1-p^{-s_1-s_2}) (1-p^{-2s_1-s_2})
(1-p^{-s_1-2s_2})} \cdot \prod_{p|v} \frac{1}{(1-p^{-s_1}) (1-p^{-s_2})}
 \\ &\cdot \prod_{p\nmid v}
\frac{1-p^{-3}}{(1+p^{-s_1-1}) (1 + p^{-2s_1-1})}
\prod_{p\nmid v}
\frac{1-p^{-3}}{(1+p^{-s_2 -1}) (1 + p^{-2 s_2 - 1})}\\
&= C_{v,s_1,s_2} \cdot \frac{
\zeta(s_1+1) \zeta(s_2+1)
\zeta(2s_1+1) \zeta(2s_2 +1 ) 
}{\zeta(3)^4 \zeta(4)
(\zeta(s_1+s_2)  \zeta(2 s_1 + s_2) \zeta(s_1 + 2 s_2))^{-1}
},
\end{aligned}\end{equation}
where \[C_{v,s_1,s_2} = \begin{cases} 1 &\text{if $v=1$},\\
\frac{
 (1-2^{-s_1-2s_2}) 
(1+2^{-s_1-1}) 
(1+2^{-2s_1-1}) (1+ 2^{-s_2-1}) (1+2^{-2s_2-1})}{(1-2^{-{s_1+s_2}})^{-1} (1-2^{-2s_1-s_2})^{-1}
(1-2^{-s_1})(1-2^{-s_2}) (1-2^{-3})^4
(1-2^{-4})}
&\text{if $v=2$.}\end{cases}\]

For $1\leq t\leq x$,
(\ref{eq:marraki}) and (\ref{eq:ramare}) imply
\begin{equation}\label{eq:kustor}
f(t) \leq \begin{cases} \sqrt{\frac{2}{t}} & \text{if $x\leq 10^{10}$}\\
\sqrt{\frac{2}{t}} + \frac{0.03}{\log x} \left(\frac{x}{t}\right)^{ 
  \frac{\log \log 10^{10}}{\log x - \log 10^{10}}} &\text{if $x>10^{10}$},
\end{cases}
\end{equation}
where we are using the fact that $\log x$ is convex-down. Note that,
again by convexity,
\[\frac{\log \log x - \log \log 10^{10}}{\log x - \log 10^{10}}
< (\log t)'|_{t=\log 10^{10}} = \frac{1}{\log 10^{10}} = 0.0434294\dots
\]
Obviously, $\sqrt{2/t}$ in (\ref{eq:kustor}) can be replaced by $(2/t)^{1/2-
\epsilon}$ for any $\epsilon\geq 0$. 

By (\ref{eq:onno}) and (\ref{eq:kustor}),
\[|g_v(x)|\leq \left(\frac{2}{x}\right)^{1-2\epsilon}
 F_v(1/2+\epsilon,1/2+\epsilon,x)\]
for $x\leq 10^{10}$. We set $\epsilon=1/\log x$ and obtain
from (\ref{eq:tausend}) that
\begin{equation}\label{eq:rot}\begin{aligned}
F_v(1/2+\epsilon,1/2+\epsilon,x) &\leq
C_{v,\frac{1}{2} + \epsilon,\frac{1}{2} + \epsilon}
\frac{\zeta(1+2\epsilon) \zeta(3/2)^4 \zeta(2)^2}{\zeta(3)^4
\zeta(4)}\\
&\leq 55.768\cdot C_{v,\frac{1}{2} + \epsilon,\frac{1}{2} + \epsilon}
\cdot \left(1 + \frac{\log x}{2}\right) 
,\end{aligned}\end{equation}
where we use the easy bound $\zeta(s)< 1+ 1/(s-1)$ obtained by
\[\sum n^s < 1 + \int_1^{\infty} t^s dt.\] (For sharper bounds, see
\cite{MR1924708}.) Now
\[\begin{aligned}
C_{2,\frac{1}{2}+\epsilon,\frac{1}{2} + \epsilon} &\leq
\frac{
 (1-2^{-3/2-\epsilon})^2 
(1+2^{-3/2})^2  (1+2^{-2})^2 (1-2^{-1-2\epsilon})}{
(1-2^{-1/2})^2 (1-2^{-3})^4 (1-2^{-4})}
&\leq 14.652983
,\end{aligned}\]
whereas $C_{1,\frac{1}{2}+\epsilon,\frac{1}{2}+\epsilon}=1$.
(We are assuming $x\geq 10^6$, and so $\epsilon\leq 1/(\log 10^6)$.)
Hence
\[|g_v(x)|\leq
\begin{cases} 
\frac{1}{x} (111.536 + 55.768\log x) &\text{if $v=1$,}\\
\frac{1}{x} (1634.34 + 817.168\log x) 
& \text{if $v=2$.}\end{cases}\] 
for $10^6\leq x< 10^{10}$.

For general $x$, we must use the second bound in (\ref{eq:kustor}).
Define $c = 1/(\log 10^{10})$. We see that, if $x>10^{10}$,
\[\begin{aligned}
|g_v(x)|&\leq \frac{0.03^2}{(\log x)^2} F_1(1-c,1-c) \cdot C_{v,1-c,1-c}\\&+ 2
\cdot \frac{\sqrt{2}}{\sqrt{x}} \frac{0.03}{\log x} F(1-c,1/2) 
\cdot C_{v,1-c,1/2}\\ &+
\frac{1}{x} (111.536 + 55.768 \log x)\cdot C_{v,\frac{1}{2}+\epsilon,
\frac{1}{2}+\epsilon}.\end{aligned}\]
For $v=1$, this gives
\[\begin{aligned}
|g_1(x)| 
&\leq \frac{0.0044325}{(\log x)^2} + \frac{2.1626}{\sqrt{x} \log x} +
\frac{1}{x} (111.536 + 55.768 \log x)\\
&\leq \frac{0.0044325}{(\log x)^2} + \frac{0.1079}{\sqrt{x}} ;
\end{aligned}\]
for $v=2$, we obtain
\[\begin{aligned}
|g_2(x)| &\leq
\frac{0.038128}{(\log x)^2} 
+ \frac{25.607}{\sqrt{x} \log x} + 
\frac{1}{x} (1634.34 + 817.168\log x) \\
&\leq \frac{0.038128}{(\log x)^2}  + \frac{0.2046}{\sqrt{x}}.
\end{aligned}\]
\end{proof}

\subsubsection{Estimating the triple sum}
We will now be able to bound the triple sum in (\ref{eq:fleming}), viz.,
\begin{equation}\label{eq:grotto}
\mathop{\sum_{s\leq S}}_{\gcd(s,v)=1} \frac{1}{s} \int_{1/2}^{1} g_v(uS/s) du,\end{equation}
where $g_v$ is as in Lemma \ref{lem:yutto}.

As we will soon see, Lemma \ref{lem:yutto} that (\ref{eq:grotto}) is bounded
by a constant (essentially because the integral $\int_0^{1/2} 1/t(\log t)^2$
converges).
We must give as good a constant as we can, since it will affect the largest
term in the final result.

Clearly $g_v(R) = g_v(\lfloor R\rfloor)$. The contribution of each $g_v(m)$,
$1\leq m\leq S$, to (\ref{eq:grotto}) is exactly $g_v(m)$ times
\begin{equation}\label{eq:greco}\begin{aligned}
&\mathop{\sum_{\frac{S}{m+1} < s \leq \frac{S}{m}} \frac{1}{s}}_{\gcd(s,v)=1} 
\int_{ms/S}^1 du + 
\mathop{\sum_{\frac{S}{2 m} < s \leq \frac{S}{m+1}} \frac{1}{s}}_{\gcd(s,v)=1}  
\int_{ms/S}^{(m+1)s/S} du 
+ \mathop{\sum_{\frac{S}{2 (m+1)} < s \leq \frac{S}{2 m}} \frac{1}{s}}_{\gcd(s,v)=1}  
\int_{1/2}^{(m+1)s/S} du\\
&= 
\mathop{
\sum_{\frac{S}{m+1} < s \leq \frac{S}{m}}}_{\gcd(s,v)=1} 
 \left(\frac{1}{s} - \frac{m}{S}\right) + 
\mathop{\sum_{\frac{S}{2 m} < s \leq \frac{S}{m+1}}}_{\gcd(s,v)=1}  \frac{1}{S} +
\mathop{\sum_{\frac{S}{2 (m+1)} < s \leq \frac{S}{2 m}}}_{\gcd(s,v)=1} 
 \left(\frac{m+1}{S} - \frac{1}{2s}\right).
\end{aligned}\end{equation}
Write $f(t) = 1/S$ for $S/2m < t\leq S/(m+1)$, $f(t)=0$ for $t>S/m$ or
$t<S/2(m+1)$, $f(t) = 1/t - m/S$ for $S/(m+1) <t\leq S/m$ and 
$f(t) = (m+1)/S - 1/2t$ for $S/2(m+1) < t\leq S/2m$; then (\ref{eq:greco})
equals $\sum_{n: \gcd(n,v)=1} f(n)$. By Euler-Maclaurin (second order),
\begin{equation}\label{eq:etex}\begin{aligned}
\sum_n f(n) &= \int_{-\infty}^{\infty} f(x) - \frac{1}{2} B_2(\{ x\}) f''(x) dx
= \int_{-\infty}^{\infty} f(x) + O^*\left(\frac{1}{12} |f''(x)|\right) dx\\
&= \int_{-\infty}^{\infty} f(x) dx + \frac{1}{6} \cdot O^*\left(\left|f'\left(
\frac{3}{2m}\right)\right| + \left|f'\left(\frac{s}{m+1}\right)\right|\right)\\
&= \frac{1}{2} \log\left(1 + \frac{1}{m}\right) +\frac{1}{6}\cdot
 O^*\left(\left(\frac{2 m}{s}\right)^2 + \left(\frac{m+1}{s}\right)^2\right).\end{aligned}\end{equation}
Similarly,
\[\begin{aligned}
\sum_{\text{$n$ odd}} f(n) &= \int_{-\infty}^{\infty} f(2x+1) - \frac{1}{2} B_2(\{ x\}) 
\frac{d^2 f(2x+1)}{d x^2} dx \\ &= \frac{1}{2} \int_{-\infty}^{\infty} f(x) dx -
2\int_{-\infty}^{\infty} \frac{1}{2} B_2\left(\left\{\frac{x-1}{2}\right\}\right)
f''(x) dx\\ &= \frac{1}{2} \int_{-\infty}^{\infty} f(x) dx + 
\frac{1}{6}  \int_{-\infty}^\infty
O^*\left(
|f''(x)|\right) dx \\ &=
\frac{1}{4} \log\left(1 + \frac{1}{m}\right) + \frac{1}{3} 
\cdot O^*\left(\left(\frac{2 m}{s}\right)^2 + \left(\frac{m+1}{s}\right)^2\right).
\end{aligned}\]

We use these expressions for $m\leq C_0$, where $C_0\geq 33$ is a
constant to be computed later; they will give us the main term. For $m> C_0$,
we use the bounds on $|g(m)|$ that Lemma \ref{lem:yutto} gives us.

(Starting now and for the rest of the paper, we will focus on the cases
$v=1$, $v=2$ when giving explicit computational estimates. All of our 
procedures would allow higher values of $v$ as well, but, as will become clear
much later, the gains 
from higher values of $v$ are offset by losses and complications elsewhere.)

Let us estimate (\ref{eq:grotto}).
Let \[c_{v,0} = \begin{cases} 1/6 &\text{if $v=1$,}\\ 1/3 &\text{if
    $v=2$,}\end{cases}
\;\;\;\;c_{v,1} = \begin{cases} 1 &\text{if $v=1$,}\\ 2.5 &\text{if
    $v=2$,}\end{cases}\]
\[c_{v,2} = \begin{cases} 55.768\dotsc &\text{if $v=1$,}\\ 817.168\dotsc &\text{if
    $v=2$,}\end{cases}\;\;\;\;
c_{v,3} = \begin{cases} 111.536\dotsc &\text{if $v=1$,}\\ 1634.34\dotsc &\text{if
    $v=2$,}\end{cases}\]
\[c_{v,4} = \begin{cases} 0.0044325\dotsc &\text{if $v=1$,}\\ 0.038128\dotsc&\text{if
    $v=2$,}\end{cases}\;\;\;\;
c_{v,5} = \begin{cases}  0.1079\dotsc&\text{if $v=1$,}\\ 0.2046\dotsc &\text{if
    $v=2$.}\end{cases}\]
Then (\ref{eq:grotto}) equals
\[\begin{aligned}
\sum_{m\leq C_0} &g_v(m) \cdot\left(\frac{\phi(v)}{2 v} \log\left(1+ \frac{1}{m}\right)
+ O^*\left(c_{v,0} \frac{5 m^2 + 2 m+1}{S^2}\right)\right)\\
&+ \sum_{S/10^6 \leq s < S/C_0} \frac{1}{s} \int_{1/2}^1 
O^*\left(\frac{c_{v,1}}{uS/s}\right) du \\ &+ 
\sum_{S/10^{10}\leq s<S/10^6} \frac{1}{s} 
\int_{1/2}^1 
O^*\left(
\frac{c_{v,2} \log(uS/s) + c_{v,3}}{uS/s}\right) 
du  \\ &+
\sum_{s < S/10^{10}} \frac{1}{s} \int_{1/2}^1 O^*\left(\frac{c_{v,4}}{(\log uS/s)^2}
+ \frac{c_{v,5}}{\sqrt{uS/s}}\right) du,\end{aligned}\]
which is
\[\begin{aligned}
\sum_{m\leq C_0} &g_v(m) \cdot \frac{\phi(v)}{2 v} \log\left(1+ \frac{1}{m}\right)
+ \sum_{m\leq C_0} |g(m)|\cdot O^*\left(
c_{v,0} \frac{5 m^2 + 2 m+1}{S^2}\right)\\
&+ O^*\left(c_{v,1} \frac{\log 2}{C_0}
    + \frac{\log 2}{10^6} \left( c_{v,3} + c_{v,2} (1 + \log 10^6)\right)
+ \frac{2-\sqrt{2}}{10^{10/2}} c_{v,5}\right)
\\ &+ O^*\left(\sum_{s<S/10^{10}} \frac{c_{v,4}/2}{s (\log
    S/2s)^2}\right)
\end{aligned}\]
for $S\geq (C_0+1)$. 
Note that $\sum_{s<S/10^{10}} \frac{1}{s (\log S/2s)^2} = 
\int_0^{2/10^{10}} \frac{1}{t (\log t)^2} dt$.


Now 
\[\frac{c_{v,4}}{2} \int_0^{2/10^{10}} \frac{1}{t (\log t)^2} dt =
\frac{c_{v,4}/2}{\log(10^{10}/2)} = 
\begin{cases} 0.00009923\dotsc &\text{if $v=1$}\\
0.000853636\dotsc &\text{if $v=2$.}\end{cases}\]
and
\[\frac{\log 2}{10^6} \left( c_{v,3} + c_{v,2} (1 + \log 10^6)\right)
+ \frac{2-\sqrt{2}}{10^{5}} c_{v,5} = 
\begin{cases}
0.0006506\dotsc &\text{if $v=1$}\\
0.009525\dotsc &\text{if $v=2$.}\end{cases}\]
For $C_0 = 10000$,
\[\begin{aligned}
\frac{\phi(v)}{v} \frac{1}{2} \sum_{m\leq C_0} g_v(m) \cdot
               \log\left(1+ \frac{1}{m}\right)
&= \begin{cases} 0.362482\dotsc &\text{if $v=1$,}\\
0.360576\dotsc &\text{if $v=2$,}\end{cases}\\
c_{v,0} \sum_{m\leq C_0} |g_v(m)| (5 m^2 + 2 m + 1) &\leq 
\begin{cases} 6204066.5\dotsc &\text{if $v=1$,}\\
15911340.1\dotsc &\text{if $v=2$,}\end{cases}
\end{aligned}\]
and 
\[c_{v,1}\cdot (\log 2)/C_0 = \begin{cases} 0.00006931\dotsc
&\text{if $v=1$,}\\ 0.00017328\dotsc &\text{if $v=2$.}\end{cases}\]

Thus, for $S\geq 100000$, 
\begin{equation}\label{eq:corto}
\mathop{\sum_{s\leq S}}_{\gcd(s,v)=1} \frac{1}{s} \int_{1/2}^{1} g_v(uS/s) du
\leq \begin{cases}0.36393
&\text{if $v=1$,}\\ 
0.37273 &\text{if $v=2$.}\end{cases}
\end{equation}

For $S < 100000$, we proceed as above, but using the exact expression
(\ref{eq:greco}) instead of (\ref{eq:etex}). Note (\ref{eq:greco}) is
of the form $f_{s,m,1}(S) + f_{s,m,2}(S)/S$, where both $f_{s,m,1}(S)$ and
$f_{s,m,2}(S)$ depend only on $\lfloor S\rfloor$ (and on $s$ and $m$).
Summing over $m\leq S$, we obtain a bound of the form
\[\mathop{\sum_{s\leq S}}_{\gcd(s,v)=1} \frac{1}{s} \int_{1/2}^{1} g_v(uS/s) du \leq
G_v(S)\]
with \[G_v(S) = K_{v,1}(|S|) + K_{v,2}(|S|)/S,\]
where $K_{v,1}(n)$ and $K_{v,2}(n)$ can be computed explicitly for 
each integer $n$. (For example, $G_v(S) = 1 - 1/S$ for
$1\leq S < 2$ and $G_v(S) = 0$ for $S<1$.)

It is easy to check numerically
 that this implies that (\ref{eq:corto}) holds not just
for $S\geq 100000$ but also for $40\leq S<100000$ (if $v=1$) or
$16\leq S < 100000$ (if $v=2$). Using the fact that $G_v(S)$ is non-negative,
we can compare $\int_1^T G_v(S) dS/S$ with $\log(T+1/N)$ for each
$T\in \lbrack 2,40\rbrack \cap \frac{1}{N} \mathbb{Z}$ ($N$ a large integer)
to show, again numerically, that
\begin{equation}\label{eq:passi}
\int_1^T G_v(S) \frac{dS}{S} \leq \begin{cases} 0.3698 \log T
&\text{if $v=1$,}\\
0.37273 \log T &\text{if $v=2$.}\end{cases}\end{equation}
(We use $N=100000$ for $v=1$; already $N=10$ gives us the answer above for
$v=2$. Indeed, computations suggest the better bound $0.358$ instead of
$0.37273$; we are committed to using 
$0.37273$ because of (\ref{eq:corto}).)

Multiplying by $6 v/\pi^2\sigma(v)$, we conclude that
\begin{equation}\label{eq:menson1}S_1(U,W) = 
\frac{x}{W} \cdot H_1\left(\frac{x}{W U}\right)
+ O^*\left(5.08 \zeta(3/2)^3 \frac{x^{3/2}}{W^{3/2} U} \right)
\end{equation}
if $v=1$,
\begin{equation}\label{eq:menson2}S_1(U,W) = 
\frac{x}{W} \cdot H_2\left(\frac{x}{W U}\right)
+O^*\left(1.27 \zeta(3/2)^3 \frac{x^{3/2}}{W^{3/2} U} \right)
\end{equation}
if $v=2$, where 
\begin{equation}\label{eq:palmiped}
H_1(S) = \begin{cases} \frac{6}{\pi^2} G_1(S)
 &\text{if $1\leq S < 40$,}\\ 0.22125 &\text{if $S\geq 40$,}
\end{cases}\;\;\;\;\;\;\;\;
H_2(s) = \begin{cases} \frac{4}{\pi^2} G_2(S) 
&\text{if $1\leq S < 16$,}\\ 0.15107 &\text{if $S\geq 16$.}
\end{cases}\end{equation}
Hence (by (\ref{eq:passi}))
\begin{equation}\label{eq:velib}\begin{aligned}
\int_1^{T} H_v(S) \frac{dS}{S} &\leq
\begin{cases} 0.22482 \log T &\text{if $v=1$,}\\
0.15107 \log T &\text{if $v=2$;}\end{cases}
\end{aligned}\end{equation}
 moreover,
$H_1(S)\leq 3/\pi^2$, $H_2(S)\leq 2/\pi^2$ for all $S$.
\begin{center}
* * *
\end{center}

{\em Note.} There is another way to obtain cancellation on $\mu$, applicable
when $(x/W)> Uq$ (as is unfortunately never the case in our main
application). For this alternative
to be taken, one must either apply Cauchy-Schwarz on $n$ rather than $m$
(resulting in exponential sums over $m$) or lump together all $m$ near each 
other and in the same
congruence class modulo $q$ before applying Cauchy-Schwarz on $m$ (one can
indeed do this if $\delta$ is small). We could then write
\[\mathop{\sum_{m\sim W}}_{m\equiv r \mo q} \mathop{\sum_{d|m}}_{d>U} \mu(d) =
 - \mathop{\sum_{m\sim W}}_{m\equiv r \mo q} \mathop{\sum_{d|m}}_{d\leq U} \mu(d) =
 -\sum_{d\leq U} \mu(d) (W/qd + O(1))\]
and obtain cancellation on $d$. If $Uq\geq (x/W)$, however,
the error term dominates.

\subsection{The sum $S_2$: the large sieve, primes and tails}
We must now bound
\begin{equation}\label{eq:honi}
S_2(U',W',W) = 
\mathop{\sum_{U'<m\leq \frac{x}{W}}}_{\gcd(m,v)=1} 
\left|\sum_{W' < p\leq W} (\log p)
 e(\alpha m p)\right|^2.\end{equation}
for $U' = \max(U,x/2W)$, $W' = \max(V,W/2)$.
(The condition $\gcd(p,v)=1$ will be fulfilled automatically by the
assumption $V>v$.)

From a modern perspective, this is clearly a case for a large sieve. 
It is also clear that we ought to try to apply
a large sieve for sequences of prime support. What is subtler here is how to do
things well for very large $q$ (i.e., $x/q$ small). This is in some sense
a dual problem to that of $q$ small, but it poses additional complications;
for example, it is not obvious how to take advantage of prime support
for very large $q$.

As in type I, we avoid this entire issue by forbidding $q$ large and then taking
advantage of the error term $\delta/x$ in the approximation 
$\alpha = \frac{a}{q} + \frac{\delta}{x}$. This is one of the main innovations
here. Note this alternative method will allow us to take advantage of prime
support.

A key situation to study is that of
frequencies $\alpha_i$ clustering around given rationals $a/q$ while
nevertheless keeping at a certain small distance from each other.

\begin{lem}\label{lem:ogor}
Let $q\geq 1$.
Let $\alpha_1,\alpha_2,\dotsc,\alpha_k\in \mathbb{R}/\mathbb{Z}$ be of
the form $\alpha_i = a_i/q + \upsilon_i$, $0\leq a_i<q$, 
where the elements $\upsilon_i\in \mathbb{R}$
all lie in an interval of length $\upsilon>0$, and where $a_i=a_j$ implies
$|\upsilon_i - \upsilon_j|>\nu>0$. Assume $\nu+\upsilon\leq 1/q$.
Then, for any $W, W'\geq 1$, $W'\geq W/2$,
\begin{equation}\label{eq:crut}\begin{aligned}
\sum_{i=1}^{k} \left|\sum_{W'<p\leq W} (\log p) e(\alpha_i p)\right|^2
&\leq \min\left(1, \frac{2 q}{\phi(q)} \frac{1}{\log\left((q (\nu+\upsilon))^{-1}
\right)}\right)\\ &\cdot \left(W-W'+\nu^{-1}\right) \sum_{W'<p\leq W} (\log p)^2.
\end{aligned}
\end{equation}
\end{lem}
\begin{proof}
For any distinct $i$, $j$, the angles $\alpha_i$, $\alpha_j$ are separated
by at least $\nu$ (if $a_i=a_j$) or at least $1/q - |\upsilon_i-\upsilon_j|\geq
1/q-\upsilon\geq \nu$ (if $a_i\ne a_j$). Hence we can apply the large sieve
(in the optimal $N+\delta^{-1}-1$
form due to Selberg \cite{Sellec} and Montgomery-Vaughan
\cite{MR0337775})
and obtain the bound in (\ref{eq:crut}) with $1$ instead of $\min(1,\dotsc)$
immediately.

We can also apply Montgomery's inequality
(\cite{MR0224585}, \cite{MR0311618}; see the expositions in
\cite[pp. 27--29]{MR0337847} and  \cite[\S 7.4]{MR2061214}). 
This gives us that the left side of (\ref{eq:crut}) is at most
\begin{equation}\label{eq:toyor}
\left(\mathop{\sum_{r\leq R}}_{\gcd(r,q)=1} \frac{(\mu(r))^2}{\phi(r)}
\right)^{-1} 
\mathop{\sum_{r\leq R}}_{\gcd(r,q)=1} \mathop{\sum_{a' \mo r}}_{\gcd(a',r)=1} 
 \sum_{i=1}^{k} \left|\sum_{W'<p\leq W} (\log p) e((\alpha_i+a'/r) p)\right|^2
\end{equation}
If we add all possible fractions
of the form $a'/r$, $r\leq R$, $\gcd(r,q)=1$, to the fractions $a_i/q$,
we obtain fractions that are
separated by at least $1/qR^2$. If $\nu+\upsilon\geq 1/qR^2$, then the resulting
angles $\alpha_i + a'/r$ are still separated by at least $\nu$. Thus
we can apply the large sieve to (\ref{eq:toyor}); setting $R = 
1/\sqrt{(\nu+\upsilon) q}$, we see that we gain a factor of
\begin{equation}\label{eq:werst}
\mathop{\sum_{r\leq R}}_{\gcd(r,q)=1} \frac{(\mu(r))^2}{\phi(r)} \geq
\frac{\phi(q)}{q} \sum_{r\leq R} \frac{(\mu(r))^2}{\phi(r)} \geq
\frac{\phi(q)}{q} \sum_{d\leq R} \frac{1}{d} \geq
\frac{\phi(q)}{2 q} \log\left((q (\nu+\upsilon))^{-1}\right),\end{equation}
since $\sum_{d\leq R} 1/d \geq \log(R)$ for all $R\geq 1$ (integer or not).
\end{proof}

Let us first give a bound on sums of the type of 
$S_2(U,V,W)$ using prime support
but not the error terms (or Lemma \ref{lem:ogor}).
\begin{lem}\label{lem:kastor1}
Let $W\geq 1$, $W'\geq W/2$. Let 
$\alpha = a/q + O^*(1/q Q)$, $q\leq Q$.  Then
\begin{equation}\label{eq:pokor1}\begin{aligned}
\sum_{A_0<m\leq A_1} &\left|\sum_{W'<p\leq W} (\log p) e(\alpha m p)\right|^2
\\ &\leq 
\left\lceil \frac{A_1-A_0}{\min(q,\lceil Q/2\rceil)} \right\rceil\cdot
(W-W'+2q) \sum_{W'<p\leq W} (\log p)^2.\end{aligned}\end{equation}
If $q<W/2$ and $Q\geq 3.5 W$, the following bound also holds:
\begin{equation}\label{eq:pokor2}\begin{aligned}
\sum_{A_0<m\leq A_1} &\left|\sum_{W'<p\leq W} (\log p) e(\alpha m p)\right|^2\\
&\leq \left\lceil \frac{A_1-A_0}{q}\right\rceil\cdot
\frac{q}{\phi(q)} \frac{W}{\log(W/2q)}
\cdot \sum_{W'<p\leq W} (\log p)^2.\end{aligned}\end{equation}
If $A_1-A_0\leq \varrho q$ and $q\leq \rho Q$, $\varrho,\rho \in \lbrack 0,1
\rbrack$, the following bound also holds:
\begin{equation}\label{eq:zerom}\begin{aligned}
\sum_{A_0<m\leq A_1} &\left|\sum_{W'<p\leq W} (\log p) e(\alpha m p)\right|^2\\
&\leq (W-W'+q/(1-\varrho \rho)) \sum_{W'<p\leq W} (\log p)^2.\end{aligned}
\end{equation}
\end{lem}
The inequality (\ref{eq:pokor2}) can be stronger than (\ref{eq:pokor2})
only when $q<W/7.2638\dotsc$ (if $q$ is odd) or $q<W/92.514\dotsc$
(if $q$ is even).
\begin{proof}
Let $k = \min(q,\lceil Q/2\rceil) \geq \lceil q/2\rceil$. We split
$(A_0,A_1\rbrack$ into $\lceil (A_1-A_0)/k\rceil$ blocks of at most $k$
consecutive integers $m_0+1, m_0+2,\dotsc$. 
For $m$, $m'$ in such a block, $\alpha m$ and
$\alpha m'$ are separated by a distance of at least 
\[|\{(a/q) (m-m')\}| - O^*(k/qQ) = 1/q - O^*(1/2q) \geq 1/2q.\]
By the large sieve 
\begin{equation}\label{eq:roussel}
\sum_{a=1}^{q} 
\left|\sum_{W'<p\leq W} (\log p) e(\alpha (m_0+a) p)\right|^2
\leq ((W-W') + 2q) \sum_{W'<p\leq W} (\log p)^2.\end{equation}
We obtain (\ref{eq:pokor1}) by summing over all
$\lceil (A_1-A_0)/k\rceil$ blocks.

If $A_1-A_0\leq |\varrho q|$ and $q\leq \rho Q$, $\varrho,\rho \in \lbrack 0,1
\rbrack$, we obtain (\ref{eq:zerom}) simply by applying the large sieve
without splitting the interval $A_0<m\leq A_1$.

Let us now prove (\ref{eq:pokor2}). We will
 use Montgomery's inequality, followed by Montgomery and Vaughan's
large sieve with weights.
An angle $a/q+a_1'/r_1$ is separated from other angles $a'/q+a_2'/r_2$
($r_1, r_2\leq R$, $\gcd(a_i,r_i)=1$) 
by at least $1/q r_1 R$, rather than just $1/q R^2$.
We will choose $R$ so that $q R^2 < Q$; this implies $1/Q < 1/q R^2\leq
1/q r_1 R$. 

By Montgomery's inequality
\cite[Lemma 7.15]{MR2061214}, applied
(for each $1\leq a\leq q$)
 to $S(\alpha) = \sum_n a_n e(\alpha n)$
with $a_n = \log(n) e(\alpha(m_0+a) n)$ if $n$ is prime and $a_n=0$ otherwise,
\begin{equation}\label{eq:cortomalt}\begin{aligned}
\frac{1}{\phi(r)} &\left|\sum_{W'<p\leq W} (\log p) e(\alpha (m_0+a)p)\right|^2 \\
&\leq
\mathop{\sum_{a' \mo r}}_{\gcd(a',r)=1} \left|
\sum_{W'<p\leq W} (\log p) e\left(\left(\alpha \left(m_0+a\right) +
\frac{a'}{r}\right) p\right)\right|^2 .\end{aligned}\end{equation}
for each square-free $r\leq W'$.
We multiply both sides of (\ref{eq:cortomalt}) by 
$(W/2 + (3/2) (1/q r R - 1/Q)^{-1})^{-1}$ and sum over all $a=0,1,\dotsc,q-1$ and
all square-free $r\leq R$ coprime to $q$;
we will later make sure that $R\leq W'$. We obtain that
\begin{equation}\label{eq:malheur}\begin{aligned}
\mathop{\sum_{r\leq R}}_{\gcd(r,q)=1}
 &\left(\frac{W}{2} + \frac{3}{2} \left(\frac{1}{q r R} - \frac{1}{Q}\right)^{-1}\right)^{-1} 
\frac{\mu(r)^2}{\phi(r)}\\ &\cdot \sum_{a=1}^q 
\left|\sum_{W'<p\leq W} (\log p) e(\alpha (m_0+a)p)\right|^2\end{aligned}
\end{equation}
is at most
\begin{equation}\label{eq:douben}\begin{aligned}
\mathop{\mathop{\sum_{r\leq R}}_{\gcd(r,q)=1}}_{\text{$r$ sq-free}}
&\left(\frac{W}{2} + \frac{3}{2} \left(\frac{1}{q r R} - \frac{1}{Q}\right)^{-1}\right)^{-1} \\
&\sum_{a=1}^q \mathop{\sum_{a' \mo r}}_{\gcd(a',r)=1} 
\left|\sum_{W'<p\leq W} (\log p) e\left(\left(\alpha \left(m_0+a\right) +
\frac{a'}{r}\right) p\right)\right|^2 
\end{aligned}\end{equation}
We now apply the large sieve with weights
\cite[(1.6)]{MR0374060}, recalling that each angle 
$\alpha (m_0+a) + a'/r$ is separated from the others by at least 
$1/q r R - 1/Q$; we obtain that (\ref{eq:douben}) is at most
$\sum_{W'<p\leq W} (\log p)^2$. It remains to estimate the sum in the
first line of (\ref{eq:malheur}).
(We are following here a procedure analogous to that used in
 \cite{MR0374060} to prove the Brun-Titchmarsh theorem.)

Assume first that $q\leq W/13.5$. Set
\begin{equation}\label{eq:gorg}
R = \left(\sigma \frac{W}{q}\right)^{1/2} ,
\end{equation}
where $\sigma = 1/2 e^{2\cdot 0.25068} = 0.30285\dotsc$.
It is clear that $q R^2<Q$,
$q<W'$ and $R\geq 2$. Moreover, for $r\leq R$,
\[\frac{1}{Q} \leq \frac{1}{3.5 W}\leq \frac{\sigma}{3.5} \frac{1}{\sigma W} =
\frac{\sigma}{3.5} \frac{1}{q R^2} \leq \frac{\sigma/3.5}{q r R}.\]
Hence
\[\begin{aligned}
\frac{W}{2} + \frac{3}{2} \left(\frac{1}{q r R} - \frac{1}{Q}\right)^{-1}
&\leq \frac{W}{2} + \frac{3}{2} \frac{ q r R}{1-\sigma/3.5} = 
\frac{W}{2} + \frac{3 r}{2 \left(1-\frac{\sigma}{3.5}\right) R} \cdot
2\sigma \frac{W}{2}\\
&= \frac{W}{2} \left(1 + \frac{3\sigma}{1-\sigma/3.5} \frac{r W}{R}\right) < 
\frac{W}{2} \left(1 + \frac{r W}{R}\right)\end{aligned}\]
and so
\[\begin{aligned}
\mathop{\sum_{r\leq R}}_{\gcd(r,q)=1} &\left(\frac{W}{2} + \frac{3}{2} \left(\frac{1}{q r R} - \frac{1}{Q}\right)^{-1}\right)^{-1} 
\frac{\mu(r)^2}{\phi(r)}\\ 
&\geq \frac{2}{W} \mathop{\sum_{r\leq R}}_{\gcd(r,q)=1} (1 + r R^{-1})^{-1} 
\frac{\mu(r)^2}{\phi(r)}
\geq \frac{2}{W} \frac{\phi(q)}{q} 
\sum_{r\leq R} (1 + r R^{-1})^{-1} 
\frac{\mu(r)^2}{\phi(r)}
.\end{aligned}\]
For $R\geq 2$, 
\[\sum_{r\leq R} (1 + r R^{-1})^{-1} 
\frac{\mu(r)^2}{\phi(r)} > \log R + 0.25068;\]
this is true for $R\geq 100$ by \cite[Lemma 8]{MR0374060} 
and easily verifiable numerically for $2\leq R<100$. (It suffices to verify
this for $R$ integer with $r<R$ instead of $r\leq R$, as that is the worst
case.) 

Now
\[\log R = \frac{1}{2} \left(\log \frac{W}{2 q} + \log 2\sigma\right)
= \frac{1}{2} \log \frac{W}{2 q} - 0.25068 .\]
Hence
\[\sum_{r\leq R} (1 + r R^{-1})^{-1} 
\frac{\mu(r)^2}{\phi(r)} > \frac{1}{2} \log \frac{W}{2 q}\] and
the statement follows. 

Now consider the case $q > W/13.5$. If  $q$ is even, then, in this range,
inequality (\ref{eq:pokor1}) is always better than
(\ref{eq:pokor2}), and so we are done.
 Assume, then, that $W/13.5 <q \leq W/2$ and $q$ is odd. We
set $R=2$; clearly $q R^2 < W\leq Q$ and $q<W/2\leq W'$, and so this
choice of $R$ is valid. It remains to check that
\[\frac{1}{\frac{W}{2} + \frac{3}{2} 
\left(\frac{1}{2q} - \frac{1}{Q}\right)^{-1}} + 
\frac{1}{\frac{W}{2} + \frac{3}{2} 
\left(\frac{1}{4q} - \frac{1}{Q}\right)^{-1}} \geq \frac{1}{W} \log
\frac{W}{2q}.\]
This follows because
\[\frac{1}{\frac{1}{2} + \frac{3}{2} 
\left(\frac{t}{2} - \frac{1}{3.5}\right)^{-1}} + 
\frac{1}{\frac{1}{2} + \frac{3}{2} 
\left(\frac{t}{4} - \frac{1}{3.5}\right)^{-1}} \geq \log \frac{t}{2}\]
for all $2\leq t\leq 13.5$.

\end{proof} 
We need a version of Lemma \ref{lem:kastor1} with $m$ restricted to the
odd numbers.
\begin{lem}\label{lem:kastor2}
Let $W\geq 1$, $W'\geq W/2$. Let 
$2\alpha = a/q + O^*(1/q Q)$, $q\leq Q$.  Then
\begin{equation}\label{eq:pokor1b}\begin{aligned}
\mathop{\sum_{A_0<m\leq A_1}}_{\text{$m$ odd}}
 &\left|\sum_{W'<p\leq W} (\log p) e(\alpha m p)\right|^2
\\ &\leq 
\left\lceil \frac{A_1-A_0}{\min(2q,Q)} \right\rceil\cdot
(W-W'+2q) \sum_{W'<p\leq W} (\log p)^2.\end{aligned}\end{equation}
If $q<W/2$ and $Q\geq 3.5 W$, the following bound also holds:
\begin{equation}\label{eq:pokor2b}\begin{aligned}
\mathop{\sum_{A_0<m\leq A_1}}_{\text{$m$ odd}}
 &\left|\sum_{W'<p\leq W} (\log p) e(\alpha m p)\right|^2\\
&\leq \left\lceil \frac{A_1-A_0}{2 q}\right\rceil\cdot
\frac{q}{\phi(q)} \frac{W}{\log(W/2q)}
\cdot \sum_{W'<p\leq W} (\log p)^2.\end{aligned}\end{equation}
If $A_1-A_0\leq 2\varrho q$ and $q\leq \rho Q$, $\varrho,\rho \in \lbrack 0,1
\rbrack$, the following bound also holds:
\begin{equation}\label{eq:zeromb}\begin{aligned}
\sum_{A_0<m\leq A_1} &\left|\sum_{W'<p\leq W} (\log p) e(\alpha m p)\right|^2\\
&\leq (W-W'+q/(1-\varrho \rho)) \sum_{W'<p\leq W} (\log p)^2.\end{aligned}
\end{equation}
\end{lem}
\begin{proof}
We follow the proof of Lemma \ref{lem:kastor1}, noting the differences.
Let $k=\min(q,\lceil Q/2\rceil)\geq \lceil q/2\rceil$, just as before.
We split $(A_0,A_1\rbrack$ into $\lceil (A_1-A_0)/k\rceil$ blocks of at most
$2k$ consecutive integers; any such block contains at most $k$ odd numbers.
For odd $m$, $m'$ in such a block, $\alpha m$ and $\alpha m'$ are
separated by a distance of 
\[|\{\alpha (m-m')\}| = \left|\left\{2\alpha \frac{m-m'}{2}\right\}\right|
= |\{(a/q) k\}| - O^*(k/qQ) \geq 1/2q.\]

We obtain (\ref{eq:pokor1b}) and (\ref{eq:zeromb})
 just as we obtained (\ref{eq:pokor1}) and (\ref{eq:zerom}) before.
To obtain (\ref{eq:pokor2b}), proceed again as before, noting that
the angles we are working with can be labelled as $\alpha (m_0+2a)$,
$0\leq a < q$.
\end{proof}

The idea now (for large $\delta$) is that, if $\delta$ is not negligible, then, 
as $m$ increases, $\alpha m$ loops around the circle $\mathbb{R}/\mathbb{Z}$
roughly repeats itself every $q$ steps -- but with a slight displacement.
This displacement gives rise to a configuration to which Lemma \ref{lem:ogor}
is applicable. 
\begin{prop}\label{prop:kraken}
Let $x\geq W\geq 1$, $W'\geq W/2$, $U'\geq x/2W$. Let 
$Q\geq
3.5W$.
Let $2 \alpha = a/q + 
\delta/x$, $\gcd(a,q)=1$, $|\delta/x|\leq 1/qQ$, $q\leq Q$. 
Let $S_2(U',W',W)$ be as in (\ref{eq:honi}) with $v=2$.

For $q\leq \rho Q$, where $\rho\in \lbrack 0,1\rbrack$,
\begin{equation}\label{eq:garn1b}\begin{aligned}
S_2(U',W',W) &\leq
\left(\max(1,2\rho) \left(\frac{x}{8q} + \frac{x}{2W}\right) + \frac{W}{2} + 2 q\right)
\cdot  \sum_{W'<p\leq W} (\log p)^2\\
\end{aligned}
\end{equation}

If $q< W/2$,
\begin{equation}\label{eq:garn1a}
S_2(U',W',W)\leq 
\left(
\frac{x}{4 \phi(q)} \frac{1}{\log(W/2q)}
 + \frac{q}{\phi(q)} \frac{W}{\log(W/2q)} \right)\cdot
 \sum_{W'<p\leq W} (\log p)^2 .\end{equation}


If $W> x/4 q$,
the following bound also holds:
\begin{equation}\label{eq:gargamel}
S_2(U',W',W) \leq \left(\frac{W}{2} + \frac{q}{1-x/4Wq}\right)\sum_{W'< p\leq W} (\log p)^2.
\end{equation}

If $\delta\ne 0$ and $x/4W + q \leq x/|\delta| q$, 
\begin{equation}\label{eq:procida2}
S_2(U',W',W) \leq \min\left(1,\frac{2 q/\phi(q)}{\log\left(\frac{x}{|\delta q|} 
\left(q + \frac{x}{4W}\right)^{-1}
\right)}\right) \cdot 
\left(\frac{x}{|\delta q|} + \frac{W}{2}\right)
\sum_{W'<p\leq W} (\log p)^2 .
\end{equation}

Lastly, if $\delta\ne 0$ and $q\leq \rho Q$, where $\rho\in \lbrack 0,1)$,
\begin{equation}\label{eq:procida3}
 S_2(U',W',W) \leq 
\left(\frac{x}{|\delta q|} + \frac{W}{2} + \frac{x}{8 (1-\rho) Q} + 
\frac{x}{4 (1-\rho) W}\right) \sum_{W'<p\leq W} (\log p)^2 .\end{equation}
\end{prop}

The trivial bound would be in the order
of \[S_2(U',W',W) = (x/2\log x) \sum_{W'<p\leq W} (\log p)^2.\]
In practice, (\ref{eq:gargamel}) gets applied when $W\geq x/q$.

\begin{proof}
Let us first prove statements (\ref{eq:garn1a}) and (\ref{eq:garn1b}), which
do not involve $\delta$. Assume first $q\leq W/2$. Then, by
(\ref{eq:pokor2b}) with $A_0 = U'$, $A_1 = x/W$,
\[S_2(U',W',W)\leq \left(\frac{x/W - U'}{2 q} + 1\right)
\frac{q}{\phi(q)} \frac{W}{\log(W/2q)} \sum_{W'<p\leq W} (\log p)^2.\]
Clearly $(x/W-U') W\leq (x/2W) \cdot W = x/2$. 
 Thus (\ref{eq:garn1a}) holds.

Assume now that $q\leq \rho Q$. Apply (\ref{eq:pokor1b}) with 
$A_0 = U'$, $A_1 = x/W$. Then
\[S_2(U',W',W)\leq \left(\frac{x/W - U'}{q\cdot \min(2,\rho^{-1})} + 1\right)
(W - W' + 2q) 
\sum_{W'<p\leq W} (\log p)^2.\]
Now
\[\begin{aligned}
&\left(\frac{x/W - U'}{q\cdot \min(2,\rho^{-1})} + 1\right)\cdot 
(W - W' + 2q) \\&\leq \left(\frac{x}{W}-U'\right) 
\frac{W-W'}{q \min(2,\rho^{-1})} + 
\max(1,2\rho) \left(\frac{x}{W} - U'\right) + W/2 + 2q 
\\ &\leq \frac{x/4}{q \min(2,\rho^{-1})}
 + \max(1,2\rho) \frac{x}{2 W} + W/2 + 2 q .
\end{aligned}\]
This implies (\ref{eq:garn1b}).

If $W> x/4 q$, apply (\ref{eq:zerom}) with $\varrho = x/4Wq$, $\rho=1$.
This yields (\ref{eq:gargamel}).

Assume now that $\delta\ne 0$ and $x/4W+q\leq x/|\delta q|$.
Let $Q' = x/|\delta q|$. For any $m_1$, $m_2$ with $x/2W< m_1,m_2\leq x/W$,
we have
$|m_1-m_2|\leq x/2W \leq 2 (Q'-q)$, and so
\begin{equation}\label{eq:radclas}
\left|\frac{m_1-m_2}{2} \cdot \delta/x + q\delta/x\right|\leq Q'|\delta|/x = \frac{1}{q}.\end{equation}
The conditions of Lemma \ref{lem:ogor}
are thus 
fulfilled with $\upsilon = (x/4W)\cdot |\delta|/x$ and $\nu=|\delta q|/x$.
We obtain that $S_2(U',W',W)$ is at most
\[
 \min\left(1, \frac{2 q}{\phi(q)} \frac{1}{\log\left((q (\nu+\upsilon))^{-1}
\right)}\right)
 \left(W-W'+\nu^{-1}\right) \sum_{W'<p\leq W} (\log p)^2.\]
Here $W -W' + \nu^{-1} = W-W' + x/|q\delta|\leq W/2 + x/|q\delta|$
and \[(q (\nu+\upsilon))^{-1} = 
\left(q \frac{|\delta|}{x}\right)^{-1}
\left(q + \frac{x}{4W}\right)^{-1} .\]

Lastly, assume $\delta\ne 0$ and $q\leq \rho Q$. We let
$Q' = x/|\delta q| \geq Q$ again, and we split the range $U'<m\leq x/W$
into intervals of length $2(Q'-q)$, so that (\ref{eq:radclas}) still
holds within each interval. We apply Lemma \ref{lem:ogor}
with $\upsilon = (Q'-q)\cdot |\delta|/x$ and $\nu=|\delta q|/x$.
We obtain that 
$S_2(U',W',W)$ is at most
\[
 \left(1 + \frac{x/W-U}{2(Q'-q)}\right)
 \left(W-W'+\nu^{-1}\right) \sum_{W'<p\leq W} (\log p)^2.\]
Here $W-W'+\nu^{-1} \leq W/2 + x/q|\delta|$ as before. Moreover,
\[\begin{aligned}
\left(\frac{W}{2} + \frac{x}{q|\delta|}\right) \left(1 + \frac{x/W-U}{2(Q'-q)}\right)
&\leq \left(\frac{W}{2} + Q'\right) \left(1 + 
\frac{x/2W}{2 (1-\rho) Q'}\right)\\
&\leq \frac{W}{2} + Q' + \frac{x}{8 (1-\rho) Q'} + 
\frac{x}{4 W (1-\rho)}\\
&\leq \frac{x}{|\delta q|} + \frac{W}{2} + \frac{x}{8 (1-\rho) Q} + 
\frac{x}{4 (1-\rho) W}.
\end{aligned}\]
Hence (\ref{eq:procida3}) holds.
\end{proof}

\section{Totals}\label{sec:osval}

Let $x$ be given. We will choose $U$, $V$, $W$ later; assume from the start that
 $2\cdot 10^6 \leq V< x/4$ and $UV\leq x$. 
Starting in section \ref{subs:totcho}, we will also assume that
$x\geq x_0 = 10^{25}$.

Let $\alpha \in \mathbb{R}/\mathbb{Z}$ be given. We choose an approximation
$2 \alpha = a/q + \delta/x$, $\gcd(a,q)=1$, $q\leq Q$, $|\delta/x|\leq 1/qQ$.
We assume $Q\geq \max(16,2 \sqrt{x})$ and $Q\geq \max(2 U,x/U)$.
Let $S_{I,1}$, $S_{I,2}$, $S_{II}$, $S_0$ be as in (\ref{eq:nielsen}), with
the smoothing function $\eta=\eta_2$ as in (\ref{eq:eqeta}).

The term $S_0$ is $0$ because $V<x/4$ and $\eta_2$ is supported on
$\lbrack -1/4,1\rbrack$.  We set $v=2$.

\subsection{Contributions of different types}\label{subs:putmal}
\subsubsection{Type I terms: $S_{I,1}$.}\label{subs:renzo}
The term $S_{I,1}$ can be handled directly by Lemma \ref{lem:bostb1},
with $\rho_0=4$ and $D = U$. (Condition (\ref{eq:puella}) is 
valid thanks to (\ref{eq:cloclo}).)
Since $U\leq Q/2$,
 the contribution of $S_{I,1}$ gets bounded by (\ref{eq:cupcake2}) and
(\ref{eq:kuche2}): the absolute value of $S_{I,1}$ is at most
\begin{equation}\label{eq:cosI1}\begin{aligned}&\frac{x}{q} 
 \min\left(1,\frac{c_0/\delta^2}{(2\pi)^2}\right)
\left|\mathop{\sum_{m\leq \frac{U}{q}}}_{\gcd(m,q)=1} \frac{\mu(m)}{m} 
\log \frac{x}{m q}\right| + \frac{x}{q}
|\widehat{\log \cdot \eta}(-\delta)| \left|\mathop{\sum_{m\leq \frac{U}{q}}}_{\gcd(m,q)=1} \frac{\mu(m)}{m}\right|\\
&+ \frac{2 \sqrt{c_0 c_1}}{\pi}  \left( 
U \log \frac{e x}{U} + \sqrt{3} q \log \frac{q}{c_2} +
\frac{q}{2} \log \frac{q}{c_2} \log^+ \frac{2 U}{q}\right)
+ \frac{3 c_1}{2}
 \frac{x}{q} \log \frac{q}{c_2} \log^+ \frac{U}{\frac{c_2 x}{q}}\\
&+ \frac{3 c_1}{2} \sqrt{\frac{2 x}{c_2}} \log \frac{2 x}{c_2} 
+ 
\left(\frac{c_0}{2} - \frac{2 c_0}{\pi^2}
\right)
\left(\frac{U^2}{4 q x} \log \frac{e^{1/2} x}{
U} + \frac{1}{e} \right) 
\\ &+ \frac{2 |\eta'|_1}{\pi} q \max\left(1, \log \frac{c_0 e^3 q^2}{4 \pi |\eta'|_1 x}\right) \log x,
\end{aligned}\end{equation}
where $c_0 = 31.521$ (by Lemma \ref{lem:octet}), 
$c_1 =  1.0000028 > 1 + (8 \log 2)/V \geq 1 + (8 \log 2)/(x/U)$,
$c_2 = 6 \pi/5\sqrt{c_0}
= 0.67147\dotsc$. By (\ref{eq:malcros}),
(\ref{eq:koasl}) and Lemma \ref{lem:marengo},
\[|\widehat{\log \cdot \eta}(-\delta)| \leq
\min\left(2-\log 4,\frac{24 \log 2}{\pi^2 \delta^2}\right).\]

By (\ref{eq:grara}), (\ref{eq:ronsard}) and (\ref{eq:meproz}),
the first line of (\ref{eq:cosI1}) is at most
\[\begin{aligned}
&\frac{x}{q} \min\left(1,\frac{c_0'}{\delta^2}\right)
\left(\min\left(\frac{4}{5} \frac{q/\phi(q)}{\log^+ 
\frac{U}{q^2}}, 1\right) \log \frac{x}{U} + 1.00303 \frac{q}{\phi(q)}\right)
\\+ &\frac{x}{q} \min\left(2 - \log 4,\frac{c_0''}{\delta^2}\right)
\min\left(\frac{4}{5} \frac{q/\phi(q)}{\log^+ 
\frac{U}{q^2}}, 1\right),
\end{aligned}\]
where $c_0' = 0.798437 > c_0/(2\pi)^2$, $c_0'' = 1.685532$. Clearly 
$c_0''/c_0 > 1 > 2 - \log 4$. 

Taking derivatives, we see that $t\mapsto (t/2) \log(t/c_2) \log^+ 2 U/t$
takes its maximum (for $t\in \lbrack 1,2 U\rbrack$) 
when $\log(t/c_2) \log^+ 2 U/t =
\log t/c_2 - \log^+ 2 U /t$; since $t\to \log t/c_2 - \log^+ 2U/t$
is increasing on $\lbrack 1, 2U\rbrack$, we conclude that
\[\frac{q}{2} \log \frac{q}{c_2} \log^+ \frac{2 U}{q} \leq U \log \frac{2 U}{c_2}.\]
Similarly, $t\mapsto t \log(x/t) \log^+(U/t)$ takes its maximum at a point
$t\in \lbrack 0, U$ for which
$\log(x/t) \log^+(U/t) = \log(x/t) + \log^+(U/t)$, and so
\[\frac{x}{q} \log \frac{q}{c_2} \log^+ \frac{U}{\frac{c_2 x}{q}} \leq
\frac{U}{c_2} (\log x + \log U).\]

We conclude that
\begin{equation}\label{eq:lavapie}\begin{aligned}
|S_{I,1}| &\leq \frac{x}{q} \min\left(1, \frac{c_0'}{\delta^2}\right)
\left(\min\left(\frac{4 q/\phi(q)}{5\log^+ 
\frac{U}{q^2}}, 1\right) \left(\log \frac{x}{U}
+ c_{3,I}\right)
 + c_{4,I} \frac{q}{\phi(q)} \right)\\
&+ \left(c_{7,I}
\log \frac{q}{c_2} + c_{8,I} \log x \max\left(
1,  \log \frac{c_{11,I} q^2}{x}\right) \right) q 
+
c_{10,I} \frac{U^2}{4 q x} \log \frac{e^{1/2} x}{
U}\\
&+ \left(c_{5,I} \log \frac{2 U}{c_2} +
c_{6,I} \log x U\right) U 
+ c_{9,I} \sqrt{x} \log \frac{2 x}{c_2} 
+ \frac{c_{10,I}}{e}
,\end{aligned}\end{equation}
where $c_2$ and $c_0'$ are as above, $c_{3,I} = 2.11104 > c_0''/c_0'$, 
$c_{4,I} = 1.00303$, $c_{5,I}= 3.57422 > 2 \sqrt{c_0 c_1}/\pi$,
$c_{6,I} = 2.23389 > 3 c_1/2 c_2$, $c_{7,I} = 6.19072 > 2 \sqrt{3 c_0 c_1}/\pi$, 
$c_{8,I} = 3.53017 > 2 (8 \log 2)/\pi$, $c_{9,I} = 2.58877 > 
3 \sqrt{2} c_1/2 \sqrt{c_2}$, $c_{10,I} = 
9.37301 > c_0 (1/2 - 2/\pi^2)$ and $c_{11,I} = 9.0857
> c_0 e^3/(4 \pi \cdot 8 \log 2)$.

\subsubsection{Type I terms: $S_{I,2}$.}
{\em The case $q\leq Q/V$.}

If $q\leq Q/V$, then, for $v\leq V$,
\[2 v\alpha = \frac{v a}{q} + O^*\left(\frac{v}{Q q}\right) = 
\frac{v a}{q} + O^*\left(\frac{1}{q^2}\right) ,\]
and so $va/q$ is a valid approximation to $2 v\alpha$. (Here we are using
$v$ to label an integer variable bounded above by $v\leq V$; we no longer
need $v$ to label the quantity in (\ref{eq:joroy}), since that
has been set equal to the constant $2$.)
Moreover, for
$Q_v = Q/v$, we see that $2 v\alpha = (va/q) + O^*(1/q Q_v)$.
 If $\alpha = a/q + 
\delta/x$, then $v\alpha = v a/q + \delta/(x/v)$.
Now
\begin{equation}\label{eq:jotoco}
S_{I,2} = \mathop{\sum_{v\leq V}}_{\text{$v$ odd}} \Lambda(v) 
\mathop{\sum_{m\leq U}}_{\text{$m$ odd}} \mu(m) \mathop{\sum_n}_{
\text{$n$ odd}}
e((v\alpha) \cdot m n) \eta(mn/(x/v)).\end{equation}
 We can thus
estimate $S_{I,2}$ by applying Lemma \ref{lem:bosta2}
to each inner double sum in (\ref{eq:jotoco}).
We obtain that, if $|\delta|\leq 1/2 c_2$, where 
$c_2 = 6\pi/5\sqrt{c_0}$ and $c_0 = 31.521$, then $|S_{I,2}|$ is at most
\begin{equation}\label{eq:putbarat}
\sum_{v\leq V} \Lambda(v) \left(
\frac{x/v}{2 q_v} \min\left(1,\frac{c_0}{(\pi \delta)^2}
\right) \left|\mathop{\sum_{m\leq M_v/q}}_{\gcd(m,2 q)=1} \frac{\mu(m)}{m}\right|
+ \frac{c_{10,I} q}{4 x/v} \left(\frac{U}{q_v} + 1\right)^2
\right)
\end{equation}
plus
\begin{equation}\label{eq:douze}\begin{aligned}
&\sum_{v\leq V} \Lambda(v) 
\left( \frac{2 \sqrt{c_0 c_1}}{\pi} U + 
\frac{3 c_1}{2} \frac{x}{v q_v}
\log^+ \frac{U}{\frac{c_2 x}{v q_v}} + \frac{\sqrt{c_0 c_1}}{\pi}
q_v \log^+ \frac{U}{q_v/2}\right)\\
+ &\sum_{v\leq V} \Lambda(v) \left(c_{8,I}
\max\left(
 \log \frac{c_{11,I} q_v^2}{x/v}, 1\right) q_v 
+ \left(\frac{2 \sqrt{3 c_0 c_1}}{\pi} + \frac{3 c_1}{2 c_2} + 
\frac{55 c_0 c_2}{6 \pi^2} \right) q_v
\right),
\end{aligned}\end{equation}
where $q_v = q/\gcd(q,v)$, $M_v \in \lbrack \min(Q/2v,U),U\rbrack$ 
and $c_1 = 1.0000028$;
if $|\delta|\geq 1/2c_2$, then $|S_{I,2}|$ is at most (\ref{eq:putbarat})
plus
\begin{equation}\label{eq:cheaslu}\begin{aligned}
&\sum_{v\leq V} \Lambda(v) 
\left(\frac{\sqrt{c_0 c_1}}{\pi/2} U  +
 \frac{3 c_1}{2}
 \left(2 + \frac{(1+\epsilon)}{\epsilon} \log^+ \frac{2 U}{
\frac{x/v}{|\delta| q_v}}
\right) \frac{x/v}{Q/v} + \frac{35 c_0 c_2}{3 \pi^2} q_v\right)\\
&+\sum_{v\leq V} \Lambda(v) \frac{\sqrt{c_0 c_1}}{\pi/2} 
 (1+\epsilon) \min\left(\left\lfloor\frac{x/v}{|\delta| q_v}\right\rfloor + 1, 2 U\right)
 \left(\sqrt{3+2\epsilon} +
 \frac{\log^+ \frac{2 U}{\left\lfloor \frac{x/v}{|\delta| q_v}\right\rfloor + 
1}}{2}\right)
\end{aligned}\end{equation}

Write $S_V = \sum_{v\leq V} \Lambda(v)/(v q_v)$.
By (\ref{eq:rala}),
\begin{equation}\label{eq:avamys}\begin{aligned}
S_V &\leq \sum_{v\leq V} \frac{\Lambda(v)}{v q} +
\mathop{\sum_{v\leq V}}_{\gcd(v,q)>1} \frac{\Lambda(v)}{v} \left(\frac{\gcd(q,v)}{q} - 
\frac{1}{q}\right)\\ &\leq \frac{\log V}{q} +
\frac{1}{q} \sum_{p|q} (\log p) \left(v_p(q) + \mathop{\sum_{\alpha \geq 1}}_{
p^{\alpha+v_p(q)} \leq V} \frac{1}{p^\alpha} - \mathop{\sum_{\alpha \geq 1}}_{
p^{\alpha}\leq V} \frac{1}{p^\alpha} \right)\\ &\leq \frac{\log V}{q} + \frac{1}{q}
\sum_{p|q} (\log p) v_p(q) = \frac{\log V q}{q}.
\end{aligned}\end{equation}
This helps us to estimate (\ref{eq:putbarat}). We could also use this
to estimate the second term in the first line of (\ref{eq:douze}),
but, for that purpose, it will actually be wiser to use the simpler bound
\begin{equation}\label{eq:qotro}\sum_{v\leq V} \Lambda(v) 
\frac{x}{v q_v} \log^+ \frac{U}{\frac{c_2 x}{v q_v}}
\leq \sum_{v\leq V} \Lambda(v) \frac{U/c_2}{e} \leq
\frac{1.0004}{e c_2} UV 
\end{equation}
(by (\ref{eq:trado1}) and the fact that $t \log^+ A/t$ takes its maximum at
$t=A/e$).



We bound the sum over $m$ in (\ref{eq:putbarat}) by (\ref{eq:grara}) and
(\ref{eq:ronsard}). To bound the terms involving $(U/q_v+1)^2$, we use
\[\begin{aligned}
\sum_{v\leq V} \Lambda(v) v &\leq 0.5004 V^2 \;\;\;\;\;\;\text{(by
(\ref{eq:nicro})),}\\
\sum_{v\leq V} \Lambda(v) v \gcd(v,q)^j &\leq \sum_{v\leq V} \Lambda(v) v +
V \mathop{\sum_{v\leq V}}_{\gcd(v,q)\ne 1} \Lambda(v) \gcd(v,q)^j
\end{aligned},\]
\[\begin{aligned}
\mathop{\sum_{v\leq V}}_{\gcd(v,q)\ne 1} \Lambda(v) \gcd(v,q) &\leq
 \sum_{p|q} (\log p) \sum_{1\leq \alpha \leq \log_p V} p^{v_p(q)}
\leq \sum_{p|q} (\log p) \frac{\log V}{\log p} p^{v_p(q)} \\ &\leq (\log V)
\sum_{p|q} p^{v_p(q)} \leq q \log V
\end{aligned}\] and
\[\begin{aligned}
\mathop{\sum_{v\leq V}}_{\gcd(v,q)\ne 1} \Lambda(v) \gcd(v,q)^2 &\leq
\sum_{p|q} (\log p) \sum_{1\leq \alpha\leq \log_p V} p^{v_p(q)+\alpha}\\&\leq
\sum_{p|q} (\log p) \cdot 2 p^{v_p(q)} \cdot p^{\log_p V} \leq 2 q V \log q .
\end{aligned}\]
 Using (\ref{eq:trado1}) and
(\ref{eq:avamys}) as well, we conclude that (\ref{eq:putbarat}) is at most
\[\begin{aligned}
\frac{x}{2 q} &\min\left(1, \frac{c_0}{(\pi \delta)^2}\right)
\min\left(\frac{4}{5} \frac{q/\phi(q)}{\log^+ \frac{\min(Q/2V,U)}{2q}},1\right) \log Vq\\
&+ \frac{c_{10,I}}{4 x} \left(0.5004 V^2 q \left(\frac{U}{q}+1\right)^2 +
2 U V q \log V + 2 U^2 V \log V\right).\end{aligned}\]

Assume $Q \leq 2 U V/e$.
Using (\ref{eq:trado1}), (\ref{eq:qotro}), (\ref{eq:charol})
 and the inequality $vq \leq V q\leq Q$ (which implies $q/2\leq U/e$), 
we see that (\ref{eq:douze})
is at most
\[\begin{aligned}
&1.0004 \left(\left(\frac{2\sqrt{c_0 c_1}}{\pi} + \frac{3 c_1}{2 e c_2}
\right) UV + 
\frac{\sqrt{c_0 c_1}}{\pi} Q \log \frac{U}{q/2}\right)
\\ + &\left(c_{5,I_2} 
\max\left(
 \log \frac{c_{11,I} q^2}{x}, 2\right)
 + c_{6,I_2}\right) Q,\end{aligned}\]
where $c_{5,I_2} = 3.53312 > 1.0004 \cdot c_{8,I}$ 
and \[c_{6,I_2} = \frac{2 \sqrt{3 c_0 c_1}}{\pi} + \frac{3 c_1}{2 c_2} + 
\frac{55 c_0 c_2}{6 \pi^2} .\]

The expressions in (\ref{eq:cheaslu}) get estimated similarly. In particular,
\[\begin{aligned}
\sum_{v\leq V} &\Lambda(v) \min\left(\left\lfloor
\frac{x/v}{|\delta| q_v}\right\rfloor + 1, 2 U\right) 
\cdot \frac{1}{2}
\log^+ \frac{2 U}{\left\lfloor \frac{x/v}{|\delta| q_v}\right\rfloor + 1}\\
&\leq \sum_{v\leq V} \Lambda(v) \max_{t>0} t \log^+ \frac{U}{t}
 \leq
\sum_{v\leq V} \Lambda(v)  \frac{U}{e} = \frac{1.0004}{e} UV,
\end{aligned}\]
but
\[\begin{aligned}
\sum_{v\leq V} &\Lambda(v) \min\left(\left\lfloor
\frac{x/v}{|\delta| q_v}\right\rfloor + 1, 2 U\right) 
 \leq \sum_{v\leq \frac{x}{2 U |\delta| q}} \Lambda(v) \cdot 2 U \\ &+ 
\mathop{\sum_{\frac{x}{2 U |\delta| q} < v \leq V}}_{\gcd(v,q)=1} \Lambda(v)
\frac{x/|\delta|}{
v q} + \sum_{v\leq V} \Lambda(v) + \mathop{\sum_{v\leq V}}_{\gcd(v,q)\ne 1}
\Lambda(v) \frac{x/|\delta|}{v} \left(\frac{1}{q_v} - \frac{1}{q}\right)\\
&\leq 1.03883 \frac{x}{|\delta| q} + \frac{x}{|\delta| q} 
\max\left(\log V - \log \frac{x}{2 U |\delta| q} + \log \frac{3}{\sqrt{2}},0\right)\\
&+ V + \frac{x}{|\delta|} \frac{1}{q} \sum_{p|q} (\log p) v_p(q)\\
&\leq \frac{x}{|\delta| q} \left(1.03883+ \log q + 
 \log^+ \frac{6 U V |\delta| q}{\sqrt{2} x}\right) + 1.0004 V 
\end{aligned}\]
by (\ref{eq:rala}), (\ref{eq:ralobio}), (\ref{eq:trado1})
 and (\ref{eq:trado2}); we are proceeding
much as in (\ref{eq:avamys}).

If $|\delta|\leq 1/2 c_2$, then, assuming $Q \leq 2 U V/e$,
we conclude that
$|S_{I,2}|$ is at most
\begin{equation}\label{eq:chusan}\begin{aligned}
&\frac{x}{2 \phi(q)} \min\left(1, \frac{c_0}{(\pi \delta)^2}\right)
\min\left(\frac{4/5}{\log^+ \frac{Q}{4 V q^2}},1\right) \log V q \\ &+
c_{8,I_2} \frac{x}{q} \left(\frac{U V}{x}\right)^2
\left(1 + \frac{q}{U}\right)^2 + \frac{c_{10,I}}{2} \left(\frac{UV}{x}
q \log V + \frac{U^2 V}{x} \log V\right)
\end{aligned}\end{equation}
plus
\begin{equation}\label{eq:fausto}
(c_{4,I_2} + c_{9,I_2}) UV + (c_{10,I_2} \log \frac{U}{q} +
c_{5,I_2} \max\left(
 \log \frac{c_{11,I} q^2}{x}, 2\right) 
 + c_{12,I_2})\cdot  Q,
\end{equation}
where \[\begin{aligned}
c_{4,I_2} &= 3.57422 >
2 \sqrt{c_0 c_1}/\pi,\\ 
c_{5,I_2} &= 3.53312 > 1.0004 \cdot c_{8,I},\\
c_{8,I_2} &= 1.17257 > \frac{c_{10,I}}{4}\cdot 0.5004,\\
c_{9,I_2} &= 0.82214 > 3 c_1 \cdot 1.0004/2 e c_2,\\
c_{10,I_2} &= 1.78783 > 1.0004 \sqrt{c_0 c_1}/\pi,\\ 
c_{12,I_2} &= 28.26771 > c_{6,I_2} + c_{10,I_2} \log 2.\end{aligned}\]
If $|\delta|\geq 1/2 c_2$, then
$|S_{I,2}|$ is at most
(\ref{eq:chusan}) plus 
\begin{equation}\label{eq:magus}\begin{aligned}
&(c_{4,I_2} + (1+\epsilon) c_{13,I_2}) U V +  c_\epsilon \left(
c_{14,I_2} \left(\log q + \log^+ \frac{6 UV |\delta| q}{\sqrt{2} x}\right)
+ c_{15,I_2}\right) \frac{x}{|\delta| q} \\ &+ 
c_{16,I_2} \left(2 + \frac{1+\epsilon}{\epsilon} \log^+ \frac{2 UV
|\delta| q}{x}\right) \frac{x}{Q/V} + c_{17,I_2} Q + 
c_\epsilon\cdot 
c_{18,I_2} V,
\end{aligned}\end{equation}
where 
\[\begin{aligned}
c_{13,I_2} &= 1.31541 > \frac{2\sqrt{c_0 c_1}}{\pi} \cdot \frac{1.0004}{e},\\
c_{14,I_2} &= 3.57422 > \frac{2 \sqrt{c_0 c_1}}{\pi} ,\\
c_{15,I_2} &= 3.71301 > \frac{2 \sqrt{c_0 c_1}}{\pi} \cdot
1.03883 ,\\ c_{16,I_2} &= 
1.50061 > 1.0004\cdot 3 c_1/2\\
c_{17,I_2} &= 25.0295 > 1.0004\cdot \frac{35 c_0 c_2}{3 \pi^2},\\
c_{18,I_2} &= 3.57565 > \frac{2 \sqrt{c_0 c_1}}{\pi} \cdot 1.0004,
\end{aligned}\]
and $c_\epsilon = (1+\epsilon) \sqrt{3+2\epsilon}$.
We recall that $c_2 = 6\pi/5\sqrt{c_0} = 0.67147\dotsc$.
We will choose $\epsilon \in (0,1)$ later.

{\em The case $q>Q/V$.}
We use Lemma \ref{lem:bogus} in this case.

\subsubsection{Type II terms.} As we showed in 
(\ref{eq:adoucit})--(\ref{eq:negli}), $S_{II}$ (given in (\ref{eq:adoucit}))
is at most
\begin{equation}\label{eq:secint}
4 \int_V^{x/U} \sqrt{S_1(U,W) \cdot S_2(U,V,W)} \frac{d W}{W} +
4 \int_V^{x/U} \sqrt{S_1(U,W) \cdot S_3(W)} \frac{d W}{W},\end{equation}
where $S_1$, $S_2$ and $S_3$ are as in (\ref{eq:costo}) and (\ref{eq:negli}).
We bounded $S_1$ in (\ref{eq:menson1}) and (\ref{eq:menson2}), $S_2$ in Prop. \ref{prop:kraken}
and $S_3$ in (\ref{eq:negli}).

We first recall our estimate for $S_1$. In the whole range 
$\lbrack V,x/U\rbrack$ for $W$, we know from (\ref{eq:menson1})
and (\ref{eq:menson2}) that
$S_1(U,W)$ is at most
\begin{equation}\label{eq:nehru2}\frac{2}{\pi^2} \frac{x}{W} + 
\kappa_{0} \zeta(3/2)^3 \frac{x}{W} \sqrt{\frac{x/WU}{U}} ,
\end{equation} where
\[\kappa_{0} = 1.27 .\]
(We recall we are working with $v=2$.)

We have better estimates for the constant in front in some parts of the
range; in what is usually the main part, 
 (\ref{eq:menson2}) and (\ref{eq:velib}) give us a constant of $0.15107$ 
instead of $2/\pi^2$. Note that $1.27 \zeta(3/2)^3 =
22.6417\dotsc$. We should choose $U$, $V$ so that
the first term dominates. For the while being, assume only
\begin{equation}\label{eq:curious}
U\geq 5\cdot 10^5 \frac{x}{VU};
\end{equation} then (\ref{eq:nehru2}) gives
\begin{equation}\label{eq:crudo}
S_1(U,W)\leq \kappa_{1} \frac{x}{W},\end{equation}
where \[\kappa_{1} =  
\frac{2}{\pi^2} + \frac{22.6418}{\sqrt{10^6/2}} \leq  0.2347.\]
This will suffice for our cruder estimates.

The second integral in (\ref{eq:secint}) is now easy to bound.
By (\ref{eq:negli}),
\[S_3(W)\leq 1.0171 x + 2.0341 W \leq 1.0172 x,\]
since $W \leq x/U \leq x/5 \cdot 10^5$. Hence
\[\begin{aligned}
4 \int_V^{x/U} \sqrt{S_1(U,W) \cdot S_3(W)}\; \frac{d W}{W}
&\leq 
4\int_V^{x/U} \sqrt{\kappa_{w,1} \frac{x}{W}\cdot 1.0172 x}\; \frac{dW}{W}\\
&\leq \kappa_{9} \frac{x}{\sqrt{V}},
\end{aligned}\]
where \[\kappa_{9} = 8 \cdot \sqrt{1.0172\cdot \kappa_{1}}
\leq 3.9086 .\]
(We are using the easy bound $\sqrt{a+b+c}\leq \sqrt{a}+\sqrt{b}+
\sqrt{c}$.)

Let us now examine $S_2$, which was bounded in Prop. \ref{prop:kraken}. 
Recall $W' = \max(V,W/2)$, $U' = \max(U,x/2W)$. 
Since $W'\geq W/2$ and $W\geq V\geq 117$, we can always bound
\begin{equation}\label{eq:immer}
\sum_{W'<p\leq W} (\log p)^2 \leq \frac{1}{2} W (\log W).\end{equation}
by (\ref{eq:kast}).

{\em Bounding $S_2$ for $\delta$ arbitrary.}
We set
\[W_0 = \min(\max(2 \theta q,V),x/U),\]
where $\theta\geq e$ is a parameter that will be set later.

For $V\leq W<W_0$, we use the bound (\ref{eq:garn1b}):
\[\begin{aligned}
S_2(U',W',W) &\leq 
\left(\max(1,2\rho) 
\left(\frac{x}{8 q} + \frac{x}{2 W}\right) + \frac{W}{2} + 2 q\right)
\cdot \frac{1}{2} W (\log W) \\
&\leq 
\max\left(\frac{1}{2},\rho\right) 
\left(\frac{W}{8 q} + \frac{1}{2}\right) x \log W + \frac{W^2 \log W}{4} + q W \log W
,\end{aligned}\]
where 
$\rho = q/Q$. 

If $W_0>V$, the contribution of the terms with $V\leq W<W_0$ 
to (\ref{eq:secint}) is 
(by \ref{eq:crudo}) bounded by
\begin{equation}\label{eq:rook}\begin{aligned}
4 &\int_V^{W_0} \sqrt{\kappa_{1}
\frac{x}{W} \left(\frac{\rho_0}{4}
\left(\frac{W}{4q} + 1\right) x \log W + \frac{W^2 \log W}{4} + q W \log W
\right)}\; \frac{dW}{W} \\ 
&\leq \frac{\kappa_{2}}{2} \sqrt{\rho_0}
x \int_V^{W_0} \frac{\sqrt{\log W}}{W^{3/2}} dW +
\frac{\kappa_{2}}{2} \sqrt{x} \int_V^{W_0}  \frac{\sqrt{\log W}}{W^{1/2}} dW \\ &+
\kappa_{2} \sqrt{\frac{\rho_0 x^2}{16 q} + q x}
 \int_V^{W_0}\frac{\sqrt{\log W}}{W} dW\\
&\leq \left(\kappa_{2} \sqrt{\rho_0}
\frac{x}{\sqrt{V}} + \kappa_{2} \sqrt{x W_0}\right) \sqrt{\log W_0}\\
 &+ \frac{2 \kappa_{2}}{3} \sqrt{\frac{\rho_0 x^2}{16 q} + q x}
 \left((\log W_0)^{3/2} - (\log V)^{3/2} \right) 
,\end{aligned}\end{equation}
where $\rho_0 = \max(1,2\rho)$
and
\[\kappa_{2} = 4 \sqrt{\kappa_{1}} \leq 
1.93768 .
\]

We now examine the terms with $W\geq W_0$. 
(If $\theta q > x/U$, then $W_0=U/x$, the contribution of the case is nil,
and the computations below can be ignored.)

We use (\ref{eq:garn1a}): 
\[
S_2(U',W',W)\leq 
\left(\frac{x}{4 \phi(q)} \frac{1}{\log(W/2q)}
 + \frac{q}{\phi(q)} \frac{W}{\log(W/2q)} \right)\cdot \frac{1}{2} W \log W.\]
By $\sqrt{a+b}\leq \sqrt{a}+\sqrt{b}$, we can take out the 
$q/\phi(q) \cdot W/\log(W/2q)$
term and estimate its contribution on its own; it is at most
\begin{equation}\label{eq:vivaldi}\begin{aligned}
4 \int_{W_0}^{x/U} &\sqrt{\kappa_{1} \frac{x}{W} \cdot \frac{q}{\phi(q)}
\cdot \frac{1}{2} W^2 \frac{\log W}{\log W/2q}}\; \frac{dW}{W}\\
&= \frac{\kappa_{2}}{\sqrt{2}} \sqrt{\frac{q}{\phi(q)}}
\int_{W_0}^{x/U} \sqrt{\frac{x \log W}{W \log W/2q}} dW\\
&\leq \frac{\kappa_{2}}{\sqrt{2}} \sqrt{\frac{q x}{\phi(q)}}
\int_{W_0}^{x/U} \frac{1}{\sqrt{W}} \left(1 + \sqrt{\frac{\log 2q}{\log W/2q}}
\right) dW
\end{aligned}\end{equation}
\\ 

Now
\[
\int_{W_0}^{x/U} 
 \frac{1}{\sqrt{W}} \sqrt{\frac{\log 2q}{\log W/2q}} dW \leq
\sqrt{2q\log 2q} \int_{\max(\theta,V/2q)}^{x/2Uq} \frac{1}{\sqrt{t \log t}} dt.\]
We bound this last integral somewhat crudely: for $T\geq e$,
\begin{equation}\label{eq:notung}
\int_e^T \frac{1}{\sqrt{t \log t}} dt \leq
2.3 \sqrt{\frac{T}{\log T}}
 \end{equation}
(by numerical work for $e\leq T\leq T_0$ and by comparison of derivatives
for $T>T_0$, where $T_0=e^{(1-2/2.3)^{-1}} = 2135.94\dotsc$). Since $\theta\geq e$,
this gives us that
\[\begin{aligned}
\int_{W_0}^{x/U} 
&\frac{1}{\sqrt{W}} \left(1+ 
 \sqrt{\frac{\log 2q}{\log W/2q}}\right) dW 
\\ &\leq 2 \sqrt{\frac{x}{U}} 
+ 2.3 \sqrt{2 q\log 2q}
 \cdot \sqrt{\frac{x/2Uq}{
\log x/2Uq}} ,
\end{aligned}\]
and so (\ref{eq:vivaldi}) is at most
\[\sqrt{2} \kappa_{2} 
\sqrt{\frac{q}{\phi(q)}} \left(1 + 1.15 \sqrt{\frac{\log 2q}{\log x/2Uq}}
\right) \frac{x}{\sqrt{U}}.\]

We are left with what will usually be the main term, viz.,
\begin{equation}\label{eq:soledad}
4\int_{W_0}^{x/U} \sqrt{S_1(U,W)\cdot \left(\frac{x}{8 \phi(q)}
\frac{\log W}{\log W/2q}\right) W} \frac{dW}{W},\end{equation}
which, by (\ref{eq:menson2}),  is at most $x/\sqrt{\phi(q)}$ times the 
integral of
\[
\frac{1}{W}
\sqrt{\left(2 H_2\left(\frac{x}{WU}\right) + \frac{\kappa_{4}}{2}
\sqrt{\frac{x/WU}{U}}\right) \frac{\log W}{\log W/2q}}
\]
for $W$ going from $W_0$ to $x/U$, where $H_2$ is as in (\ref{eq:palmiped})
and \[\kappa_{4} = 4 \kappa_{0} \zeta(3/2)^3 \leq 90.5671 .\]
By the arithmetic/geometric mean inequality, the integrand is at most
$1/W$ times
\begin{equation}\label{eq:dane}
\frac{\beta + \beta^{-1}\cdot 2 H_2(x/WU)}{2} + \frac{\beta^{-1}}{2} 
\frac{\kappa_{4}}{2} \sqrt{\frac{x/WU}{U}} 
+  \frac{\beta}{2} \frac{\log 2q}{\log W/2q}\end{equation}
for any $\beta>0$. We will choose $\beta$ later.

The first summand in (\ref{eq:dane})
gives what we can think of as the main or worst term
in the whole paper; let us compute it first. The integral is
\begin{equation}\label{eq:quartma}\begin{aligned}
\int_{W_0}^{x/U} \frac{\beta+\beta^{-1}\cdot 2 H_2(x/WU)}{2} \frac{dW}{W} &=
\int_1^{x/U W_0} \frac{\beta+\beta^{-1}\cdot 2 H_2(s)}{2} \frac{ds}{s}
\\ &\leq \left(\frac{\beta}{2} + \frac{\kappa_{6}}{4\beta}\right)
 \log \frac{x}{U W_0} 
\end{aligned}\end{equation}
by (\ref{eq:velib}),
where \[\kappa_{6} = 0.60428.\] 

Thus the main term is simply
\begin{equation}\label{eq:gogolo}
\left(\frac{\beta}{2} + \frac{\kappa_{6}}{4 \beta}\right) 
\frac{x}{\sqrt{\phi(q)}} \log \frac{x}{U W_0}.
\end{equation}

The integral of the second summand is at most
\[\begin{aligned}
\beta^{-1} \cdot \frac{\kappa_{4}}{4} \frac{\sqrt{x}}{U} 
\int_V^{x/U} \frac{d W}{W^{3/2}} 
&\leq \beta^{-1} \cdot \frac{\kappa_{4}}{2} \sqrt{\frac{x/UV}{U}} .
\end{aligned}\]
By (\ref{eq:curious}), this is at most 
\[\frac{\beta^{-1}}{\sqrt{2}}\cdot  10^{-3} \cdot \kappa_{4} \leq 
\beta^{-1} \kappa_{7}/2,\]
where \[\kappa_{7}= \frac{\sqrt{2} \kappa_{4}}{1000} \leq
 0.1281 .\]
Thus the contribution
of the second summand is at most
\[\frac{\beta^{-1} \kappa_{7}}{2}\cdot \frac{x}{\sqrt{\phi(q)}} .\]
The integral of the third summand in (\ref{eq:dane}) is
\begin{equation}\label{eq:resist}
\frac{\beta}{2} \int_{W_0}^{x/U} \frac{\log 2q}{\log W/2q} \frac{dW}{W}.
\end{equation}
If $V < 2 \theta q\leq x/U$, this is
\[\begin{aligned}
\frac{\beta}{2} \int_{2 \theta q}^{x/U} \frac{\log 2q}{\log W/2q} 
\frac{dW}{W} &=
\frac{\beta}{2} \log 2q \cdot \int_{\theta}^{x/2Uq} \frac{1}{\log t} \frac{dt}{t}\\
&= \frac{\beta}{2} \log 2q \cdot \left(\log \log \frac{x}{2Uq} - \log \log 
\theta\right).
\end{aligned}\]
If $2\theta q > x/U$, 
the integral is over an empty range and its 
contribution is hence $0$.

If $2\theta q\leq V$, (\ref{eq:resist}) is 
\[\begin{aligned}
\frac{\beta}{2} \int_V^{x/U} \frac{\log 2q}{\log W/2q} \frac{dW}{W} &=
\frac{\beta \log 2q}{2} \int_{V/2q}^{x/2Uq} \frac{1}{\log t} \frac{dt}{t} \\&=
\frac{\beta \log 2q}{2} \cdot (\log \log \frac{x}{2Uq} - \log \log V/2q)\\ &=
\frac{\beta \log 2q}{2} \cdot \log \left(1+ \frac{\log x/UV}{\log V/2q}\right).
\end{aligned}\]
(Of course, $\log (1+(\log x/UV)/(\log V/2q))\leq (\log x/UV)/(\log V/2q)$;
this is smaller than $(\log x/UV)/\log 2q$ when $V/2q>2q$.)

The total bound for (\ref{eq:soledad}) is thus
\begin{equation}\label{eq:egmont}
\frac{x}{\sqrt{\phi(q)}} \cdot \left(
\beta \cdot \left(\frac{1}{2} \log \frac{x}{U W_0} + \frac{\Phi}{2}\right) +
\beta^{-1} \left(\frac{1}{4} \kappa_{6} \log \frac{x}{U W_0} + 
\frac{\kappa_{7}}{2}
\right)\right),\end{equation}
where
\begin{equation}\label{eq:cocot}
\Phi = 
\begin{cases}
\log 2q
\left(\log \log \frac{x}{2Uq} - \log \log \theta\right) 
&\text{if $V/2\theta<q<x/(2\theta U)$.}\\
\log 2q
\log \left(1+ \frac{\log x/UV}{\log V/2q}\right) 
&\text{if $q\leq V/2\theta$.}
\end{cases}
\end{equation}
Choosing $\beta$ optimally, we obtain that (\ref{eq:soledad}) is at most
\begin{equation}\label{eq:valmont}
  \frac{x}{\sqrt{2 \phi(q)}} \sqrt{\left(\log \frac{x}{U W_0} + \Phi\right)
\left(\kappa_{6} \log \frac{x}{U W_0} + 2 \kappa_{7}\right)},
\end{equation}
where $\Phi$ is as in (\ref{eq:cocot}).

{\em Bounding $S_2$ for $|\delta|\geq 8$.} 
Let us see how much a non-zero $\delta$ can help us. It makes
sense to apply (\ref{eq:procida2}) only when $|\delta|\geq 4$; otherwise
(\ref{eq:garn1a}) is almost certainly better. Now, by definition,
$|\delta|/x\leq 1/qQ$, and so $|\delta|\geq 8$ can happen only when 
$q \leq x/8Q$. 

With this in mind, let us apply (\ref{eq:procida2}).
Note first that
\[\begin{aligned}
\frac{x}{|\delta q|} \left(q + \frac{x}{4 W}\right)^{-1} &\geq
\frac{1/|\delta q|}{\frac{q}{x} + \frac{1}{4 W}} \geq
\frac{4/|\delta q|}{\frac{1}{2 Q} + \frac{1}{W}}\\ &\geq
\frac{4 W}{|\delta| q} \cdot \frac{1}{1 + \frac{W}{2 Q}}
\geq \frac{4 W}{|\delta| q} \cdot \frac{1}{1 + \frac{x/U}{2 Q}}
.\end{aligned}\]
This is at least $2 \min(2 Q,W)/|\delta q|$.
Thus we may apply (\ref{eq:procida2})--(\ref{eq:procida3}) 
when $|\delta q|\leq 2 \min(2Q,W)$.
Since $Q\geq x/U$, we know that $\min(2Q,W)=W$ for all $W\leq x/U$, and
so it is enough to assume that $|\delta q|\leq 2 W$. 

Recalling also (\ref{eq:immer}), we see that (\ref{eq:procida2}) gives us
\begin{equation}\label{eq:thislife}\begin{aligned}
S_2(U',W',W) &\leq \min\left(1,\frac{2 q/\phi(q)}{
\log \left(\frac{4 W}{|\delta| q} \cdot \frac{1}{1 + \frac{x/U}{2 Q}}
\right)}\right)
 \left(\frac{x}{|\delta q|} + \frac{W}{2}\right)
\cdot \frac{1}{2} W (\log W).\end{aligned}\end{equation}

Similarly to before, we define $W_0 = \max(V, \theta |\delta q|)$, where
$\theta\geq 1$ will be set later. For $W\geq W_0$, we certainly have
$|\delta q|\leq 2 W$. Hence the part of (\ref{eq:secint}) coming from
the range $W_0\leq W < x/U$ is
\begin{equation}\label{eq:tort}
\begin{aligned}4 &\int_{W_0}^{x/U} \sqrt{S_1(U,W) \cdot S_2(U,V,W)} \frac{dW}{W} 
\\ &\leq 4 \sqrt{\frac{q}{\phi(q)}} \int_{W_0}^{x/U} 
 \sqrt{S_1(U,W) \cdot \frac{\log W}{
\log \left(\frac{4 W}{|\delta| q} \cdot \frac{1}{1 + \frac{x/U}{2 Q}}\right)}
   \left(\frac{W x}{|\delta q|} + \frac{W^2}{2}\right)} \frac{dW}{W}.
\end{aligned}\end{equation}
By (\ref{eq:menson2}), the contribution of the term $W x/|\delta q|$ to
(\ref{eq:tort}) is at most
\[\frac{4 x}{\sqrt{|\delta| \phi(q)}} 
\int_{W_0}^{x/U} \sqrt{\left(H_2\left(\frac{x}{WU}\right) + 
\frac{\kappa_4}{4} \sqrt{\frac{x/WU}{U}}\right) 
\frac{\log W}{\log \left(
\frac{4 W}{|\delta| q} \cdot \frac{1}{1 + \frac{x/U}{2 Q}}\right)}}
 \frac{dW}{W}\]
Note that $1+(x/U)/2Q\leq 3/2$.
Proceeding as in (\ref{eq:soledad})--(\ref{eq:valmont}), we obtain that this is
at most 
\[\frac{2 x}{\sqrt{|\delta| \phi(q)}}
\sqrt{\left(\log \frac{x}{U W_0} + \Phi\right) \left(\kappa_{6}
\log \frac{x}{U W_0} + 2 \kappa_{7}\right)},\]
where \begin{equation}\label{eq:regxo1}
\Phi = \begin{cases}
\log \frac{(1+\epsilon_1) |\delta q|}{4} \log \left(1 + \frac{\log x/UV}{
\log 4 V/|\delta| (1+ \epsilon_1) q}\right) &\text{if $|\delta q|\leq V/\theta$,}\\
\log \frac{3 |\delta q|}{8} 
\left(\log \log \frac{8 x}{3 U |\delta q|} - \log \log \frac{8 \theta}{3}\right)
&\text{if $V/\theta < |\delta q| \leq x/\theta U$,}
\end{cases}\end{equation}
where $\epsilon_1 = x/2 U Q$. This is what we think of as the main term.

By (\ref{eq:crudo}), the contribution of the term 
$W^2/2$ to (\ref{eq:tort}) is at most
\begin{equation}\label{eq:heeren}
4 \sqrt{\frac{q}{\phi(q)}} \int_V^{x/U} \sqrt{\frac{\kappa_1}{2} x} \frac{
dW}{\sqrt{W}}
\cdot \max_{V\leq W\leq \frac{x}{U}} \sqrt{\frac{\log W}{\max\left(\log \frac{2 W}{|\delta
q|},2\right)}}.\end{equation}
Since $t\to (\log t)/(\log t/c)$ is decreasing for $t>c$, (\ref{eq:heeren})
 is at most
\[ 4 \sqrt{2 \kappa_1} \sqrt{\frac{q}{\phi(q)}} \frac{x}{\sqrt{U}}
\sqrt{\frac{\log W_0}{\max\left(\log \frac{8 W_0}{3 |\delta q|},\frac{8}{3}\right)}}.\]

If $W_0 > V$, we also have to consider the range $V\leq W < W_0$.
The part of (\ref{eq:secint}) coming from this is
\[
4 \int_{V}^{\theta |\delta q|} \sqrt{S_1(U,W) \cdot
(\log W) \left(\frac{W x}{2 |\delta q|} + \frac{W^2}{4} + 
\frac{W x}{16 (1-\rho) Q} +
\frac{x}{8 (1-\rho)}\right)} \frac{dW}{W}. 
\]
We have already counted the contribution of $W^2/4$ in the above.
The terms $W x/2 |\delta| q$ and $W x/(16 (1-\rho) Q)$  contribute at most
\[\begin{aligned} 4 \sqrt{\kappa_1} \int_{V}^{\theta |\delta q|} 
&\sqrt{\frac{x}{W} \cdot (\log W) W \left(
\frac{x}{2 |\delta q|} + \frac{x}{16 (1-\rho) Q}\right)
} \frac{dW}{W}\\
&= 4 \sqrt{\kappa_1} x \left(\frac{1}{\sqrt{2 |\delta| q}} +
\frac{1}{4 \sqrt{(1-\rho) Q}}\right) 
 \int_{V}^{\theta |\delta q|}  
\sqrt{\log W}\; \frac{dW}{W}\\
&\leq \frac{2 \kappa_2}{3} x \left(\frac{1}{\sqrt{2 |\delta| q}} +
\frac{1}{4 \sqrt{(1-\rho) Q}}\right) \left((\log \theta |\delta| q)^{3/2} -
(\log V)^{3/2}\right).\end{aligned}
\]
The term $x/8(1-\rho)$ contributes
\[\begin{aligned}
\sqrt{\kappa_1} x \int_{V}^{\theta |\delta q|} \sqrt{\frac{\log W}{W (1-\rho)}}
\frac{dW}{W} &\leq \frac{\sqrt{\kappa_1} x}{\sqrt{1-\rho}} 
\int_V^\infty \frac{\sqrt{\log W}}{W^{3/2}} dW\\ &\leq
\frac{\kappa_2 x}{2 \sqrt{(1-\rho) V}} (\sqrt{\log V} + \sqrt{1/\log V}), 
\end{aligned}\]

where we use the estimate 
\[\begin{aligned}
\int_V^{\infty} \frac{\sqrt{\log W}}{W^{3/2}} dW &= 
\frac{1}{\sqrt{V}} \int_1^{\infty} \frac{\sqrt{\log u + \log V}}{u^{3/2}} du\\
&\leq \frac{1}{\sqrt{V}} \int_1^{\infty} \frac{\sqrt{\log V}}{u^{3/2}} du +
\frac{1}{\sqrt{V}} \int_1^{\infty} \frac{1}{2\sqrt{\log V}} 
\frac{\log u}{u^{3/2}} du \\
&= 2 \frac{\sqrt{\log V}}{\sqrt{V}} + \frac{1}{2\sqrt{ V\log V}}\cdot 4
\leq \frac{2}{\sqrt{V}} \left(\sqrt{\log V} + \sqrt{1/\log V}\right) .
\end{aligned}\]

\begin{center} * * * \end{center}

It is time to collect all type II terms. Let us start with the case of
general $\delta$. We will set $\theta\geq e$ later. If $q\leq V/2\theta$,
then $|S_{II}|$ is at most
\begin{equation}\label{eq:vinland1}\begin{aligned}
&\frac{x}{\sqrt{2 \phi(q)}} \cdot 
\sqrt{\left(\log \frac{x}{U V} + \log 2q \log \left(1 + \frac{\log x/UV}{\log V/2q}\right)\right) \left(\kappa_{6} \log
\frac{x}{U V} + 2 \kappa_{7}\right)}\\
&+ \sqrt{2} \kappa_{2} \sqrt{\frac{q}{\phi(q)}} \left(1 + 1.15
\sqrt{\frac{\log 2q}{\log x/2Uq}}
\right) \frac{x}{\sqrt{U}} +
 \kappa_{9} \frac{x}{\sqrt{V}} .
\end{aligned}
\end{equation}
If $V/2\theta <q \leq x/2\theta U$, then $|S_{II}|$ is at most
\begin{equation}\label{eq:vinland2}\begin{aligned}
&\frac{x}{\sqrt{2 \phi(q)}} \cdot 
\sqrt{\left(\log \frac{x}{U\cdot 2\theta q} + \log 2q \log \frac{\log x/2Uq}{\log \theta}\right) 
\left(\kappa_{6} \log
\frac{x}{U \cdot 2 \theta q} + 2 \kappa_{7}\right)}\\
&+ \sqrt{2} \kappa_{2} 
\sqrt{\frac{q}{\phi(q)}} \left(1 + 1.15\sqrt{\frac{\log 2q}{\log x/2Uq}}
\right) \frac{x}{\sqrt{U}} +
(\kappa_{2} \sqrt{\log 2\theta q} + \kappa_{9}) \frac{x}{\sqrt{V}}\\
&+
\frac{\kappa_{2}}{6}
\left((\log 2\theta q)^{3/2} - (\log V)^{3/2}\right) \frac{x}{\sqrt{q}} \\
&+ \kappa_{2} \left(\sqrt{2\theta \cdot \log 2\theta q} + \frac{2}{3}
((\log 2\theta q)^{3/2} - (\log V)^{3/2}) \right) \sqrt{q x} ,
\end{aligned}\end{equation}
where we use the fact that $Q\geq x/U$ (implying that
$\rho_0 = \max(1,2q/Q)$ equals $1$ for $q\leq x/2 U$).
Finally, if $q>x/2\theta U$,
\begin{equation}\label{eq:vinland3}\begin{aligned}|S_{II}|&\leq 
(\kappa_{2} \sqrt{2 \log x/U} + \kappa_{9}) \frac{x}{\sqrt{V}}
+ \kappa_{2} \sqrt{\log x/U} \frac{x}{\sqrt{U}}\\
&+ \frac{2 \kappa_{2}}{3} ((\log x/U)^{3/2} - (\log V)^{3/2})
\left(\frac{x}{2 \sqrt{2 q}} + \sqrt{q x}\right).
\end{aligned}\end{equation}

Now let us examine the alternative bounds for $|\delta|\geq 8$.
If $|\delta q|\leq V/\theta$, then $|S_{II}|$ is at most 
\begin{equation}\label{eq:eriksaga}
\begin{aligned}
&\frac{2 x}{\sqrt{|\delta| \phi(q)}}
\sqrt{
\log \frac{x}{U V} + \log \frac{|\delta q| (1+\epsilon_1)}{4} \log 
\left(1 + \frac{\log x/UV}{\log \frac{4 V}{
|\delta| (1+\epsilon_1) q}}\right)}\\
&\cdot
\sqrt{\kappa_{6}
\log \frac{x}{U V} + 2 \kappa_{7}}
\\ &+ \kappa_{2} \sqrt{\frac{2 q}{\phi(q)}} 
\cdot \sqrt{\frac{\log V}{\log 2 V/|\delta q|}} \cdot \frac{x}{\sqrt{U}}
+ \kappa_9 \frac{x}{\sqrt{V}},
\end{aligned}
\end{equation}
where $\epsilon_1 = x/2UQ$.
If $|\delta q|> V/\theta$, then $|S_{II}|$ is at most 
\begin{equation}\label{eq:vinlandsaga}
\begin{aligned}
&\frac{2 x}{\sqrt{|\delta| \phi(q)}}
\sqrt{\left(\log \frac{x}{U\cdot \theta |\delta| q} + 
\log \frac{3 |\delta q|}{8} 
\log \frac{\log \frac{8 x}{ 3 U |\delta q|}}{\log 8\theta/3}\right) 
\left(\kappa_{6}
\log \frac{x}{U\cdot \theta |\delta q|} + 2 \kappa_{7}\right)}\\
&+ \frac{2 \kappa_{2}}{3} 
\left(\frac{x}{\sqrt{2 |\delta q|}} +
\frac{x}{4 \sqrt{Q-q}}\right) 
\left((\log \theta |\delta q|)^{3/2} - (\log V)^{3/2}\right)
\\ &+
\left(\frac{\kappa_2}{\sqrt{2 (1-\rho)}} \left(\sqrt{\log V} + 
\sqrt{1/\log V}\right) + \kappa_{9}\right) \frac{x}{\sqrt{V}} 
\\ &+
\kappa_2 \sqrt{\frac{q}{\phi(q)}} \cdot \sqrt{\log \theta |\delta q|}
\cdot \frac{x}{\sqrt{U}},
\end{aligned}
\end{equation}
where $\rho = q/Q$.
 (Note that $|\delta|\leq x/Q q$ implies
$\rho\leq x/4 Q^2$, and so $\rho$ will be very small
and $Q-q$ will be very close to $Q$.)
\subsection{Adjusting parameters. Calculations.}\label{subs:totcho}
We must bound the exponential sum $\sum_n \Lambda(n) e(\alpha n)
\eta(n/x)$. By (\ref{eq:bob}), it is enough to sum the bounds
we obtained in \S \ref{subs:putmal}.
We will now see how it will be best to set $U$, $V$ and other parameters.

Usually, the largest terms will be
\begin{equation}\label{eq:llama}
C_0 UV,\end{equation}
where
\begin{equation}\label{eq:guanaco}
C_0 = \begin{cases}
c_{4,I_2} + c_{9,I_2} = 4.39636 &\text{if $|\delta|\leq 1/2c_2 \sim 0.74463$,}\\
c_{4,I_2} + (1+\epsilon) c_{13,I_2} = 4.88963+ 1.31541 \epsilon
&\text{if $|\delta|> 1/2 c_2$}\end{cases}
\end{equation}
(from (\ref{eq:fausto}) and (\ref{eq:magus}), type I;
 $\epsilon\in (0,1)$ will be set later) and 
\begin{equation}\label{eq:codo1}\begin{aligned}
\frac{x}{\sqrt{\delta_0 \phi(q)}}
\sqrt{
\log \frac{x}{U V} + (\log \delta_0 (1+\epsilon_1) q) \log 
\left(1 + \frac{\log \frac{x}{UV}}{\log \frac{V}{\delta_0 (1+\epsilon_1) q}}\right)} 
\sqrt{\kappa_{6}
\log \frac{x}{U V} + 2 \kappa_{7}}
\end{aligned}\end{equation}
(from (\ref{eq:vinland1}) and (\ref{eq:eriksaga}), type II; here
 $\delta_0 = \max(2,|\delta|/4)$, while $\epsilon_1 = x/2 U Q$ for 
$|\delta|>8$ and $\epsilon_1 = 0$ for $|\delta|<8$.


We set $UV = \varkappa x/\sqrt{q \delta_0}$; we must choose $\varkappa>0$.

Let us first optimise $\varkappa$ in the case $|\delta|\leq 4$, so that
$\delta_0 = 2$ and $\epsilon_1=0$. For the purpose of choosing $\varkappa$,
we replace $\sqrt{\phi(q)}$ by 
$\sqrt{q}/C_1$, where $C_1 = 2.3536 \sim 510510/\phi(510510)$,
and also replace $V$ by $q^2/c$, $c$ a constant.
We use the approximation
\[\begin{aligned}
\log \left(1 + \frac{\log \frac{x}{U V}}{\log \frac{V}{|2 q|}}\right)
&= \log \left(1 + \frac{\log(\sqrt{2 q}/\varkappa)}{\log(q/2c)}\right) =
\log \left(\frac{3}{2} + \frac{\log 2 \sqrt{c}/\varkappa}{\log q/2c}\right)\\ 
&\sim
\log \frac{3}{2} + \frac{2 \log 2 \sqrt{c}/\varkappa}{
3 \log q/2c}.
\end{aligned}\]
What we must minimize, then, is
\begin{equation}\label{eq:cojono}\begin{aligned}
&\frac{C_0 \varkappa}{\sqrt{2 q}} + \frac{C_1}{\sqrt{2 q}}
\sqrt{\left(\log \frac{\sqrt{2 q}}{\varkappa} + \log 2q \left(\log \frac{3}{2}
+ \frac{2 \log \frac{2 \sqrt{c}}{\varkappa}}{
3 \log \frac{q}{2c}}
\right)\right) \left(\kappa_{6}
\log \frac{\sqrt{2 q}}{\varkappa} + 
2 \kappa_{7}\right)}\\
 &\leq \frac{C_0 \varkappa}{\sqrt{2 q}}  + \frac{C_1}{2 \sqrt{q}}
\frac{\sqrt{\kappa_{6}}}{\sqrt{\kappa_1'}}
\sqrt{\kappa_1' \log q 
- \left(\frac{5}{3} + \frac{2}{3} \frac{\log 4c }{\log \frac{q}{2c}}\right) \log \varkappa + \kappa_2'}
\\ &\cdot \sqrt{
\kappa_1' \log q - 2 \kappa_1'
 \log \varkappa + \frac{4 \kappa_1' \kappa_{7}}{\kappa_{6}} +
\kappa_1' \log 2} \\
&\leq
\frac{C_0}{\sqrt{2 q}} \left(\varkappa  + \kappa_{4}' 
\left(\kappa_1' \log q - \left(\left(\frac{5}{6} + \kappa_1'\right) 
+ \frac{1}{3} \frac{\log 4c}{\log \frac{q}{2c}}\right)
\log \varkappa + \kappa_{3}'\right)\right),\end{aligned}\end{equation}
where 
\[\begin{aligned}
\kappa_1' &= \frac{1}{2} + \log \frac{3}{2},\;\;\;\;
\kappa_2' = \log \sqrt{2} + \log 2 \log \frac{3}{2} + \frac{\log 4 c \log 2 q}{3 \log q/2c},\\
\kappa_{3}' &= \frac{1}{2} \left(\kappa_2' + \frac{4 \kappa_1'
\kappa_{7}}{\kappa_{6}} + \kappa_1' \log 2 \right) = 
\frac{\log 4 c}{6} + \frac{(\log 4c)^2}{6 \log \frac{q}{2 c}} + 
\kappa_5',\\
\kappa_{4}' &= 
\frac{C_1}{C_0} \sqrt{\frac{\kappa_{6}}{2 \kappa_1'}} \sim
\begin{cases}
0.30925 & \text{if $|\delta|\leq 4$}\\
\frac{0.27805}{1 + 0.26902 \epsilon} & \text{if $|\delta|>4$},\end{cases}\\
\kappa_5' &= \frac{1}{2} (\log \sqrt{2} + \log 2 \log \frac{3}{2} +
\frac{4 \kappa_1' \kappa_7}{\kappa_6} + \kappa_1' \log 2) \sim
1.01152 . 
\end{aligned}\]
Taking derivatives, we see that the minimum is attained when
\begin{equation}\label{eq:jotoka}
\varkappa = \left(\frac{5}{6}+\kappa_1'+
\frac{1}{3} \frac{\log 4c}{\log \frac{q}{2c}}\right) 
\kappa_{4}' \sim \left(1.7388 + \frac{\log 4c}{3 \log \frac{q}{2c}}\right) 
\cdot 0.30925 
\end{equation}
provided that $|\delta|\leq 4$.
(What we obtain for $|\delta|>4$ is essentially the same, only with 
$\log \delta_0 q = \log |\delta| q/4$ instead of $\log q$,
and $0.27805/(1+0.26902 \epsilon)$ in place of $0.30925$.) 
For $q=5\cdot 10^5$, $c=2.5$ and $|\delta|\leq 4$ 
(typical values in the most delicate range), we get that
 $\varkappa$ should be $0.55834\dotsc$, and
the last line
of (\ref{eq:cojono}) is then
$0.02204\dotsc$; for $q=10^6$, $c=10$, $|\delta|\leq 4$, we get that
$\varkappa$ should be $0.57286\dotsc$, 
and the last line of (\ref{eq:cojono}) is then
 $0.01656\dotsc$. If $|\delta|>4$, $|\delta| q = 5\cdot 10^5$, $c=2.5$
and $\epsilon=0.2$ 
(say), then $\varkappa = 0.47637\dotsc$, and the last line of (\ref{eq:cojono}) is
$0.02243\dotsc$; if $|\delta|>4$, $|\delta| q = 10^6$, $c=10$ and $\epsilon=0.2$, then
$\varkappa = 0.48877\dotsc$, 
and the last line of (\ref{eq:cojono}) is $0.01684\dotsc$.

(A back-of-the-envelope calculation suggests that choosing $w=1$
instead of $w=2$ would have given bounds worse by about $15$ percent.)

We make the choices
\[
\varkappa = 1/2,\;\;\;\; \text{and so}\;\;\;\;\;\; UV = \frac{1}{2 \sqrt{q \delta_0}}
\]
for the sake of simplicity. (Unsurprisingly, (\ref{eq:cojono}) changes
very slowly around its minimum.)


Now we must decide how to choose $U$, $V$ and $Q$, given our choice of 
$UV$. We will actually make two sets of choices. First, we will use the
$S_{I,2}$ estimates for $q\leq Q/V$ to treat all $\alpha$ of the form
$\alpha = a/q + O^*(1/qQ)$, $q\leq y$. (Here $y$ is a parameter satisfying
$y\leq Q/V$.) The remaining $\alpha$ then
get treated with the (coarser) $S_{I,2}$ estimate for $q>Q/V$, with 
 $Q$ reset to a lower value (call it $Q'$). If $\alpha$ was
not treated in the first go (so that it must be dealt with the coarser
estimate) then $\alpha = a'/q'+\delta'/x$, where either $q'>y$ or
$\delta' q' > x/Q$. (Otherwise, $\alpha = a'/q'+O^*(1/q'Q)$ would be a valid 
estimate with $q'\leq y$.)

The value of $Q'$ is set to be smaller than $Q$
both because this is helpful (it diminishes error terms that would be large
for large $q$) and because this is now harmless (since we are no longer
assuming that $q\leq Q/V$). 

\subsubsection{First choice of parameters: $q\leq y$}\label{subs:cojor}

The largest items affected strongly by our choices at this point are
\begin{equation}\label{eq:ned1}\begin{aligned}
c_{16,I_2} \left(2 + \frac{1+\epsilon}{\epsilon} \log^+ \frac{2 U V |\delta| q}{
x}\right) \frac{x}{Q/V} + c_{17,I_2} Q
 \;\;\;\;\;\;\text{(from $S_{I,2}$, $|\delta|> 1/2c_2$)},\\
\left(c_{10,I_2} \log \frac{U}{q} + 2 c_{5,I_2} + c_{12,I_2}\right) Q
 \;\;\;\;\;\;\text{(from $S_{I,2}$, $|\delta|\leq 1/2c_2$)},
\end{aligned}\end{equation}
and
\begin{equation}\label{eq:ned2}
\kappa_{2} \sqrt{\frac{2 q}{\phi(q)}} \left(1 + 1.15\sqrt{\frac{\log 2q}{\log x/2Uq}}
\right) \frac{x}{\sqrt{U}} +
 \kappa_{9} \frac{x}{\sqrt{V}} 
\;\;\;\; \text{(from $S_{II}$)}.
\end{equation}
In addition, we have a relatively mild but important dependence on $V$
in the main term (\ref{eq:codo1}). We must also respect the 
condition $q\leq Q/V$, the lower bound on
$U$ given by (\ref{eq:curious}) and the assumptions made at the
beginning of section \ref{sec:osval} (e.g. $Q\geq x/U$, $V\geq 2\cdot 10^6$). 
Recall that $UV = x/\sqrt{q \delta}$.

We set 
\[Q = \frac{x}{8 y},\]
since we will then have not just $q\leq y$ but also $q |\delta| \leq
x/Q = 8y$, and so $q \delta_0 \leq 4y$.
We want $q\leq Q/V$ to be true whenever $q\leq y$; this means that
\[q\leq \frac{Q}{V} = \frac{Q U}{U V} = \frac{Q U}{x/2 \sqrt{q \delta_0}}
= \frac{U \sqrt{q \delta_0}}{4 y}\]
must be true when $q\leq y$, and so it is enough to set
$U = 4 y^2/\sqrt{q \delta_0}$.
The following choices make sense: we will work with the
parameters
\begin{equation}\label{eq:humid}\begin{aligned}
y &= \frac{x^{1/3}}{6},\;\;\;\;\;\; 
Q=\frac{x}{8 y} = \frac{3}{4} x^{2/3},\;\;\;\;\;\; x/UV = 2 \sqrt{q \delta_0}\leq 2 \sqrt{2 y},\\
U &= \frac{4 y^2}{\sqrt{q \delta_0}}
= \frac{x^{2/3}}{9 \sqrt{q \delta_0}},
\;\;\;\;\;\; 
V = \frac{x}{(x/UV)\cdot U} 
= 
\frac{x}{8 y^2} = \frac{9 x^{1/3}}{2},
\end{aligned}\end{equation}
where, as before, $\delta_0 = \max(2,|\delta|/4)$. 
Thus $\epsilon_1 \leq x/2 U Q \leq 2 \sqrt{6}/x^{1/6}$. Assuming
\begin{equation}\label{eq:voil}
x \geq 2.16 \cdot 10^{20},\end{equation}
we obtain that 
$U/(x/UV) \geq (x^{3/2}/9\sqrt{q\delta_0})/(2 \sqrt{q \delta_0}) = x^{2/3}/18
q \delta_0\geq x^{1/3}/6 \geq 5\cdot 10^5$,
and so (\ref{eq:curious}) holds. We also get that $\epsilon_1\leq 0.002$.

Since $V=x/8 y^2 = (9/2) x^{1/3}$, (\ref{eq:voil}) also implies that
$V\geq 2\cdot 10^6$ (in fact, $V\geq 27\cdot 10^6$). It is easy to check that
\begin{equation}\label{eq:herring}
V<x/4,\;\;\; UV\leq x,\;\;\;\;
Q\geq \sqrt{ex},\;\;\;\;Q\geq \max(U,x/U),\end{equation}
as stated at the beginning of section \ref{sec:osval}.
Let $\theta = (3/2)^3 = 27/8$. Then
\begin{equation}\label{eq:werto}\begin{aligned}
\frac{V}{2 \theta q} &= \frac{x/8 y^2}{2 \theta q} \geq \frac{x}{16 \theta y^3
} = \frac{x}{54 y^3} = 4 > 1,\\
\frac{V}{\theta |\delta q|} &= \frac{x/8 y^2}{8 \theta y} 
\geq \frac{x}{64 \theta y^3} = \frac{x}{216 y^3} = 1.\end{aligned}
\end{equation}



The first type I bound is
\begin{equation}\label{eq:dikaiopolis}\begin{aligned}
|S_{I,1}| &\leq 
\frac{x}{q} \min\left(1, \frac{c_0'}{\delta^2}\right)
\left(\min\left(\frac{\frac{4}{5} \frac{q}{\phi(q)}}{\log^+ 
\frac{x^{2/3}}{9 q^{\frac{5}{2}} \delta_0^\frac{1}{2}}}, 1\right) \left(\log 9 x^{\frac{1}{3}} \sqrt{q \delta_0}
+ c_{3,I}\right)
 + \frac{c_{4,I} q}{\phi(q)} \right)\\
&+ \left(c_{7,I}
\log \frac{y}{c_2} + c_{8,I} \log x \right) y 
+ \frac{c_{10,I} x^{1/3}}{3^4 2^2 q^{3/2} \delta_0^{\frac{1}{2}}} (\log 9 x^{1/3}
\sqrt{e q \delta_0})\\
&+ \left(c_{5,I} \log \frac{2 x^{2/3}}{9 c_2 \sqrt{q \delta_0}} +
c_{6,I} \log \frac{x^{5/3}}{9 \sqrt{q \delta_0}}\right) \frac{x^{2/3}}{9 \sqrt{q
\delta_0}} 
+ c_{9,I} \sqrt{x} \log \frac{2 x}{c_2} 
+ \frac{c_{10,I}}{e}
,\end{aligned}\end{equation}
where the constants are as in \S \ref{subs:renzo}.
The function $x\to (\log cx)/(\log x/R)$, $c,R\geq 1$, attains its maximum on
$\lbrack R',\infty\rbrack$, $R'>R$, at $x=R'$. Hence, for
$q \delta_0$ fixed,
\begin{equation}\label{eq:peergynt}
\min\left(\frac{4/5}{\log^+ 
\frac{4 x^{2/3}}{9 (\delta_0 q)^{\frac{5}{2}}}}, 1\right) \left(\log 9 x^{\frac{1}{3}} 
\sqrt{q \delta_0} + c_{3,I}\right)
\end{equation}
attains its maximum at $x = (27/8) e^{6/5} (q \delta_0)^{15/4}$,
and so 
\[\begin{aligned}
&\min\left(\frac{4/5}{\log^+ 
\frac{4 x^{2/3}}{9 (\delta_0 q)^{\frac{5}{2}}}}, 1\right) \left(\log 9 x^{\frac{1}{3}} 
\sqrt{q \delta_0} + c_{3,I}\right) + c_{4,I}
\\ &\leq \log \frac{27}{2} e^{2/5} (\delta_0 q)^{7/4} + c_{3,I} + c_{4,I} \leq 
\frac{7}{4} \log \delta_0 q + 6.11676 .
\end{aligned}\]
Examining the other terms in (\ref{eq:dikaiopolis})
and using (\ref{eq:voil}), we conclude that
\begin{equation}\label{eq:therwald}\begin{aligned}
|S_{I,1}| &\leq \frac{x}{q} \min\left(1, \frac{c_0'}{\delta^2}\right)
\cdot \min\left(\frac{q}{\phi(q)} \left(\frac{7}{4} \log \delta_0 q + 6.11676
\right), \frac{1}{2} \log x + 5.65787
\right) \\ &+
\frac{x^{2/3}}{\sqrt{q \delta_0}} (0.67845 \log x - 1.20818) + 
0.0507 x^{2/3},\end{aligned}\end{equation}
where we are using (\ref{eq:voil}) to simplify the smaller error terms.
(The bound $(1/2) \log x + 5.65787$ comes from a trivial bound on 
(\ref{eq:peergynt}).) We recall that $c_0' = 0.798437 > c_0/(2\pi)^2$.

Let us now consider $S_{I,2}$. The terms that appear both for $|\delta|$ small
and $|\delta|$ large are given in (\ref{eq:chusan}). The second line
in (\ref{eq:chusan}) equals 
\[\begin{aligned}
c_{8,I_2} &\left(\frac{x}{4 q^2 \delta_0} + \frac{2 U V^2}{x} + 
\frac{q V^2}{x}\right) +
\frac{c_{10,I_2}}{2} \left( \frac{q}{2 \sqrt{q \delta_0}} +
\frac{x^{2/3}}{18 q \delta_0}\right) \log \frac{9 x^{1/3}}{2}\\
&\leq c_{8,I_2} \left(\frac{x}{4 q^2 \delta_0} + 
\frac{9 x^{1/3}}{2 \sqrt{2}} + \frac{27}{8}\right) +
 \frac{c_{10,I_2}}{2} \left( \frac{y^{1/6}}{2^{3/2}} +
 \frac{x^{2/3}}{18 q \delta_0}\right) \left(\frac{1}{3}\log x + \log \frac{9}{2}\right)\\
&\leq 0.29315 \frac{x}{q^2 \delta_0} + (0.00828 \log x + 0.03735) \frac{x^{2/3}}{\sqrt{q \delta_0}} +
 0.00153 \sqrt{x},
\end{aligned}\]
where we are using (\ref{eq:voil}) to simplify. Now
\begin{equation}\label{eq:gorachy}
\min\left(\frac{4/5}{\log^+ \frac{Q}{4 V q^2}},1\right) \log V q =
\min\left(\frac{4/5}{\log^+ \frac{y}{4 q^2}},1\right) \log 
\frac{9 x^{1/3} q}{2} 
\end{equation} can be bounded
trivially by $\log(9 x^{1/3} q/2)\leq 
(2/3) \log x + \log 3/4$. We can also bound (\ref{eq:gorachy})
as we bounded (\ref{eq:peergynt}) before, namely, by fixing $q$
and finding the maximum for $x$ variable. In this way, we obtain that
(\ref{eq:gorachy}) is maximal for $y = 4 e^{4/5} q^2$; since, by
definition, $x^{1/3}/6=y$,
(\ref{eq:gorachy}) then equals
\[\log \frac{9 (6\cdot 4 e^{4/5} q^2) q}{2} = 3 \log q + \log 108 + \frac{4}{5}
\leq 3 \log q + 5.48214.\]

If $|\delta|\leq 1/2c_2$, we must consider (\ref{eq:fausto}). This is at most
\[\begin{aligned}
&(c_{4,I_2}+c_{9,I_2}) \frac{x}{2 \sqrt{q \delta_0}} +
(c_{10,I_2} \log \frac{x^{2/3}}{9 q^{3/2} \sqrt{\delta_0}} + 2 c_{5,I_2} + c_{12,I_2}) \cdot \frac{3}{4}
x^{2/3}\\
&\leq \frac{2.19818 x}{\sqrt{q \delta_0}} + (0.89392 \log x + 23.0896) x^{2/3}.
\end{aligned}\]
If $|\delta|>1/2c_2$, we must consider (\ref{eq:magus}) instead. For
$\epsilon = 0.07$, that is at most
\[\begin{aligned}
&(c_{4,I_2}+(1+\epsilon) c_{13,I_2}) \frac{x}{2 \sqrt{q \delta_0}} +
(3.30386 \log \delta q^3 + 16.4137) \frac{x}{|\delta| q}\\
&+ (68.8137 \log |\delta| q + 36.7795) x^{2/3} + 29.7467 x^{1/3}\\
&= 2.49086 \frac{x}{\sqrt{q \delta_0}} +
(3.30386 \log \delta q^3 + 16.4137) \frac{x}{|\delta| q} +
(22.9379 \log x + 56.576) x^{\frac{2}{3}}.
\end{aligned}\]
Hence
\begin{equation}\label{eq:cleson}\begin{aligned}
|S_{I,2}|&\leq 2.49086 \frac{x}{\sqrt{q \delta_0}} \\ &+ 
x\cdot \min\left(1,\frac{4 c_0'}{\delta^2}\right)
 \min\left(\frac{\frac{3}{2} \log q + 2.74107}{\phi(q)}, 
\frac{\frac{1}{3} \log x + \frac{1}{2} \log
\frac{3}{4}}{q} \right)
\\ &+ 0.29315 \frac{x}{q^2 \delta_0} + 
(22.9462 \log x + 56.6134) x^{2/3} \end{aligned}\end{equation}
plus a term $(3.30386 \log \delta q^2 + 16.4137) \cdot 
(x/|\delta| q)$
that appears if and only if $|\delta|\geq 1/2c_2$.

 For type II, we have to consider two cases: (a) 
$|\delta|<8$, and (b)  $|\delta|\geq 8$. Consider first $|\delta|<8$. 
Then $\delta_0 = 2$.
Recall that $\theta =27/8$.
We have $q \leq V/2\theta$ and $|\delta q| \leq V/\theta$ thanks to (\ref{eq:werto}).
We apply (\ref{eq:vinland1}), and obtain that, for $|\delta|<8$,
\begin{equation}\label{eq:pell}\begin{aligned}
|S_{II}|&\leq
\frac{x}{\sqrt{2 \phi(q)}} \cdot 
\sqrt{\frac{1}{2} \log 4q\delta_0 + \log 2q \log \left(1+ 
\frac{\frac{1}{2} \log 4q\delta_0}{\log \frac{V}{2q}}\right)}\\ &\cdot 
\sqrt{0.30214 \log 4 q \delta_0 + 0.2562}\\
&+ 8.22088 \sqrt{\frac{q}{\phi(q)}} \left(1 + 1.15 
\sqrt{\frac{\log 2q}{\log \frac{9 x^{1/3} \sqrt{\delta_0}}{2 \sqrt{q}}}}\right) (q \delta_0)^{1/4} x^{2/3}  +  1.84251 x^{5/6}\\
&\leq \frac{x}{\sqrt{2 \phi(q)}} \cdot 
\sqrt{C_{x,2 q} \log 2q + \frac{\log q}{2}} \cdot 
\sqrt{0.30214 \log 2q  + 0.67506}\\
&+ 16.404 \sqrt{\frac{q}{\phi(q)}} x^{3/4} +  1.84251 x^{5/6}
\end{aligned}\end{equation}
where we define
\[C_{x,t} := 
\log \left(1 + \frac{\log 4 t}{2 \log \frac{9 x^{1/3}}{2.004 t}}\right)\]
for $0 < t < 9 x^{1/3}/2$. (We have $2.004$ here instead of $2$ because
we want a constant $\geq 2 (1+\epsilon_1)$ in later occurences of
$C_{x,t}$, for reasons that will soon become clear.)

For purposes of later comparison, we remark
that $16.404\leq 1.5785 x^{3/4-4/5}$ for $x\geq 2.16\cdot 10^{20}$.

Consider now case (b), namely, $|\delta|\geq 8$. Then $\delta_0  = |\delta|/4$.
By (\ref{eq:werto}), $|\delta q|\leq V/\theta$. Hence,
 (\ref{eq:eriksaga}) gives us that
\begin{equation}\label{eq:meli}\begin{aligned}
|S_{II}|&\leq
\frac{2 x}{\sqrt{|\delta| \phi(q)}} \cdot 
\sqrt{\frac{1}{2} \log |\delta q|  + \log \frac{|\delta q| (1+\epsilon_1)}{4} 
\log \left(1+ \frac{\log |\delta| q}{2 \log \frac{18 x^{1/3}}{|\delta|
(1+ \epsilon_1) q}}
\right)}\\
&\cdot \sqrt{0.30214 \log |\delta| q + 0.2562} \\ &+ 8.22088
\sqrt{\frac{q}{\phi(q)}}\sqrt{\frac{\log \frac{9 x^{1/3}}{2}}{\log
\frac{12 x^{1/3}}{|\delta q|}}}\cdot (q \delta_0)^{1/4} x^{2/3}  
 +  1.84251 x^{5/6}\\
&\leq \frac{x}{\sqrt{\delta_0 \phi(q)}} 
\sqrt{C_{x, \delta_0 q} \log \delta_0 (1+\epsilon_1) q + \frac{\log 4 \delta_0 q}{2}}
\sqrt{0.30214 \log \delta_0 q + 0.67506} \\ 
&+ 1.68038 \sqrt{\frac{q}{\phi(q)}} x^{4/5} + 
1.84251 x^{5/6} ,
\end{aligned}\end{equation}
since
\[8.22088 \sqrt{\frac{\log \frac{9 x^{1/3}}{2}}{\log
\frac{12 x^{1/3}}{|\delta q|}}}\cdot (q \delta_0)^{1/4} \leq
8.22088 \sqrt{\frac{\log \frac{9 x^{1/3}}{2}}{\log 9}}
\cdot (x^{1/3}/3)^{1/4} \leq 1.68038 x^{4/5-2/3}\]
for $x\geq 2.16\cdot 10^{20}$. Clearly
\[\log \delta_0 (1+\epsilon_1) q \leq \log \delta_0 q + \log(1+\epsilon_1)
\leq \log \delta_0 q + \epsilon_1.\]
Now note the fact (\cite[Thm. 15]{MR0137689}) that 
$q/\phi(q) < \digamma(q)$, where
\begin{equation}\label{eq:locos}\digamma(q) = 
e^{\gamma} \log \log q + 
\frac{2.50637}{\log \log q}.\end{equation}
Moreover, $q/\phi(q)\leq 3$ for $q<30$. Since $\digamma(30)>3$ and
$\digamma(t)$ is increasing for
$t\geq 30$, we conclude that, for any $q$ and for any 
$r\geq \max(q,30)$, $q/\phi(q) < \digamma(r)$. In particular,
$q/\phi(q)\leq \digamma(y) = \digamma(x^{1/3}/6)$ (since, by
(\ref{eq:voil}), $x\geq 180^3$).
It is easy to check that $x\to \sqrt{\digamma(x^{1/3}/6)} x^{4/5-5/6}$
is decreasing for $x\geq 180^3$. Using (\ref{eq:voil}), we conclude that
 $1.67718 \sqrt{q/\phi(q)} x^{4/5} \leq 0.83574 x^{5/6}$.
This allows us to simplify the last lines of (\ref{eq:pell}) and 
(\ref{eq:meli}).

It is time to sum up $S_{I,1}$, $S_{I,2}$ and $S_{II}$. The main terms
come from the first lines of (\ref{eq:pell}) and (\ref{eq:meli}) and
the first term of (\ref{eq:cleson}). Lesser-order terms can be dealt with
roughly: we bound $\min(1,c_0'/\delta^2)$ and $\min(1,4 c_0'/\delta^2)$
from above by $2/\delta_0$ (somewhat brutally) and $1/q^2 \delta_0$ by
$1/q \delta_0$ (again, coarsely). For $|\delta|\geq 1/2 c_2$,
\[\frac{1}{|\delta|}\leq \frac{4 c_2}{\delta_0},\;\;\;\;\;\;\;
\frac{\log |\delta|}{|\delta|}\leq \frac{2}{e \log 2}\cdot \frac{\log \delta_0}{\delta_0};\]
we use this to bound the term in the comment after (\ref{eq:cleson}).
The terms inversely proportional to $q$, $\phi(q)$ or $q^2$ thus add up to 
at most
\[\begin{aligned}
&\frac{2 x}{\delta_0} \cdot \min\left(\frac{\frac{7}{4} \log \delta_0 q
+ 6.11676}{\phi(q)}, \frac{\frac{1}{2} \log x + 5.65787}{q}\right)\\
&+ \frac{2 x}{\delta_0} \cdot \min\left(\frac{\frac{3}{2} \log q + 
2.74107}{\phi(q)}, \frac{\frac{1}{3} \log x + \frac{1}{2} \log \frac{3}{4}}{q}
\right)\\
&+ 0.29315 \frac{x}{q \delta_0} + \frac{4 c_2 x}{q \delta_0}
(3.30386 \log q^2 + 16.4137) + \frac{2 x}{(e \log 2) q \delta_0} \cdot
3.30386 \log \delta_0\\
&\leq \frac{2 x}{\delta_0}
\min\left(\frac{\log \delta_0^{7/4} q^{13/4} + 8.858}{\phi(q)},
\frac{\frac{5}{6} \log x + 5.515}{q}\right)\\
&+ \frac{2 x}{\delta_0 q} (8.874 \log q + 1.7535 \log \delta_0 + 22.19)\\
&\leq \frac{2 x}{\delta_0} \left(
\min\left(\frac{\log \delta_0^{7/4} q^{13/4} + 8.858}{\phi(q)},
\frac{\frac{5}{6} \log x + 5.515}{q}\right)
+ \frac{\log q^{\frac{80}{9}} \delta_0^{\frac{16}{9}} + 22.19}{q}\right)
.\end{aligned}\]

As for the other terms -- we use (\ref{eq:voil}) to bound $x^{2/3}$ and
$x^{2/3} \log x$ by a small constant times $x^{5/6}$. We bound $x^{2/3}/\sqrt{
q\delta_0}$ by $x^{2/3}/\sqrt{2}$ (in (\ref{eq:therwald})).

The sums $S_{0,\infty}$ and $S_{0,w}$
in (\ref{eq:sofot}) are $0$ (by (\ref{eq:voil})). 
We conclude that, for $q\leq y = x^{1/3}/6$, $x\geq 2.16\cdot 10^{20}$
and $\eta=\eta_2$ as in (\ref{eq:eqeta}),
\begin{equation}\label{eq:duaro}\begin{aligned}
&|S_\eta(x,\alpha)| \leq |S_{I,1}| + |S_{I,2}| + |S_{II}|\\
 &\leq \frac{x}{\sqrt{\phi(q) \delta_0}} 
\sqrt{C_{x, \delta_0 q} (\log \delta_0 q + 0.002)+ 
\frac{\log 4 \delta_0 q}{2}}
\sqrt{0.30214 \log \delta_0 q + 0.67506} 
\\ &+ \frac{2.49086 x}{\sqrt{q \delta_0}}
+ 
\frac{2 x}{\delta_0} 
\min\left(\frac{\log \delta_0^{\frac{7}{4}} q^{\frac{13}{4}} + \frac{80}{9}}{\phi(q)},
\frac{\frac{5}{6} \log x + \frac{50}{9}}{q}\right)
+ \frac{2 x}{\delta_0} 
\frac{\log q^{\frac{80}{9}} \delta_0^{\frac{16}{9}} + \frac{111}{5}}{q}
\\ &+ 3.14624 x^{5/6},
\end{aligned}\end{equation}
where
\begin{equation}\label{eq:raisin}\begin{aligned}
\delta_0 &= \max(2,|\delta|/4),\;\;\;\;\;\;
C_{x,t} =
\log \left(1 + \frac{\log 4 t}{2 \log \frac{9 x^{1/3}}{2.004 t}}\right).
\end{aligned}\end{equation}
Since $C_{x,t}$ is an increasing function as a function of
$t$ (for $x$ fixed and $t\leq 9 x^{1/3}/2.004$) 
and $\delta_0 q\leq 2y$, we see that
$C_{x,t}\leq C_{x,2y}$. It is clear that
$x\mapsto C_{x,t}$ (fixed $t$) is decreasing function of $x$. 
For $x = 2.16 \cdot 10^{20}$, $C_{x,2y} = 1.39942\dotsc$.
Also, compare the value $C_{3.1\cdot 10^{28},2\cdot 10^6}= 0.64020\dotsc$ 
given by (\ref{eq:raisin}) to the value
 of $1.196\dotsc-0.5 = 0.696\dotsc$ for  
$C_{3.1\cdot 10^{28},2\cdot 10^6}$ 
in a previous version \cite{Helf} of the present paper. (The largest
gains are elsewhere.)

\subsubsection{Second choice of parameters}\label{subs:espan}
If, with the original choice of parameters, we obtained $q>y= x^{1/3}/6$, 
we now reset our parameters ($Q$, $U$ and $V$).
Recall that, while the value of $q$ may now change (due to the change in 
$Q$),
we will be able to assume that either $q>y$ or $|\delta q|>x/(x/8y)
= 8y$.

We want $U/(x/UV)\geq 5\cdot 10^5$ (this is (\ref{eq:curious})).
We also want $UV$ small. With this in mind, we let
\[
V = \frac{x^{1/3}}{3},\;\;\;\;\;\;\; U = 500 \sqrt{6} x^{1/3},\;\;\;\;\;\;\;\;
Q = \frac{x}{U} = \frac{x^{2/3}}{500 \sqrt{6}}.
\]
Then
(\ref{eq:curious}) holds (as an equality). 
Since we are assuming (\ref{eq:voil}), we have
$V\geq 2\cdot 10^6$.
It is easy to check that (\ref{eq:voil}) also
implies that $U<\sqrt{x}$ and $Q> \sqrt{e x}$, and so
the inequalities in (\ref{eq:herring}) hold.

Write $2\alpha = a/q + \delta/x$ for the new approximation; we must have 
either $q>y$ or $|\delta| > 8 y/q$, since otherwise $a/q$ would already be
a valid approximation under the first choice of parameters. 
Thus, either (a) $q>y$, or both (b1) $|\delta|>8$ and (b2) $|\delta| q > 
8y$. Since now $V=2y$, we have $q>V/2\theta$ in case (a) and 
$|\delta q|> V/\theta$ in case (b) 
for any $\theta\geq 1$. We set $\theta = e^2$.

By (\ref{eq:lavapie}),
\[\begin{aligned}
|S_{I,1}|&\leq \frac{x}{q} \min\left(1,\frac{c_0'}{\delta^2}\right)
\left(\log x^{2/3} - \log 500 \sqrt{6} + c_{3,I} + c_{4,I} \frac{q}{\phi(q)}
\right)\\
&+ \left(c_{7,I} \log \frac{Q}{c_2} + c_{8,I} \log x \log c_{11,I} \frac{Q^2}{x}
\right) Q + c_{10,I} \frac{U^2}{4 x} \log \frac{e^{1/2} x^{2/3}}{500 \sqrt{6}}
 + \frac{c_{10,I}}{e}\\
&+ \left(c_{5,I} \log \frac{1000 \sqrt{6} x^{1/3}}{c_2} + 
c_{6,I} \log 500 \sqrt{6} x^{4/3}\right) \cdot 500 \sqrt{6} x^{1/3} +
c_{9,I} \sqrt{x} \log \frac{2 x}{c_2}\\
&\leq \frac{x}{q} \min\left(1,\frac{c_0'}{\delta^2}\right)
\left(\frac{2}{3} \log x - 4.99944 + 1.00303 \frac{q}{\phi(q)}\right) + 
\frac{1.063}{10000} x^{2/3} (\log x)^2,
\end{aligned}\]
where we are bound $\log c_{11,I} Q^2/x$ by $\log x^{1/3}$. Just as 
before, we use the assumption (\ref{eq:voil}) when we have to bound
a lower-order term (such as $x^{1/2} \log x$) by a multiple of a higher-order
term (such as $x^{2/3} (\log x)^2$).

We have $q/\phi(q)\leq \digamma(Q)$ (where $\digamma$ is as in
 (\ref{eq:locos})) and we can check that 
\[1.00303 \digamma(Q)\leq 0.0327 \log x + 4.99944\]
for all $x\geq 10^6$.
We have either $q>y$ or $q |\delta| > 8 y$; if $q |\delta|> 8 y$ but 
$q\leq y$, then $|\delta|\geq 8$, and so $c_0'/\delta^2 q < 1/8 |\delta| q
< 1/64 y < 1/y$. Hence 
\[\begin{aligned}
|S_{I,1}|&\leq 4.1962 x^{2/3} \log x + 0.090843 x^{2/3} +
0.001063 x^{2/3} (\log x)^2\\
&\leq  4.1982 x^{2/3} \log x +
0.001063 x^{2/3} (\log x)^2. 
\end{aligned}\]

We bound $|S_{I,2}|$ using Lemma \ref{lem:bogus}. First we bound
(\ref{eq:cupcake3}): this is at most
\[\begin{aligned}
&\frac{x}{2 q} \min\left(1,\frac{4 c_0'}{\delta^2}\right)
\log \frac{x^{1/3} q}{3}\\&+
c_0 \left(\frac{1}{4}-\frac{1}{\pi^2}\right) 
\left(\frac{(UV)^2 \log \frac{x^{1/3}}{3}}{2 x} + \frac{3 c_4}{2}
 \frac{500 \sqrt{6}}{9} + \frac{(500 \sqrt{6} x^{1/3} + 1)^2 x^{1/3}}{3 x}
\right),
\end{aligned}\]
where $c_4 = 1.03884$.
We bound the second line of this using (\ref{eq:voil}). As for the first
line, we have either $q\geq y$ (and so the first line is at most
$(x/2y) (\log x^{1/3} y/3)$) or $q<y$ and $4 c_0'/\delta^2 q < 1/16 y < 1/y$
(and so the same bound applies). Hence (\ref{eq:cupcake3}) is at most
\[\frac{3}{2} x^{2/3} \left(\frac{2}{3} \log x - \log 9\right) + 
0.02017 x^{2/3} \log x.\]

Now we bound (\ref{eq:piececake}), which comes up when $|\delta|\leq 1/2 c_2$,
where $c_2 = 6\pi/5\sqrt{c_0}$, $c_0 =31.521$ (and so $c_2 = 0.6714769\dotsc$).
Since $1/2c_2 <8$, it follows that $q>y$ (the alternative $q\leq y$,
$|\delta q|>2y$ is impossible). Then (\ref{eq:piececake}) is at most
\begin{equation}\label{eq:octet}\begin{aligned}
&\frac{2 \sqrt{c_0 c_1}}{\pi} \left(UV \log \frac{UV}{\sqrt{e}} + 
Q \left(\sqrt{3} \log \frac{c_2 x}{Q} + \frac{\log UV}{2} \log \frac{UV}{Q/2}
\right)\right)\\
&+ \frac{3 c_1}{2} \frac{x}{y} \log UV \log \frac{UV}{c_2 x/y} + 
\frac{16 \log 2}{\pi} Q \log \frac{c_0 e^3 Q^2}{4\pi \cdot 8 \log 2 \cdot x}
\log \frac{Q}{2}\\&+ \frac{3 c_1}{2 \sqrt{2 c_2}} \sqrt{x} \log \frac{c_2 x}{2}
+ \frac{25 c_0}{4\pi^2} (3 c_2)^{1/2} \sqrt{x} \log x,\end{aligned}\end{equation}
where $c_1 = 1.0000028 > 1+ (8 \log 2)/V$. 
Here $\log(c_0 e^3 Q^2/(4\pi \cdot 8\log 2 \cdot x)) \log Q/2$ is at most
$\log x^{1/3} \log x^{2/3}$.
Using this and (\ref{eq:voil}),
we get that (\ref{eq:octet}) is at most
\[\begin{aligned}
1177.617 x^{2/3} \log x &+ 0.0006406 x^{2/3} (\log x)^2 + 29.5949 x^{1/2}
\log x\\ &\leq 1177.64 x^{2/3} \log x + 0.0006406 x^{2/3} (\log x)^2.\end{aligned}\]

If $|\delta|> 1/2 c_2$, then we know that $|\delta q| > \max(y/2 c_2, 2y) = 
y/2 c_2$. Thus (\ref{eq:tvorog}) (with $\epsilon=0.01$) is at most
\[\begin{aligned}
&\frac{2 \sqrt{c_0 c_1}}{\pi} UV \log \frac{UV}{\sqrt{e}} \\ &+ 
\frac{2.02 \sqrt{c_0 c_1}}{\pi} \left(\frac{x}{y/2c_2}+1\right)
\left( (\sqrt{3.02}-1) \log \frac{\frac{x}{y/2c_2} + 1}{\sqrt{2}}
+ \frac{1}{2} \log UV \log \frac{e^2 UV}{\frac{x}{y/2c_2}}\right)\\
&+ \left(\frac{3 c_1}{2} \left(\frac{1}{2} + \frac{3.03}{0.16} \log x\right)
+ \frac{20 c_0}{3 \pi^2} (2 c_2)^{3/2}\right) \sqrt{x} \log x .\end{aligned}\]
Again by (\ref{eq:voil}), this simplifies to
\[\leq 1212.591 x^{2/3} \log x + 29.131 x^{1/2} \log x \leq
1213.15 x^{2/3} (\log x)^2.\]

Hence, in total and for any $|\delta|$,
\[|S_{I,2}| \leq 1213.15 x^{2/3} (\log x) + 0.0006406 x^{2/3} (\log x)^2.\]

Now we must estimate $S_{II}$. As we said before, either (a) $q>y/4$, or both
(b1) $|\delta|>8$ and (b2) $|\delta| q>8 y$. Recall that $\theta = e^2$.
In case (a), we use 
(\ref{eq:vinland2}), and obtain that, if $y/4 < q\leq x/2 e^2 U$,
$|S_{II}|$ is at most
\begin{equation}\label{eq:hust}\begin{aligned}
&\frac{x \sqrt{\digamma(q)}}{\sqrt{2 q}}
\sqrt{\left(\log \frac{x}{U \cdot 2 e^2 q} + \log 2q 
\log \frac{\log x/(2 U q)}{\log e^2}\right) \left(\kappa_6 \log \frac{x}{
U\cdot 2 e^2 q} + 2 \kappa_7\right)}\\
&+ \sqrt{2} \kappa_2 \sqrt{\digamma\left(\frac{x}{2 e^2 U}\right)}
\left(1 + 1.15 \sqrt{\frac{\log x/e^2 U}{2}}\right) \frac{x}{\sqrt{U}} + 
(\kappa_2 \sqrt{\log x/U} + \kappa_9) \frac{x}{\sqrt{V}}\\
&+ \frac{\kappa_2}{6} \left((\log (e^2 y/2))^{3/2} - (\log y)^{3/2}\right) \frac{x}{\sqrt{y}} \\ &+ \kappa_2 \left(\sqrt{2 e^2 \cdot \log x/U} + 
\frac{2}{3} ((\log x/U)^{3/2} - (\log V)^{3/2})\right) \frac{x}{\sqrt{2 e^2 U}}
,\end{aligned}\end{equation}
where $\digamma$ is as in (\ref{eq:locos}). 
It is easy to check that $q\to (\log 2q) (\log \log q)/q$ is decreasing for
$q\geq y$ (indeed for $q\geq 9$), and so the first line of
(\ref{eq:hust}) is minimal for $q=y$. Asymptotically, the largest term
in (\ref{eq:hust}) comes from the last line (of order $x^{5/6} (\log
x)^{3/2}$), even if the first line is larger in practice (while being of
order $x^{5/6} (\log x) \log \log x$). The ratio
of (\ref{eq:hust}) (for $q=y=x^{1/3}/6$) to 
 $x^{5/6} (\log x)^{3/2}$ is descending for $x\geq x_0 = 2.16\cdot 10^{20}$;
its value at $x=x_0$ gives
\begin{equation}\label{eq:hosto}
|S_{II}|\leq 0.272652 x^{5/6} (\log x)^{3/2}\end{equation}
in case (a), for $q\leq x/2 e^2 U$.

If $x/2 e^2 U < q \leq Q$, we use (\ref{eq:vinland3}). In this range,
$x/2 \sqrt{2q} +\sqrt{qx}$ adopts its maximum at $q=Q$ (because 
$x/2\sqrt{2q}$ for $q=x/2 e^2 U$ is smaller than $\sqrt{qx}$ for
$q=Q$, by (\ref{eq:voil})). A brief calculation starting from
(\ref{eq:vinland3}) then gives that
\[|S_{II}|\leq 0.10198 x^{5/6} (\log x)^{3/2},\]
where we use (\ref{eq:voil}) yet again to simplify.

Finally, let us treat case (b), that is, $|\delta|>8$ and $|\delta| q>8y$;
we can also assume $q\leq y$, as otherwise we are in case (a),
which has already been treated. Since $|\delta/x|\leq x/Q$, we know that
$|\delta q|\leq x/Q = U$. From (\ref{eq:vinlandsaga}), we obtain
that $|S_{II}|$ is at most
\[\begin{aligned}
&\frac{2 x \sqrt{\digamma(y)}}{\sqrt{8 y}}
\sqrt{\left(\log \frac{x}{U\cdot e^2 \cdot 8y} + \log 3 y 
\log \frac{\log x/3 U y}{\log 8 e^2/3}\right) \left(\kappa_6 
\log \frac{x}{U\cdot e^2 \cdot 2y} + 2 \kappa_7\right)}\\
&+ \frac{2\kappa_2}{3} \left(\frac{x}{\sqrt{16 y}} 
((\log 8 e^2 y)^{3/2} - (\log y)^{3/2}) + 
\frac{x/4}{\sqrt{Q-y}}
((\log e^2 U)^{3/2} - (\log y)^{3/2} )\right) \\
&+ \left(\frac{\kappa_2}{\sqrt{2 (1-y/Q)}} \left(\sqrt{\log V} + 
\sqrt{1/\log V}\right) + \kappa_{9}\right) \frac{x}{\sqrt{V}} 
\\ &+
\kappa_2 \sqrt{2\digamma(y)} \cdot \sqrt{\frac{\log e^2 U}{\log 8 e^2/3}}
\cdot \frac{x}{\sqrt{U}},
\end{aligned}\]
We take the maximum of the ratio of this to $x^{5/6} (\log x)^{3/2}$, and
obtain
\[|S_{II}|\leq 0.24956 x^{5/6} (\log x)^{3/2}.\]
Thus (\ref{eq:hosto}) gives the worst case.

We now take totals, and obtain
\begin{equation}\label{eq:bague}\begin{aligned}
S_{\eta}(x,\alpha) &\leq |S_{I,1}|+ |S_{I,2}|+ |S_{II}| \\ &\leq
(4.1982+1213.15) x^{2/3} \log x +
(0.001063 + 0.0006406) x^{2/3} (\log x)^2\\ &+ 0.272652 x^{5/6} (\log x)^{3/2}\\
&\leq 0.27266 x^{5/6} (\log x)^{3/2} + 1217.35 x^{2/3} \log x,\end{aligned}\end{equation}
where we use (\ref{eq:voil}) yet again.

\subsection{Conclusion}
\begin{proof}[Proof of main theorem]
We have shown that $|S_{\eta}(\alpha,x)|$ is at most (\ref{eq:duaro}) for 
$q\leq x^{1/3}/6$ and at most (\ref{eq:bague}) for $q> x^{1/3}/6$.
It remains to simplify (\ref{eq:duaro}) slightly. Let
\[\rho = \frac{C_{x_1,2 q_0} (\log 2 q_0  +0.002) + 
\frac{\log 8 q_0}{2}}{0.30214 \log 2 q_0 + 0.67506} = 3.61407\dotsc,\]
where $x_1 = 10^{25}$, $q_0 = 2\cdot 10^5$. 
(We will be optimizing matters for $x=x_1$,
$\delta_0 q = 2 q_0$, with very slight losses in nearby ranges.) By the 
geometric mean/arithmetic mean inequality,
\[\sqrt{C_{x_1,\delta_0 q} (\log \delta_0 q +0.002) + \frac{\log 4 \delta_0 q}{2}}
\sqrt{0.30214 \log \delta_0 q + 0.67506}.\]
is at most
\[\begin{aligned}
\frac{1}{2} &\left(\frac{1}{\sqrt{\rho}} 
\left(C_{x_1,\delta_0 q} (\log \delta_0 q +0.002) + \frac{\log 4 \delta_0 q}{2}
\right) + \sqrt{\rho} (0.30214 \log \delta_0 q + 0.67506)\right)\\
&\leq 
\frac{C_{x,\delta_0 q}}{2\sqrt{\rho}} (\log \delta_0 q +0.002) +
\left(\frac{1}{4\sqrt{\rho}} +  \frac{\sqrt{\rho} \cdot 0.30214}{2}\right) \log
\delta_0 q  \\ &+ 
\frac{1}{2} \left(\frac{\log 2}{\sqrt{\rho}} + \frac{\sqrt{\rho}}{2} \cdot
0.67506\right)
\\
&\leq 0.27125 \log \left(1 + \frac{\log 4 t}{
2 \log \frac{9 x^{1/3}}{2.004 t}}\right) 
 (\log \delta_0 q + 0.002) + 
0.4141 \log \delta_0 q + 0.49911.\end{aligned}\]
Now, for $x\geq x_0 = 2.16\cdot 10^{20}$,
\[\frac{C_{x,t}}{\log t} \leq \frac{C_{x_0,t}}{\log t} \leq 0.08659\]
for $t\leq 10^6$, and
\[\frac{C_{x,t}}{\log t} \leq \frac{C_{6 t^3,t}}{\log t} \leq
\frac{1}{\log t} \log \left(1 + \frac{\log 4 t}{2 \log \frac{27}{1.002}}
\right)\leq 0.08659\]
if $10^6 < t \leq x^{1/3}/6$. Hence
\[0.27125 \cdot C_{x,\delta_0 q} \cdot 0.002 \leq 0.000047 \log \delta_0 q.\]

We conclude that, for $q\leq x^{1/3}/6$,
\[\begin{aligned}&|S_{\eta}(\alpha,x)| \leq 
 \frac{R_{x,\delta_0 q} \log \delta_0 q + 0.49911}{\sqrt{\phi(q) \delta_0}} \cdot x
+ \frac{2.491 x}{\sqrt{q \delta_0}}
\\ &+ 
\frac{2 x}{\delta_0} 
\min\left(\frac{\log \delta_0^{\frac{7}{4}} q^{\frac{13}{4}} + \frac{80}{9}}{\phi(q)},
\frac{\frac{5}{6} \log x + \frac{50}{9}}{q}\right)
+ \frac{2 x}{\delta_0} 
\frac{\log q^{\frac{80}{9}} \delta_0^{\frac{16}{9}} + \frac{111}{5}}{q}
+ 3.2 x^{5/6},\end{aligned}\]
where \[R_{x,t} = 0.27125 \log \left(1 + \frac{\log 4t}{2 \log \frac{9 x^{1/3}}{2.004 t}}\right) + 0.41415.\]
\end{proof}
\appendix

\section{Norms of Fourier transforms}\label{sec:norms}

Our aim here is to give upper bounds on $|\widehat{\eta_2''}|_\infty$,
where $\eta_2$ is as in (\ref{eq:eqeta}). We will do considerably
better than the trivial bound $|\widehat{\eta''}|_\infty \leq |\eta''|_1$.
\begin{lem}\label{lem:wollust}
For every $t\in \mathbb{R}$,
\begin{equation}
|4 e(-t/4) - 4 e(-t/2) + e(-t)| \leq 7.87052. 
\end{equation}
\end{lem}
We will describe an extremely simple, but rigorous,
 procedure to find the maximum. Since 
$|g(t)|^2$ is $C^2$ (in fact smooth), there are several more
efficient and equally rigourous algorithms -- for starters, the
bisection method with error bounded in terms of $|(|g|^2)''|_\infty$.
\begin{proof}
Let \begin{equation}\label{eq:ellib}
g(t) = 4 e(-t/4) - 4 e(-t/2) + e(-t).\end{equation} For $a\leq t\leq b$,
\begin{equation}\label{eq:maran}
g(t) = g(a) + \frac{t-a}{b-a} (g(b)-g(a)) + \frac{1}{8} (b-a)^2 \cdot
O^*(\max_{v\in \lbrack a,b\rbrack} |g''(v)|).\end{equation}
(This formula, in all likelihood well-known, is easy to derive. First,
we can assume without loss of generality that $a=0$, $b=1$ and $g(a)=g(b)=0$.
Dividing by $g$ by $g(t)$, we see that we can also assume that $g(t)$ is real
(and in fact $1$). We can also assume that $g$ is real-valued, in that 
it will be enough to prove (\ref{eq:maran}) for the
real-valued function $\Re g$, as this will give us the bound
$g(t) = \Re g(t) \leq (1/8) \max_v |(\Re g)''(v)| \leq \max_v |g''(v)|$ that
we wish for.
Lastly, we can assume (by symmetry) that $0\leq t\leq 1/2$, and that 
$g$ has a local maximum or minimum at $t$.
Writing $M = \max_{u\in \lbrack 0,1\rbrack} |g''(u)|$, we then have:
\[\begin{aligned}
g(t) &= \int_0^t g'(v) dv = \int_0^t \int_t^v g''(u) du dv =
O^*\left(\int_0^t \left|\int_t^v M du\right| dv\right)\\ &= O^*\left(
\int_0^t (v-t) M dv\right) = O^*\left(\frac{1}{2} t^2 M\right) = O^*\left(\frac{1}{8} M\right),\end{aligned}\]
as desired.)

We obtain immediately from (\ref{eq:maran}) that
\begin{equation}\label{eq:elek1}
\max_{t\in \lbrack a,b\rbrack} |g(t)| \leq \max(|g(a)|,|g(b)|) + 
\frac{1}{8} (b-a)^2 \cdot \max_{v\in \lbrack a,b\rbrack} |g''(v)| .\end{equation}

For any $v\in \mathbb{R}$,
\begin{equation}\label{eq:elek2}
|g''(v)| \leq \left(\frac{\pi}{2}\right)^2 \cdot 4 + \pi^2 \cdot 4 + (2\pi)^2
= 9 \pi^2 .
\end{equation}
Clearly $g(t)$ depends only on $t \mo 4 \pi$. Hence, by (\ref{eq:elek1})
and (\ref{eq:elek2}), to estimate $\max_{t\in \mathbb{R}} |g(t)|$
 with an error of at most $\epsilon$, 
it is enough to subdivide $\lbrack 0, 4\pi\rbrack$ into intervals of length
$\leq \sqrt{8\epsilon/9\pi^2}$ each. We set $\epsilon = 10^{-6}$ and compute. 
\end{proof}

\begin{lem}\label{lem:camelo}
Let $\eta_2:\mathbb{R}^+\to \mathbb{R}$ be as in (\ref{eq:eqeta}).
Then
\begin{equation}\label{eq:elundot}
|\widehat{\eta_2''}|_\infty \leq 31.521.
\end{equation}
\end{lem}
This should be compared with $|\eta_2''|_1 = 48$.
\begin{proof}
We can write
\begin{equation}\label{eq:petruchka}
\eta_2''(x) = 4 (4 \delta_{1/4}(x) - 4 \delta_{1/2}(x) + \delta_1(x)) + f(x),
\end{equation}
where $\delta_{x_0}$ is the point measure at $x_0$ of mass $1$ (Dirac
delta function) and
\[f(x) = \begin{cases} 0 &\text{if $x< 1/4$ or $x\geq 1$,}\\
-4 x^{-2} &\text{if $1/4 \leq x < 1/2$,}\\
4 x^{-2} &\text{if $1/2 \leq x < 1$.} \end{cases}\]
Thus $\widehat{\eta_2''}(t) = 4 g(t) + \widehat{f}(t)$, where $g$ is as in
(\ref{eq:ellib}). It is easy to see that $|f'|_1 = 2 \max_x f(x) - 
2 \min_x f(x) = 160$. Therefore,
\begin{equation}\label{eq:fedsan}
\left|\widehat{f}(t)\right| = \left|\widehat{f'}(t)/(2\pi i t)\right|
\leq \frac{|f'|_1}{2\pi |t|} = \frac{80}{\pi |t|} .
\end{equation}
Since $31.521-4\cdot 7.87052 = 0.03892$, we conclude that
(\ref{eq:elundot}) follows from Lemma \ref{lem:wollust}
and (\ref{eq:fedsan}) for $|t|\geq 655 > 80/(\pi\cdot 0.03892)$.

It remains to check the range $t\in (-655,655)$; since
$4 g(-t) + \widehat{f}(-t)$ is the complex conjugate of $4 g(t) + \widehat{f}(t)$,
 it suffices to consider $t$ non-negative. We use 
(\ref{eq:elek1}) (with $4 g+\widehat{f}$ instead of $g$) and obtain that, to 
estimate $\max_{t\in \mathbb{R}} |4 g+\widehat{f}(t)|$ with an error of at most
$\epsilon$, it is enough to subdivide $\lbrack 0,655)$
into intervals of length $\leq \sqrt{2 \epsilon/|(4 g+\widehat{f})''|_\infty}$ each
 and check $|4 g+\widehat{f}(t)|$ at the endpoints. 
Now, for every $t\in \mathbb{R}$,
\[\left|\left(\widehat{f}\right)''(t)\right| = \left| (-2 \pi i)^2
\widehat{x^2 f}(t)\right| = (2\pi)^2 \cdot O^*\left(|x^2 f|_1\right) =
12 \pi^2.\]
By this and (\ref{eq:elek2}), $|(4 g + \widehat{f})''|_\infty \leq 48 \pi^2$.
Thus, intervals of length $\delta_1$ give an error term of size at most
$24 \pi^2 \delta_1^2$. We choose $\delta_1 = 0.001$ and obtain an error
term less than $0.000237$ for this stage.

To evaluate $\widehat{f}(t)$ (and hence $4 g(t)+\widehat{f}(t)$) 
at a point, we use Simpson's rule on subdivisions of the intervals
$\lbrack 1/4,1/2\rbrack$, $\lbrack 1/2,1\rbrack$ 
into $200\cdot \max(1,\lfloor \sqrt{|t|}\rfloor)$ sub-intervals 
each.\footnote{The author's code uses
D. Platt's implementation \cite{Platt} of double-precision interval arithmetic (based
on Lambov's \cite{Lamb} ideas).}
The largest value of $\widehat{f}(t)$ we find is $31.52065\dotsc$,
with an error term of at most $4.5\cdot 10^{-5}$.
\end{proof}

\begin{lem}
Let $\eta_2:\mathbb{R}^+\to \mathbb{R}$ be as in (\ref{eq:eqeta}).
Let $\eta_y(t) = \log(y t) \eta_2(t)$, where $y\geq 4$. Then
\begin{equation}
|\eta_y'|_1 < (\log y) |\eta'_2|_1.
\end{equation}
\end{lem}
This was sketched in \cite[(2.4)]{Helf}.
\begin{proof}
Recall that $\supp(\eta_2) = (1/4,1)$. For $t\in (1/4,1/2)$,
\[\eta_y'(t) = (4 \log(y t) \log 4 t)' = \frac{4 \log 4 t}{t} + \frac{4 \log y t}{t}
\geq \frac{8 \log 4 t}{t} > 0,\] whereas, for $t\in (1/2,1)$,
\[\eta_y'(t) = (- 4 \log(y t) \log t)' = -\frac{4 \log y t}{t}-\frac{4 \log t}{t}
= - \frac{4 \log y t^2}{t} < 0,
\] 
where we are using the fact that $y\geq 4$. Hence $\eta_y(t)$ is increasing on 
$(1/4,1/2)$ and decreasing on $(1/2,1)$; it is also continuous at $t=1/2$.
Hence $|\eta_y'|_1 = 2 |\eta_y(1/2)|$. We are done by
\[2 |\eta_y(1/2)| = 2 \log \frac{y}{2} \cdot \eta_2(1/2) = \log \frac{y}{2} \cdot
8 \log 2 < \log y \cdot 8 \log 2 = (\log y) |\eta'_2|_1.\]
\end{proof}

\begin{lem}\label{lem:lujur}
Let $y\geq 4$. Let $g(t) = 
4 e(-t/4) - 4 e(-t/2) + e(-t)$ and $k(t) = 
2 e(-t/4)-e(-t/2)$. Then,
for every $t\in \mathbb{R}$,
\begin{equation}\label{eq:gotora}
|g(t) \cdot \log y  - k(t) \cdot 4 \log 2| \leq 7.87052 \log y. 
\end{equation}
\end{lem}
\begin{proof}
By Lemma \ref{lem:wollust}, $|g(t)|\leq 7.87052$. Since $y\geq 4$, 
$k(t)\cdot (4 \log 2)/\log y \leq 6$. 
For any complex numbers $z_1$, $z_2$ with $|z_1|, |z_2|\leq \ell$,
 we can have $|z_1 - z_2|> \ell$ only if $|\arg(z_1/z_2)|> \pi/3$.
It is easy to check that, for all $t\in \lbrack -2,2\rbrack$,
\[\left|\arg\left(\frac{g(t) \cdot \log y}{4 \log 2 \cdot k(t)}\right)\right|
= \left|\arg\left(\frac{g(t)}{k(t)}\right)\right| < 0.7 < \frac{\pi}{3} .\]
(It is possible to bound maxima rigorously as in (\ref{eq:elek1}).)
Hence (\ref{eq:gotora}) holds.
\end{proof}

\begin{lem}\label{lem:octet}
Let $\eta_2:\mathbb{R}^+\to \mathbb{R}$ be as in (\ref{eq:eqeta}).
Let $\eta_{(y)}(t) = (\log y t) \eta_2(t)$, where $y\geq 4$. Then
\begin{equation}\label{eq:yotoman}
|\widehat{\eta_{(y)}''}|_\infty < 31.521 \cdot \log y .
\end{equation}
\end{lem}
\begin{proof}
Clearly
\[\begin{aligned}
\eta_{(y)}''(x) &= \eta''_2(x) (\log y) + \left((\log x) \eta''_2(x) + \frac{2}{x}
\eta'_2(x) - \frac{1}{x^2} \eta_2(x)\right)\\
&= \eta''_2(x) (\log y) + 4 (\log x) (4 \delta_{1/4}(x) - 4 \delta_{1/2}(x)
+ \delta_1(x)) +  h(x),
\end{aligned}\]
where \[h(x) = \begin{cases} 0 &\text{if $x<1/4$ or $x>1$,}\\
\frac{4}{x^2} (2- 2 \log 2 x)  &\text{if $1/4\leq x < 1/2$,}\\
\frac{4}{x^2} (-2 + 2 \log x) &\text{if $1/2 \leq x < 1$.}
\end{cases}\]
(Here we are using the expression 
(\ref{eq:petruchka}) for $\eta''_2(x)$.)
Hence \begin{equation}\label{eq:kokoso}
\widehat{\eta_{(y)}''}(t) = (4 g(t) + \widehat{f}(t)) (\log y) + 
(-16 \log 2 \cdot k(t) + \widehat{h}(t)),\end{equation}
where $k(t) = 2 e(-t/4) - e(-t/2)$. Just as in the proof of Lemma
\ref{lem:camelo},
\begin{equation}\label{eq:pasaremos}
|\widehat{f}(t)| \leq \frac{|f'|_1}{2\pi |t|} \leq 
\frac{80}{\pi |t|},\;\;\;\;\;
|\widehat{h}(t)| \leq \frac{160 (1 + \log 2)}{\pi |t|}.
\end{equation}
Again as before, this implies that (\ref{eq:yotoman}) holds for
\[|t| \geq \frac{1}{\pi\cdot 0.03892} \left(80 
+ \frac{160 (1 + \log 2)}{(\log 4)}\right) = 2252.51 .\]
Note also that it is enough to check (\ref{eq:yotoman}) for $t\geq 0$,
by symmetry. Our remaining task is to
 prove (\ref{eq:yotoman}) for $0\leq t\leq 2252.21$.

Let $I = \lbrack 0.3, 2252.21\rbrack \setminus \lbrack 3.25, 3.65\rbrack$.
For $t\in I$, we will have 
\begin{equation}\label{eq:hotor}
\arg\left(\frac{4 g(t) + \widehat{f}(t)}{-16 \log 2 \cdot k(t) + 
\widehat{h}(t)}\right)  \subset \left( -\frac{\pi}{3},
\frac{\pi}{3}\right) .
\end{equation}
(This is actually true for $0\leq t\leq 0.3$ as well, but we will use a
different strategy in that range in order to better control error terms.)
Consequently, by Lemma \ref{lem:camelo} and
$\log y \geq \log 4$, 
\[\begin{aligned}
|\widehat{\eta_{(y)}''}(t)| &< \max(|4 g(t)+\widehat{f}(t)| \cdot (\log y),
|16 \log 2 \cdot k(t) - \widehat{h}(t)|)\\ &<  \max(31.521 (\log y),
|48 \log 2 + 25|) = 31.521 \log y,
\end{aligned}\]
where we bound $\widehat{h}(t)$ by
(\ref{eq:pasaremos}) and by a numerical 
computation of the maximum of $|\widehat{h}(t)|$
for $0\leq t \leq 4$ as in the proof of Lemma
\ref{lem:camelo}. 

It remains to check (\ref{eq:hotor}). Here, as in the proof of Lemma 
\ref{lem:lujur}, the allowable error is relatively large (the expression
on the left of (\ref{eq:hotor}) is actually contained in $(-1,1)$
for $t\in I$).
We decide to evaluate the argument in (\ref{eq:hotor}) at all
$t\in 0.005\mathbb{Z} \cap I$,
computing $\widehat{f}(t)$ and $\widehat{h}(t)$ by numerical integration
(Simpson's rule) with a subdivision of $\lbrack -1/4,1\rbrack$ into $5000$
intervals. Proceeding as in the proof of Lemma \ref{lem:wollust}, we see
that the sampling induces an error of at most
\begin{equation}\label{eq:aros}
\frac{1}{2} 0.005^2 \max_{v\in I} ((4 |g''(v)| + |(\widehat{f})''(t)|)
\leq \frac{0.0001}{8} 48 \pi^2 < 0.00593\end{equation}
in the evaluation of $4 g(t)+\widehat{f}(t)$, and an error of at most
\begin{equation}\label{eq:oros}\begin{aligned}
\frac{1}{2} &0.005^2 \max_{v\in I} ((16 \log 2\cdot |k''(v)| + 
|(\widehat{h})''(t)|) \\ &\leq
\frac{0.0001}{8} (16 \log 2 \cdot 6 \pi^2 + 
24 \pi^2 \cdot (2-\log 2)) < 0.0121\end{aligned}\end{equation}
in the evaluation of $16 \log 2 \cdot |k''(v)| + |(\widehat{h})''(t)|$.

Running the numerical evaluation just described for $t\in I$, the estimates
for the left side of (\ref{eq:hotor}) at the sample points are at most
 $0.99134$ in absolute value; the absolute values of the estimates for
$4 g(t) + \widehat{f}(t)$ are all at least $2.7783$, 
and the absolute values of the estimates for
$|-16 \log 2 \cdot \log k(t) + \widehat{h}(t)|$ are all at least $2.1166$.
Numerical integration by Simpson's rule gives errors bounded by
$0.17575$ percent.
Hence
the absolute value of the left side of (\ref{eq:hotor}) is at most
\[\begin{aligned}
0.99134 &+ \arcsin \left(\frac{0.00593}{2.7783}
+0.0017575\right) + \arcsin\left(\frac{0.0121}{2.1166} 
 + 0.0017575\right)\\
&\leq 1.00271 < \frac{\pi}{3}\end{aligned}\]
for $t\in I$.

Lastly, for $t\in \lbrack 0,0.3\rbrack \cup \lbrack 3.25,3.65\rbrack$, 
a numerical computation (samples at 
$0.001\mathbb{Z}$; interpolation as in Lemma
\ref{lem:camelo};
integrals computed by Simpson's rule with a subdivision into $1000$ intervals)
gives
\[\max_{t\in \lbrack 0,0.3\rbrack \cup \lbrack 3.25,3.65\rbrack} \left(|(4g(t) + \widehat{f}(t))| + \frac{|-16 \log 2 \cdot k(t)
+ \widehat{h}(t)|}{\log 4}\right) < 29.08,\]
and so $\max_{t\in \lbrack 0,0.3\rbrack \cup \lbrack 3.25,3.65\rbrack} |\widehat{\eta_{(y)}''}|_\infty < 29.1 \log y < 31.521
\log y$.
\end{proof}

An easy integral gives us that the function $\log \cdot \eta_2$ satisfies 
\begin{equation}\label{eq:koasl}
|\log \cdot \eta_2|_1 = 2 - \log 4
\end{equation}
The following function will appear only in a lower-order term; thus,
an $\ell_1$ estimate will do.

\begin{lem}\label{lem:marengo}
Let $\eta_2:\mathbb{R}^+\to \mathbb{R}$ be as in (\ref{eq:eqeta}).
Then
\begin{equation}\label{eq:carengo}
|(\log \cdot \eta_2)''|_1 = 96 \log 2.
\end{equation}
\end{lem}
\begin{proof}
The function $\log \cdot \eta(t)$ is $0$ for $t\notin \lbrack 1/4,1\rbrack$,
is increasing and negative for $t\in (1/4,1/2)$ and is decreasing and positive
for $t\in (1/2,1)$. Hence
\[\begin{aligned}
|(\log \cdot \eta_2)''|_\infty &=
2 \left(  (\log \cdot \eta_2)'\left(\frac{1}{2}\right) - 
 (\log \cdot \eta_2)'\left(\frac{1}{4}\right)\right) \\ &= 
2 (16 \log 2 - (-32 \log 2)) = 96 \log 2.\end{aligned}\]
\end{proof}
\bibliographystyle{alpha}
\bibliography{arcs}
\end{document}